\documentclass[11pt]{article}

\usepackage[margin=1in]{geometry}
\usepackage{amsmath,amssymb,amsthm,mathtools,mathrsfs}
\usepackage{enumitem}
\usepackage[colorlinks=true,citecolor=blue,linkcolor=blue,urlcolor=blue]{hyperref}
\usepackage{booktabs,tabularx}
\newtheorem{theorem}{Theorem}[section]
\newtheorem{proposition}[theorem]{Proposition}
\newtheorem{lemma}[theorem]{Lemma}
\newtheorem{corollary}[theorem]{Corollary}
\newtheorem{remark}[theorem]{Remark}
\newtheorem{assumption}[theorem]{Assumption}

\numberwithin{equation}{section}

\newcommand{\R}{\mathbb R}
\newcommand{\E}{\mathbb E}
\newcommand{\cH}{\mathcal H}
\newcommand{\cP}{\mathcal P}
\newcommand{\KSD}{\mathsf{KSD}}
\newcommand{\KL}{\mathsf{KL}}
\newcommand{\Law}{\mathcal L}

\newcommand{\norm}[1]{\left\|#1\right\|}
\newcommand{\ip}[2]{\left\langle #1,#2\right\rangle}
\newcommand{\Lip}{\operatorname{Lip}}
\newcommand{\cF}{\mathcal{F}}

\newcommand{\muNn}{\mu^N_n}
\newcommand{\muNt}{\mu^N_t}
\newcommand{\mut}{\mu_t}
\newcommand{\muNs}{\mu^N_s}
\newcommand{\mus}{\mu_s}

\newcommand{\muz}{\mu_0}

\newcommand{\diver}{\operatorname{div}}
\newcommand{\Tr}{\operatorname{Tr}}
\newcommand{\Cov}{\operatorname{Cov}}
\newcommand{\Id}{I}

\newif\ifdraft
\draftfalse 
\drafttrue
\usepackage{todonotes}
\usepackage{color}
\definecolor{blush}{rgb}{0.87, 0.36, 0.51}
\ifdraft
\newcommand{\ak}[1]{\todo[color=blush,size=\footnotesize]{ \textbf{AK:}  #1}}

\title{\bf Uniform-in-time Propagation-of-Chaos for\\  Stein Variational Gradient Descent}

\author{Krishnakumar Balasubramanian\thanks{Department of Statistics, University of California, Davis.
Email: \texttt{kbala@ucdavis.edu}.}
\qquad
Sayan Banerjee\thanks{Department of Statistics and Operations Research,
University of North Carolina at Chapel Hill,
Email: \texttt{sayan@unc.edu}.}
\qquad
Anna Korba\thanks{CREST, ENSAE
Institut Polytechnique de Paris
Email: \texttt{anna.korba@ensae.fr}.}
}
\date{}

\begin{document}
\maketitle

\begin{abstract}
    We study uniform-in-time propagation-of-chaos for continuous-time Stein Variational Gradient Descent (SVGD). Classical finite-time propagation-of-chaos estimates for mean-field systems typically deteriorate rapidly with time and therefore do not directly explain the long-time relation between the finite-particle system and its mean-field limit.
  We obtain two complementary classes of uniform-in-time propagation-of-chaos results.
  
  For broad distributional metrics, we  introduce a cutoff strategy which combines finite-time propagation-of-chaos estimates up to an \(N\)-dependent horizon with independent quantitative long-time 
 convergence 
    estimates for the finite-particle and mean-field SVGD flows. This yields uniform-in-averaging-time propagation-of-chaos bounds in Langevin kernel Stein discrepancy, Wasserstein-1 distance, and Wasserstein-2 distance, with logarithmic or iterated-logarithmic rates depending on the metric, target and kernel class.

We also develop a finite-dimensional theory for matrix-valued finite-rank kernels. 
For Gaussian targets with bilinear kernels, the SVGD dynamics close exactly on first and second moments, yielding genuine uniform-in-physical-time parametric propagation-of-chaos rates in finite-dimensional Stein-feature metrics. We then prove a conjugacy principle showing that these feature-level estimates transfer to conjugate target-kernel pairs under orientation-preserving diffeomorphisms, thereby extending the theory to broad classes of nonlinear, including multimodal, targets.


Together, these results highlight the contrast between generic distributional metrics, for which our general approach yields logarithmic rates, and closed finite-dimensional Stein observables, for which parametric  \(N^{-1/2}\) propagation-of-chaos rates persist uniformly in time.\\



\end{abstract}

\section{Introduction}
\subsection{Background: Stein Variational Gradient Descent}
Stein Variational Gradient Descent (SVGD) is a deterministic interacting-particle method for approximating a target law $\pi(dx)\propto e^{-V(x)}dx$ by evolving particles through a kernelized Stein velocity \cite{LiuWang2016}.  The velocity combines transport toward high-probability regions of the target with a repulsive interaction that prevents particle collapse, and its mean-field limit can be interpreted as a gradient-flow dynamics for the relative entropy functional to the target in a Stein-induced geometry \cite{duncan2023geometry,Liu2017}.  This structure has made SVGD a central object in the analysis of sampling, variational inference, mean-field limits, and kernel Stein discrepancies \cite{LuLuNolen2019,KorbaSalimArbelLuiseGretton2020,ShiMackey2023,BBG2025}.  SVGD, and in particular its amortized version~\cite{feng2017learning}, has also attracted renewed interest due to its connection to drifting models \cite{Balasubramanian2026DriftingRates}, where related kernel-based attraction--repulsion or score-difference based velocity fields are used to move generated distributions during training for one-step generation \cite{DengLiLiDuHe2026Drifting,FranzHoffmannMartius2026,CaoWeiLiu2026,GrettonEtAl2026,turan2026generative}.

Given a positive-definite kernel \(k\), SVGD at the measure-transport level can be viewed as weak solutions $\{\mu_t\}_{t \ge 0}$ to the mean-field transport equation
\begin{equation}\label{eq:mf-svgd}
    \partial_t\mu_t+\nabla\cdot(v_{\mu_t}\mu_t)=0,
\end{equation}
where the \emph{kernelized Stein velocity} associated with a probability measure \(\rho\) is given by
\begin{equation}\label{eq:Steinveldef}
v_\rho(x)
:=
\int_{\mathbb R^d}
\left[-k(x,y)\nabla V(y)+\nabla_2 k(x,y)\right] \,\rho(dy).
\end{equation}
$v_\rho$ is obtained by projecting the velocity for the Wasserstein--2 gradient flow of the relative entropy (or Kullback-Leibler divergence) of $\rho$ with respect to $\pi$
\[
\KL(\rho|\pi)=\int \log\Big(\frac{d\rho}{d\pi}\Big)\,d\rho
\]
onto the vector-valued Reproducing Kernel Hilbert Space (RKHS) generated by \(k\). The kernel $k$ imparts `smoothness' to the velocity $v_\rho$, which is not available at the true $\KL$ gradient flow level.
This smoothness manifests itself in the key integration-by-parts identity
\(
\int k(x,y)\nabla\log\rho(y)\,\rho(dy)
=
-\int \nabla_2 k(x,y)\,\rho(dy)
\)
which lets one bypass the score function $\nabla\log\rho(y)$, leading to a particle discretization of the measure-valued flow. However, at the same time, this kernelization preserves the entropy dissipation property of the $\KL$ gradient flow. Formally, along sufficiently regular solutions,
\[
\frac{d}{dt}\KL(\mu_t|\pi)
=
-\|v_{\mu_t}\|_{\mathcal H^d}^2
=
-\KSD^2(\mu_t \|\pi),
\]
where KSD denotes the kernel Stein discrepancy~\cite{liu2016kernelized,chwialkowski2016kernel}
\[
\KSD(\rho\|\pi):= \left\|\int \{-k(\cdot,x)\nabla V(x)+\nabla_2 k(\cdot,x)\}\,\rho(dx)\right\|_{\cH^d}.
\]
The latter can be identified to a kernelized Fisher divergence \cite{KorbaSalimArbelLuiseGretton2020} or Maximum Mean Discrepancy with a so-called Stein kernel. 
For particles
$x_1^N(n),\ldots,x_N^N(n)\in \R^d$ and step size $\eta>0$, the discrete-time
SVGD update for particles $1\le i\le N $ is
\begin{equation}\label{eq:svgd-discrete}
x_i^N(n+1)
=
x_i^N(n)
-\frac{\eta}{N}\sum_{j=1}^N
\left[
k\bigl(x_i^N(n),x_j^N(n)\bigr)\nabla V\bigl(x_j^N(n)\bigr)
-\nabla_2 k\bigl(x_i^N(n),x_j^N(n)\bigr)
\right],
\end{equation}
where $k$ is a positive-definite kernel and $\nabla_2$ denotes the gradient with respect to the second argument of
$k$. The corresponding continuous-time finite-particle system, for $1\le i\le N $, is
\begin{equation}\label{eq:svgd-continuous}
\dot x_i^N(t)
=
-\frac1N\sum_{j=1}^N
k\bigl(x_i^N(t),x_j^N(t)\bigr)\nabla V\bigl(x_j^N(t)\bigr)
+\frac1N\sum_{j=1}^N
\nabla_2 k\bigl(x_i^N(t),x_j^N(t)\bigr).
\end{equation}
We write
\[
\muNn:=\frac1N\sum_{i=1}^N\delta_{x_i^N(n)},
\qquad
\muNt:=\frac1N\sum_{i=1}^N\delta_{x_i^N(t)}
\]
for the corresponding empirical laws.

Since the Stein velocity~\eqref{eq:Steinveldef} has a score-free representation, in the sense discussed above, both $\mu^N(\cdot)$ and its mean-field limit (which has been shown to exist under assumptions, see \cite{LuLuNolen2019}), can be viewed as weak solutions to~\eqref{eq:mf-svgd}.
More precisely, for every test function \(\varphi\in C_c^1(\mathbb R^d)\),
\[
\frac{d}{dt}\int_{\mathbb R^d}\varphi(x),\mu_t(dx)
=
\int_{\mathbb R^d}\nabla\varphi(x)\cdot v_{\mu_t}(x),\mu_t(dx).
\]

At the mean-field level, several complementary convergence in time mechanisms have been developed.  The geometric formulation of Duncan, N\"usken, and Szpruch~\cite{duncan2023geometry} identifies the continuous SVGD equation as the gradient flow of the relative entropy in a kernel-dependent Stein geometry, thereby relating convergence to coercivity properties of the induced metric and to the choice of kernel.  Korba et al.~\cite{KorbaSalimArbelLuiseGretton2020} provide a population analysis in discrete time for smooth target potentials and bounded kernels, deriving an \(O(T^{-1})\) rate for the time-averaged squared KSD from entropy dissipation. The smoothness assumption needed for this convergence analysis was further weakened in \cite{sun2023convergence}. Salim, Sun, and Richt\'arik~\cite{salim2022convergence} extend the discrete-time, population analysis of the scheme for possibly non-log-concave targets satisfying Talagrand's \(T_1\) inequality, together with dimension-dependent complexity bounds in KSD. 
He et al.~\cite{he2025regularized} introduce a regularized Stein variational flow interpolating between SVGD and the Wasserstein gradient flow of the relative entropy, and establish well-posedness, stability, and convergence to equilibrium.  More recently, Carrillo, Skrzeczkowski, and Warnett~\cite{carrillo2024stein} prove a Stein log-Sobolev inequality, and hence exponential convergence of suitable weak mean-field solutions, for kernels whose Fourier transforms have quadratic decay and are reweighted with exponentials of the target potentials; Carrillo and Skrzeczkowski~\cite{carrillo2025convergence} complement this result at the particle level by showing that the classical mean-field estimate justifies the joint limit only on the scale \(t\asymp\log\log N\), while a refined near-equilibrium stability argument extends control to algebraic time scales of order \(\sqrt N\).  In two recent directions, Carrillo, Skrzeczkowski, and Warnett~\cite{carrillo2026stein} identify the singular limit of highly concentrated-kernel SVGD as a local Wasserstein-type gradient flow with quadratic mobility, whereas Chizat et al.~\cite{chizat2026quantitative} establish quantitative local last-iterate convergence in strong \(L^2\)-type norms for Riesz kernels on the torus, obtaining explicit polynomial rates of convergence for the mean-field SVGD flow and, in the Coulomb case, global exponential convergence.

The first finite-particle convergence guarantee for SVGD was obtained by Shi and Mackey~\cite{ShiMackey2023}, who showed that, for sub-Gaussian targets with Lipschitz score and suitably regular kernels, an appropriately tuned discrete-time \(N\)-particle scheme achieves a KSD rate (to the target) of order \((\log\log N)^{-1/2}\).  Banerjee, Balasubramanian, and Ghosal~\cite{BBG2025} improved these rates by introducing a joint-particle entropy-dissipation argument, yielding time-averaged KSD bounds with a near-parametric \(N^{-1/2}\) finite-particle rate in both continuous and discrete time, as well as Wasserstein--2 rates for bilinear--plus--Mat\'ern kernels and long-time marginal propagation-of-chaos results.  More recently, He et al.~\cite{he2026finite} derived non-asymptotic finite-particle rates for regularized SVGD, proving convergence of time-averaged empirical measures, under suitable assumptions, in the (true, non-kernelized) Fisher divergence and, under the Transport-Information inequality (\(W_1I\)), in Wasserstein--1 distance. Their results cover continuous- and discrete-time dynamics and quantify the trade-off among the particle number, regularization level, step size, and the averaging time.

\subsection{Background: Propagation-of-Chaos}
A central question for any mean-field system, and in particular SVGD, is whether the finite-particle flow $\muNt$ remains close to its mean-field limit $\mu_t$ \emph{uniformly over infinite time horizons}. Such phenomena, which go by the name \emph{uniform-in-time propagation-of-chaos} justify the efficacy of finite-particle approximations in mimicking the long-time behavior of mean-field flows (and vice versa). Indeed, classical estimates used in establishing finite-time propagation-of-chaos (see \cite{sznitman2006topics}) deteriorate rapidly with $t$ and become ineffective for long runs or time-averaged iterates. Despite its importance, such uniform-in-time finite-particle behavior of SVGD remains poorly understood. 

For clarity, we recall the sense in which we use propagation-of-chaos. Let
\(
\muNt:=\frac1N\sum_{i=1}^N\delta_{x_i^N(t)}
\)
be the empirical law of the \(N\)-particle system at time $t$, and let \(\mu_t\) denote
the corresponding mean-field law. Given a metric or discrepancy \(\mathsf d\)
on probability measures on \(\mathbb R^d\), finite-time propagation-of-chaos
means that, for each finite horizon \(T<\infty\),
\[
\sup_{0\le t\le T}
\mathbb E \, \mathsf d(\muNt,\mut)\longrightarrow 0,
\qquad N\to\infty.
\]
Uniform-in-time propagation-of-chaos is the stronger property
\[
\sup_{t\ge0}
\mathbb E \, \mathsf d(\muNt,\mut)\longrightarrow 0,
\qquad N\to\infty.
\]
The above notions can also be established in probability in place of expectation.
Quantitative versions replace these convergences by explicit bounds, with a
rate depending on both \(N,T\) in the finite-time case and only on \(N\) in
the uniform-in-time case.

There is a related notion of $k$-particle propagation-of-chaos for fixed $k \ge 1$. Let
\(P_{N,t}^{(k)}\) denote the \(k\)-particle marginal of the \(N\)-particle
system at time \(t\). Given a metric or discrepancy \(\mathsf d_k\) on probability measures on
\((\mathbb R^d)^k\), finite-time propagation-of-chaos means that, for each
fixed \(k\) and each finite horizon \(T<\infty\),
\[
\sup_{0\le t\le T}
\mathsf d_k\bigl(P_{N,t}^{(k)},\mu_t^{\otimes k}\bigr)
\longrightarrow 0,
\qquad N\to\infty ,
\]
with uniform-in-time propagation-of-chaos denoting the $T=\infty$ case. We will mostly work with the first notion, but will relate our results to the $k$-particle notion wherever possible.

Typically, finite-time propagation-of-chaos estimates are obtained by
coupling the finite-particle dynamics with independent copies of the
corresponding McKean--Vlasov process; see, for instance,
\cite{sznitman2006topics,chaintron2022propagation}. More precisely, one
introduces associated nonlinear Markov processes \((\bar x_i(t))_{1\le i\le N}\) which are i.i.d. with
common law \(\mu_t\), and then controls the discrepancy between \(x^N_i(t)\) and
\(\bar x_i(t)\) by a Gr\"onwall argument. The time-dependence of the resulting
rates is governed by the regularity of the driving vector field. For globally
Lipschitz vector fields, this gives errors growing at most exponentially in
time. In SVGD, however, the relevant vector fields are often only locally
Lipschitz, with Lipschitz constants growing in space. The Gr\"onwall bound
then depends on finite-time moment bounds; typically these moments grow
exponentially in time, in which case, the resulting finite-time propagation-of-chaos bounds
can grow double-exponentially~\cite{LuLuNolen2019}. Korba et al.~\cite{KorbaSalimArbelLuiseGretton2020} exhibited an exponential improvement of these bounds
under the assumption that the gradient of the potential is uniformly bounded (which, for example, excludes a Gaussian target). A notable exception is \cite{liu2023towards}, where uniform-in-time control in
the Wasserstein--2 metric was proved for Gaussian targets, bilinear kernels, and Gaussian initialization. In this setting, the SVGD dynamics reduce to a
closed and stable system for the mean and covariance, reflecting a strong
convexity/contractivity structure in the effective finite-dimensional
dynamics. This mechanism is highly rigid, however, and does not extend
directly to general target--kernel pairs.

Uniform-in-time propagation-of-chaos is therefore substantially more delicate.
Classical approaches obtain such estimates by exploiting dissipativity,
convexity, or contractivity in the finite-particle and mean-field dynamics, sometimes through
functional inequalities such as logarithmic Sobolev or transportation
inequalities; see, for example,
\cite{malrieu2003granular,cattiaux2008granular,durmus2020elementary,guillin2021uniform}.
A different and sharper line of work uses the Bogoliubov–Born–Green–Kirkwood–Yvon (BBGKY) hierarchy~\cite{sznitman2006topics} hierarchy together with
entropy tensorization and functional inequalities to obtain optimal
\(k\)-particle rates, including time-uniform estimates under convexity or
small-interaction assumptions
\cite{lacker2023hierarchies,lackerleflem2023sharp}. These methods are very
powerful, but their time-uniform versions still rely on structural convexity,
contractivity, or uniform functional-inequality inputs which are not available
for SVGD in the generality considered here. Our particle system is deterministic, nonlocal, and highly nonlinear through the kernelized Stein velocity, and the mean-field dynamics are not known to satisfy a global contractive estimate in standard probability metrics in the general case (see, however, \cite{carrillo2026stein} for Stein log-Sobolev inequalities for specific kernel and potential families).

Our first approach is different. We do not try to control the discrepancy between
\(\mu_t^N\) and \(\mu_t\) uniformly for all \(t\ge0\). Instead,
we use finite-time propagation-of-chaos only up to a suitably chosen cutoff
time \(T_N\). For averaging horizons \(T\ge T_N\), we exploit the fact that
the time-averaged finite-particle empirical measure and mean-field SVGD flow each converge
quantitatively to the common target \(\pi\), by separate entropy-dissipation
estimates: the finite-particle estimate comes from
\cite{BBG2025}, while the mean-field estimate follows from the
entropy decay estimate for the SVGD mean-field flow
\cite{Liu2017,LuLuNolen2019,KorbaSalimArbelLuiseGretton2020}. The triangle inequality then
converts these two independent long-time convergence estimates to the target into a
particle--mean-field comparison. Balancing the short-time coupling regime with
the long-time target-convergence regime yields our uniform-in-averaging-horizon
propagation-of-chaos bounds in kernel Stein discrepancy (KSD), Wasserstein--1 ($W_1$) and Wasserstein--2 ($W_2$) distances. 
This cutoff philosophy is close in spirit to the recent weak
propagation-of-chaos approach of \cite{glasgow2026uniform} for shallow neural networks gradient flow dynamics, where long-time
mean-field convergence is used to prevent finite-width fluctuation errors from
accumulating indefinitely. See Remark~\ref{GBrem} for more details. 
It is also explicitly noted in \cite{glasgow2026uniform} that, when the target regression function is exactly representable by an infinite-width shallow network, their setting is an instance of particle gradient descent for kernel maximum mean discrepancy  (MMD), and that their results apply to particle and time discretizations of the corresponding Wasserstein gradient flow under suitable kernel assumptions.  

Our second approach (see Section~\ref{sec:conjgauss}) consists in identifying  a conjugacy principle for finite-rank kernels. We show
that a quantitative uniform-in-time propagation-of-chaos estimate for a given target--kernel pair, formulated in a weak metric controlling selected Stein-feature discrepancies, transfers with the same rate to an entire class of conjugate target--kernel pairs. These conjugate pairs are obtained from the given pair by orientation-preserving diffeomorphisms, under which the Stein features and SVGD velocity transform covariantly. We illustrate this principle for the Gaussian conjugacy class, where constant and bilinear kernels yield closed finite-dimensional dynamics for the latent first and second moments.

Our KSD and feature-level estimates, and the results of \cite{glasgow2026uniform}, also connect naturally with the notion of
\emph{weak propagation-of-chaos} introduced in
\cite{delarue2025uniform}. In that framework, one does not necessarily seek a
strong metric comparison between the empirical measure and its mean-field
limit; instead, one controls the error after testing against suitable
observables or functionals of the empirical measure. Our KSD bounds fit this
viewpoint because the discrepancy is defined as an integral probability metric \cite{muller1997integral} over a unit ball in a Stein-kernel RKHS, while the Gaussian-conjugate estimates are a finite-dimensional version
of the same idea, controlling only the Stein features that close under the
dynamics. 

A related idea appeared in the context of a mean-modulated pure-jump flocking model \cite{banerjee2024flocking}. There, suitable Lyapunov functions were used to prove stability of the centered finite-particle system, and to identify its stationary behavior, without relying explicitly on mean-field comparison. This led to a propagation-of-chaos result at time infinity, in the sense that the stationary empirical measure of the particle system converges to the stationary traveling-wave profile of the limiting McKean--Vlasov equation. However, establishing the strictly stronger uniform-in-time propagation-of-chaos for this model remains open.

We emphasize that our use of long-time estimates differs from the standard direction in quantitative mean-field analysis. Once a uniform-in-time propagation-of-chaos estimate is available, it can be combined with long-time convergence of the mean-field dynamics to deduce the corresponding long-time behavior of the finite-particle system; see, for example, \cite{malrieu2003granular,cattiaux2008granular}. Here, for time-averaged SVGD, we reverse this logic: rather than deriving long-time particle behavior from uniform propagation-of-chaos, we use independent long-time convergence estimates for both the particle and mean-field systems to the common target \(\pi\), together with finite-time propagation-of-chaos up to an \(N\)-dependent cutoff, to obtain a uniform-in-averaging-horizon propagation-of-chaos estimate.

\subsection{Summary of Main Results}

This paper establishes two complementary classes of long-time propagation-of-chaos results for SVGD, summarized in Table~\ref{tab:summary-poc-rates}.  At the level of probability measures, as discussed above, we combine finite-time mean-field stability with long-time entropy dissipation through a cutoff argument.  In Section~\ref{sec:KSD}, under suitable regularity, moment, entropy, and initialization assumptions, this technique yields the general two-sample Langevin KSD bound
\[
    \lim_{M \to \infty} \, \limsup_{N \to \infty}\,\sup_{T>0}\mathbb{P}\left[\KSD_\pi\bigl(\overline\mu_T^N,\overline\mu_T\bigr) \ge M (\log \log N)^{-1/2}\right] = 0,
\]
where, 
\[
\KSD_\pi(\mu,\nu)
:= \sup_{\|\varphi\|_{\cH^d}\le 1}
\left\{\int T_\pi\varphi\,d\mu-\int T_\pi\varphi\,d\nu\right\},
\]
and
$\cH^d$ is the vector-valued RKHS for the kernel $k$ and $T_\pi$ is the Langevin--Stein operator
\[
T_\pi \varphi(x) := -\nabla V(x)\cdot \varphi(x)+\nabla\cdot \varphi(x),
\qquad \varphi\in \cH^d.
\]
The probability bounds above can be upgraded to expectation bounds when the (random) second moment for the initial empirical distribution is uniformly sub-Gaussian in $N$.

In Section~\ref{sec:w1}, we extend our approach to the stronger $W_1$ metric. This strengthening comes at the cost of a restriction of the class of allowed kernels. Under a class of normalized-bilinear--plus--Mat\'ern kernels (with parameter $\nu$), see Assumption~\ref{ass:w1-normalized-matern}(i), a quantitative KSD-to-\(W_1\) comparison, combined with finite-time stability and quantitative long-time results, then gives
\[
\lim_{M \to \infty} \, \limsup_{N \to \infty}\,\sup_{T>0}\mathbb{P}\left[W_1\bigl(\overline\mu_T^N,\overline\mu_T\bigr) \ge M (\log \log N)^{-1/\{6(1+7d/6+2\nu)\}}\right] = 0.
\]
Similarly, expectation bounds are available for sub-Gaussian initial second moment assumptions as in the KSD case.

In Section~\ref{sec:wasserstein-poc}, we address propagation-of-chaos in $W_2$. This is a stronger metric than both the metrics above. The relevant kernel that exploits this geometry is the \emph{unnormalized} bilinear--plus--Mat\'ern kernel; see Assumption~\ref{ass:Wass}(i).  Indeed, this kernel enables a quantitative KSD-to-\(W_2\) comparison. However, the unboundedness of the kernel (due to the unbounded unnormalized bilinear part) places the dynamics outside the scope of most existing works. Our key insight is exploiting a \emph{Lyapunov function}, originally introduced in \cite{BBG2025}, to establish global well-posedness of the finite-particle and mean-field SVGD flows, and obtain sharper (linearly growing) temporal bounds on the time-averaged second-moments of the associated measure-valued processes. We also establish a quantitative Dobrushin-type \(W_2\) stability bound adapted to the unbounded bilinear component, which combines with the linear second-moment bounds to give an exponential improvement over the finite-time stability rates of~\cite{LuLuNolen2019}. Together with long-time estimates in \cite{BBG2025} and the corresponding KSD-to-\(W_2\) comparison, these estimates give an \emph{exponential improvement in the propagation-of-chaos rates} (obtained here in the probability sense): with \(\mathfrak r=\mathfrak r(d,\nu)\) defined in~\eqref{eq:rdef} and $\beta_{\rm init}$ in \eqref{eq:init-w2-prob},
\[
    \sup_{T>0}
    \mathbb P\!\left(
        W_2(\overline\mu_T^N,\overline\mu_T)
        >
        (\log N)^{-\mathfrak r/4}
    \right)
    \le
    C(\log N)^{-\min\{\mathfrak r/4,\beta_{\rm init}\}},
\]
as well as
\[
    \lim_{M\to\infty}\limsup_{N\to\infty}
\sup_{T>0}
\mathbb P\left[
W_2(\mu_T^N,\mu_T)>M(\log N)^{-\mathfrak r/2}
\right]
=0.
\]

In Theorems~\ref{kpartW1} and \ref{kpartW2}, using similar arguments, we also obtain the \emph{$k$-particle propagation-of-chaos estimates} for fixed $k \ge 1$, in $W_1$ and $W_2$ metrics. These results quantify the discrepancy between joint $k$-particle laws from time-averages of product-form distributions, uniformly in time, when the initial locations are exchangeable and `near i.i.d.'.

In Section~\ref{sec:conjgauss}, we obtain a second, complementary class of long-time propagation-of-chaos estimates, showing that substantially faster rates are available when one tests the dynamics through suitably chosen \emph{finite-dimensional Stein observables}. We formulate SVGD for finite-rank matrix-valued kernels and prove an exact conjugacy principle under orientation-preserving diffeomorphisms: Stein features, velocities, and feature-level propagation-of-chaos estimates transform covariantly and are therefore preserved under conjugacy. For the standard Gaussian target, the bilinear matrix-valued kernel \((1+x^\top y)I_d\) closes exactly on the first and second moments. Exploiting this finite-dimensional closure and the stability of the resulting moment dynamics, we prove a genuine uniform-in-time feature-level propagation-of-chaos estimate for non-Gaussian initial laws satisfying suitable moment assumptions and a nondegenerate covariance condition: 
\[
    \E\!\left[
        \sup_{t\ge0}
        D_2^2\!\left(
            (\mu_t^N x,\mu_t^N(xx^\top)),
            (\mu_t x,\mu_t(xx^\top))
        \right)
    \right]
    \le C N^{-1}.
\]
Here $\mu_t x := \int x \, \mu_t(dx), \, \mu_t(xx^\top) := \int xx^\top \,\mu_t(dx)$ (similarly for the quantities involving $\mu^N_t$) and $D^2$ is the moment distance
\[
        D_2\bigl((m,S),(\bar m,\bar S)\bigr)
        :=
        \left(\norm{m-\bar m}^2+\norm{S-\bar S}_F^2\right)^{1/2},
\]
where $\|\cdot\|_F$ denotes the Frobenius norm. For Gaussian initialization, this further recovers, and strengthens, the full \(W_2^2\) propagation-of-chaos rates proved in~\cite{liu2023towards}.

Finally, the conjugacy principle transfers these feature-level bounds to Gaussian-conjugate targets with nonlinear, target-adapted matrix-valued kernels. Writing $A$ for the diffeomorphism that takes $X \sim \pi$ to a Gaussian random variable, we obtain results of the form
\[
    \E\!\left[\sup_{t\ge0}d_{A,\ell}^2(\mu_t^N,\mu_t)\right]
    \le C N^{-1},
    \qquad \ell\in\{1,2\},
\]
for metrics controlling the latent first moments ($\ell=1$) and, respectively, the latent first and second moments ($\ell=2$) of $A(X)$. Thus, while broad distributional metrics for general target-kernel pairs lead to logarithmic rates through the finite-time/long-time cutoff mechanism, exact finite-dimensional Stein-feature closure and stability yield parametric \(N^{-1/2}\) propagation-of-chaos rates over suitable conjugacy classes for informative classes of observables.

\begin{table}[!t]
\centering \footnotesize \setlength{\tabcolsep}{4pt} \renewcommand{\arraystretch}{1.18} \begin{tabularx}{\textwidth}{ @{} >{\raggedright\arraybackslash}p{0.18\textwidth} >{\raggedright\arraybackslash}X >{\raggedright\arraybackslash}p{0.20\textwidth} @{} } \toprule \textbf{Metric} & \textbf{Uniform propagation-of-chaos rate} & \textbf{Result} \\ \midrule

Two-sample Langevin KSD
&
\[
\lim_{M \to \infty} \, \limsup_{N \to \infty}\,\sup_{T>0}\mathbb{P}\left[\KSD_\pi\bigl(\overline\mu_T^N,\overline\mu_T\bigr) \ge M (\log \log N)^{-1/2}\right] = 0.
\]
&
Theorem~\ref{prop:uniform-time-avg-poc}
\\
\hline
\addlinespace

\(W_1\), normalized-bilinear--plus--Mat\'ern kernel (parameter $\nu)$
&
\[
\lim_{M \to \infty} \, \limsup_{N \to \infty}\,\sup_{T>0}\mathbb{P}\left[W_1\bigl(\overline\mu_T^N,\overline\mu_T\bigr) \ge M (\log \log N)^{-\gamma/2}\right] = 0,
\]
\ \ with
$$\gamma := (3(1+7d/6+2\nu))^{-1}.$$
&
Theorem~\ref{prop:uniform-time-averaged-w1}
\\

\hline
\addlinespace

\(W_2\), bilinear--plus--Mat\'ern kernel
&
\[
\lim_{M\to\infty}\limsup_{N\to\infty}
\sup_{T>0}
\mathbb P\left[
W_2(\overline\mu_T^N,\overline\mu_T)>M(\log N)^{-\mathfrak r/2}
\right]
=0,
\]
\ \ with
$\mathfrak r$ defined in \eqref{eq:rdef}.

&
Theorem~\ref{prop:uniform-T-W2}
\\
\hline
\addlinespace

$k$-particle POC
&
\[
\lim_{M \to \infty} \, \limsup_{N \to \infty}\,\sup_{T>0}
\mathbb P\left(
W_1^{(k)}
\left(
\widehat\mu_T^{N,k},
\overline{\mu^{\otimes k}}_T
\right)
>
M(\log \log N)^{-\gamma/2}
\right)=0.
\]
\ \ (analogous statement for $W_2$.)

&
Theorem~\ref{kpartW1} and \ref{kpartW2}
\\

\hline

\addlinespace

Finite-dimensional first- and second-moment metric \(D_2\) for non-Gaussian initialization
&
\[
\E\!\left[
\sup_{t\ge0}
D_2^2\!\left(
(\widetilde m^N(t),\widetilde S^N(t)),
(\widetilde m(t),\widetilde S(t))
\right)
\right]
\lesssim N^{-1}
\]
&
Proposition~\ref{prop:affine-gaussian-poc}
\\
\hline

\addlinespace

Full \(W_2\) metric rate for Gaussian initialization 
& 
\[ \E\!\left[ \sup_{t\ge0} W_2^2(\widetilde\mu_t^N,\widetilde\mu_t) \right] \lesssim r_d(N) \] 
& 
Corollary~\ref{cor:recover-liu-wasserstein-poc} ($r_d(N)$ defined there.)
\\
\hline

\addlinespace

Gaussian-conjugate feature metrics \(d_{A,\ell}\), \(\ell\in\{1,2\}\)
&
\[
\E\!\left[
\sup_{t\ge0}
d_{A,\ell}^2(\mu_t^N,\mu_t)
\right]
\lesssim N^{-1},
\qquad \ell\in\{1,2\},
\]
and hence \(d_{A,\ell}=O_{L^2}(N^{-1/2})\)
&
Corollaries~\ref{cor:conjugate-constant-poc}
and~\ref{cor:conjugate-affine-poc}
\\

\bottomrule
\end{tabularx}
\caption{Summary of the propagation-of-chaos estimates. The KSD, \(W_1\)
and \(W_2\) estimates are uniform-in-averaging-time; and the finite-dimensional feature
bounds are uniform-in-physical-time and therefore also imply the corresponding
time-averaged estimates.}
\label{tab:summary-poc-rates}
\end{table}


\textbf{Notation. } Throughout, we will denote by $C \in (0,\infty)$ a generic constant whose value might change between displays. When clear from context, we will abbreviate $\int_{\mathbb{R}^d}$ by simply $\int$. Throughout, we will let $\muNt$ denote the empirical law of the $N$-particle continuous-time SVGD system started from $\mu^N_0$ encoding the initial particle locations, and let $\mut$ be the associated mean-field SVGD flow (if it exists). $\overline \mu^N_t$ and $\overline \mu_t$ will denote the corresponding time-averaged measures.

\section{Results in Kernel Stein Discrepancy\label{sec:KSD}}
In this section, we obtain uniform-in-time propagation-of-chaos estimates under a Langevin Kernel Stein Discrepancy (KSD) under mild assumptions on the kernel, target and initialization.

\subsection{Setup}\label{sec:ksdset}

Let $\pi(dx) \propto e^{-V(x)}dx$ be the target law on $\R^d$, and let $k$ be the positive-definite kernel used in SVGD.  Denote by $\cH^d$ the corresponding vector-valued RKHS and by $T_\pi$ the Langevin--Stein operator
\[
T_\pi \varphi(x) := -\nabla V(x)\cdot \varphi(x)+\nabla\cdot \varphi(x),
\qquad \varphi\in \cH^d.
\]
For a probability measure $\rho$ on $\R^d$, define the Stein witness map
\[
S_\pi(\rho) := \int \{-k(\cdot,x)\nabla V(x)+\nabla_2 k(\cdot,x)\}\,\rho(dx)\in \cH^d.
\]
The one-sample Kernel Stein Discrepancy (KSD) is
\[
\KSD(\rho\|\pi):= \sup_{\|\varphi\|_{\cH^d}\le 1}\int T_\pi\varphi(x)\,\rho(dx)
= \|S_\pi(\rho)\|_{\cH^d}.
\]
Following Shi--Mackey \cite{ShiMackey2023}, define the Langevin KSD with reference $\pi$ by
\[
\KSD_\pi(\mu,\nu)
:= \sup_{\|\varphi\|_{\cH^d}\le 1}
\left\{\int T_\pi\varphi\,d\mu-\int T_\pi\varphi\,d\nu\right\}.
\]
Equivalently,
\[
\KSD_\pi(\mu,\nu)=\|S_\pi(\mu)-S_\pi(\nu)\|_{\cH^d}.
\]
In particular, $\KSD_\pi$ is symmetric and satisfies the triangle inequality; see Lemma 1 in Shi--Mackey \cite{ShiMackey2023}.  Since $S_\pi(\pi)=0$, one also has
\[
\KSD_\pi(\rho,\pi)=\KSD(\rho\|\pi).
\]
Consequently,
\begin{equation}\label{eq:triangle-pi}
\KSD_\pi(\mu,\nu)
\le \KSD(\mu\|\pi)+\KSD(\nu\|\pi).
\end{equation}
\begin{assumption}\label{ass:poc}
We assume the following.
\begin{enumerate}[label=(\roman*),leftmargin=*]

\item \textbf{Kernel.} The kernel \(k\) is symmetric positive definite and belongs to
\(C_b^4(\mathbb R^d\times\mathbb R^d)\) (that is, the kernel and all its derivatives up to order $4$, assumed to exist, are bounded).

\item \textbf{Potential.} Assume \(V\) satisfies
Lu--Lu--Nolen~\cite[Assumption~2.2]{LuLuNolen2019}, with $q=2$. Also assume Shi--Mackey~\cite[Assumption~1]{ShiMackey2023}, namely the associated score function is Lipschitz, has at least one zero, and has mean zero under $\pi$. 

\item \textbf{Initialization.} 

Let \(P_0^N\) denote the joint law of
\((x_1^N(0),\ldots,x_N^N(0))\), having a $C^2$ density $p_0^N$, and let
\(
        \mu_0^N:=\frac1N\sum_{i=1}^N\delta_{x_i^N(0)}.
\)
Assume that, for some \(\beta>0\)\footnote{$\beta$ is $\Theta(1/d)$ for i.i.d. initializations with finite moments of sufficiently high orders \cite{fournier2015rate}.} and $C_{\rm init}<\infty$,
\begin{equation}\label{eq:init-w1-rate}
\E W_2(\mu^N_0,\mu_0)\le C_{\rm init}N^{-\beta},
\end{equation}
where \(\mu_0\in\mathcal P_2(\mathbb R^d)\) is a deterministic
probability measure with density \(p_0\) satisfying
\[
        \int_{\mathbb R^d}\|x\|^2\,\mu_0(dx)<\infty,
        \qquad
        \KL(\mu_0 \| \pi)<\infty .
\]
Moreover, assume that
\[
        \sup_{N\ge1}\frac1N\KL(P_0^N \| \pi^{\otimes N})<\infty .
\]
\item \textbf{Bounded $C^*$.} Assume $C^* := \sup_{z \in \mathbb{R}^d} C^*(z) < \infty$, where
\begin{equation}\label{eq:Cstardef}
C^*(z) := \nabla_2 k(z,z)\cdot \nabla V(z) + k(z,z)\Delta V(z) - \Delta_2 k(z,z).
\end{equation}
\end{enumerate}
\end{assumption}
Under Assumption~\ref{ass:poc}, the particle ODE~\eqref{eq:svgd-continuous} is globally well-posed. Moreover, by the results of~\cite{LuLuNolen2019}, $\mu^N_\cdot$ has a mean-field limit $\mu_\cdot$ and both $\mu^N$ and $\mu$ are weak solutions of~\eqref{eq:mf-svgd}.

For $T>0$, define the time-averaged empirical and mean-field measures by
\[
\overline\mu_T^N:=\frac1T\int_0^T \muNt\,dt,
\qquad
\overline\mu_T:=\frac1T\int_0^T \mut\,dt.
\]
Our goal in this section is to record a uniform-in-$T$ comparison between $\overline\mu_T^N$ and $\overline\mu_T$ in $\KSD_\pi$.

\subsection{Key estimates}

We collect the three estimates needed to establish our main result in this section.  The constants below are allowed to depend on the dimension, the target, the kernel, and the initial law, but not on $N$ or $T$.

\begin{lemma}[Finite-time POC in Langevin KSD]\label{lem:finite-time-poc}
Let Assumption~\ref{ass:poc} hold. Then there exist constants $C_0,c_0<\infty$ such that, for any $R >0$, $N\ge 3$ and
all $t\ge 0$,
\begin{equation}\label{eq:finite-time-poc}
\E\,\left[\sup_{0\le s\le t} \KSD_\pi(\muNs,\mus)\, \mathbf{1}_{\{\int \|x\|^2 \mu^N_0(dx) \le R\}}\right]
\le C_0 N^{-\beta/2}\exp\{c_0(1+R)\exp(c_0t)\}.
\end{equation}
\end{lemma}

\begin{proof}
By Lu--Lu--Nolen~\cite[Theorem~2.7 and Theorem~2.4]{LuLuNolen2019}, the mean-field SVGD flow \eqref{eq:mf-svgd} is stable
on finite time intervals. In particular, for initial laws $\nu^1,\nu^2$
satisfying $\int\|x\|^2 \, d\nu^i \le R_i$ for some $R_i>0, i=1,2$, the corresponding mean-field solutions $\mus^1,\mus^2$ satisfy
\begin{equation}\label{eq:lln-stability}
\sup_{0\le s\le t}W_1(\mus^1,\mus^2) \le \sup_{0\le s\le t}W_2(\mus^1,\mus^2)
\le A(t,R_1,R_2) W_2(\nu^1,\nu^2).
\end{equation}
Their proof gives a constant of at most double-exponential growth:
\begin{equation}\label{eq:A-double-exp}
A(t,R_1,R_2)\le C\exp\{c(R_1 + R_2)\exp(ct)\},
\end{equation}
where $C,c$ depend on $k,V$ and not $R_1,R_2$.

Applying \eqref{eq:lln-stability} with initial data $\mu_0^N$ and $\mu_0$, and using the
identification of the particle empirical dynamics with the mean-field equation started from the
empirical initial measure, gives 
\begin{equation}\label{eq:w1-poc}
\sup_{0\le s\le t}W_1(\muNs,\mus)
\le C\exp\{c(1 + R)\exp(ct)\} W_2(\mu_0^N,\mu_0),
\end{equation}
on the event $\{\int \|x\|^2 \mu^N_0(dx) \le R\}$.

Next we convert this Wasserstein comparison into Langevin KSD. Shi--Mackey
\cite[Lemma~4]{ShiMackey2023} prove a KSD--Wasserstein comparison of the form
\begin{equation}\label{eq:ksd-w1}
\KSD_\pi(\rho,\nu)
\le C_1 W_1(\rho,\nu)
+
C_2\{M_{\nu,\pi} W_1(\rho,\nu)\}^{1/2},
\end{equation}
where $M_{\nu,\pi}:= \E_{(X,Y) \sim \nu \otimes \pi} [\|X-Y\|^2]$. By the finite-time moment bounds furnished by~Lu--Lu--Nolen~\cite[Theorem~2.4]{LuLuNolen2019}, for every finite $t$,
\begin{equation}\label{eq:moment-growth}
\sup_{0\le s\le t} M_{\mus,\pi}
\le C\exp\{ct\}.
\end{equation}
Combining \eqref{eq:w1-poc}, \eqref{eq:ksd-w1} and \eqref{eq:moment-growth}, and enlarging
constants, yields on the event $\{\int \|x\|^2 \mu^N_0(dx) \le R\}$,
\begin{equation}\label{Kpath}
\sup_{0\le s\le t}
\KSD_\pi(\muNs,\mus)
\le
C\exp\{c(1 + R)\exp(ct)\}
\left[
W_2(\mu_0^N,\mu_0)
+
W_2(\mu_0^N,\mu_0)^{1/2}
\right].
\end{equation}
Taking expectations and using Jensen's inequality gives
\begin{multline*}
\E\,\left[\sup_{0\le s\le t} \KSD_\pi(\muNs,\mus)\, \mathbf{1}_{\{\int \|x\|^2 \mu^N_0(dx) \le R\}}\right]\\
\le
C\exp\{c(1 + R)\exp(ct)\}
\left[
\E W_2(\mu_0^N,\mu_0)
+
\{\E W_2(\mu_0^N,\mu_0)\}^{1/2}
\right].
\end{multline*}
By the assumed initial approximation rate \eqref{eq:init-w1-rate},
\[
\E W_2(\mu_0^N,\mu_0)
+
\{\E W_2(\mu_0^N,\mu_0)\}^{1/2}
\le
C N^{-\beta/2}.
\]
Combining the above two bounds gives \eqref{eq:finite-time-poc}.
\end{proof}
The following lemma is essentially \cite[Theorem 1]{BBG2025}, upon recalling the assumptions
$$
\limsup_{n \to \infty}N^{-1}\KL(P_0^N \| \pi^{\otimes N}) < \infty
$$
and $C^*< \infty$.
\begin{lemma}[Time-averaged finite-particle convergence to $\pi$]\label{ass:particle-to-pi}
Let Assumption~\ref{ass:poc}, without the requirement \eqref{eq:init-w1-rate}, hold. There is a constant $C_1<\infty$ such that, for all $N\ge 1$ and $T>0$,
\begin{equation}\label{eq:particle-to-pi}
\E\,\KSD(\overline\mu_T^N\|\pi)
\le C_1\left(\frac1{\sqrt T}+\frac1{\sqrt N}\right).
\end{equation}
\end{lemma}

\begin{lemma}[Time-averaged mean-field convergence to $\pi$]\label{ass:mean-field-to-pi}
Let $\KL(\muz\|\pi)<\infty$. There is a constant $C_2<\infty$ such that, for all $T>0$,
\begin{equation}\label{eq:mean-field-to-pi}
\KSD(\overline\mu_T\|\pi)
\le \frac{C_2}{\sqrt T}.
\end{equation}
\end{lemma}

\begin{proof}
This follows from the usual mean-field entropy dissipation identity
\[
\frac{d}{dt}\KL(\mut \|\pi)=-\KSD^2(\mut \|\pi),
\]
and by assumption $\KL(\muz\|\pi)<\infty$.  Indeed, as observed in~\cite[Proposition 3]{KorbaSalimArbelLuiseGretton2020},
\[
\frac1T\int_0^T \KSD^2(\mut\|\pi)\,dt
\le \frac{\KL(\muz\|\pi)}{T},
\]
while convexity of $\KSD(\cdot\|\pi)$ yields
\[
\KSD(\overline\mu_T\|\pi)
\le \frac1T\int_0^T \KSD(\mut\|\pi)\,dt
\le \left(\frac1T\int_0^T \KSD^2(\mut \|\pi)\,dt\right)^{1/2}.\qedhere
\]
\end{proof}

\subsection{Uniform-in-averaging-horizon propagation-of-chaos in KSD}

\begin{theorem}[Uniform-in-$T$ POC for the time-averaged flow in Langevin KSD]\label{prop:uniform-time-avg-poc}
Let Assumption~\ref{ass:poc} hold. Recall the constant $c_0$ from Lemma~\ref{lem:finite-time-poc}. Then, there exists a finite constant $C>0$ such that, for any $\delta>0$ and all sufficiently large $N$,
\begin{equation}\label{eq:main-bound-b}
\sup_{T>0}
\E\,\left[\KSD_\pi\bigl(\overline\mu_T^N,\overline\mu_T\bigr)\, \mathbf{1}_{\{\int \|x\|^2 \mu^N_0(dx) \le R\}}\right]
\le C N^{-\beta/2}\exp\{c_0(1+R)(\log N)^{c_0\delta}\} + C(\delta\log\log N)^{-1/2}.
\end{equation}
Consequently,
\begin{equation}\label{KSDprob}
    \lim_{M \to \infty} \, \limsup_{N \to \infty}\,\sup_{T>0}\mathbb{P}\left[\KSD_\pi\bigl(\overline\mu_T^N,\overline\mu_T\bigr) \ge M (\log \log N)^{-1/2}\right] = 0.
\end{equation}
Moreover, suppose that the empirical initial second moments satisfy the uniform sub-Gaussian bound
\[
R_N:=\int_{\mathbb R^d}|x|^2\,\mu_0^N(dx),
\qquad
\sup_{N\ge1}\mathbb E\exp\{\lambda R_N^2\}<\infty
\]
for some \(\lambda>0\). Then there exists \(C<\infty\) such that, for all sufficiently large \(N\),
\begin{equation}\label{KSDebd}
\sup_{T>0}
\mathbb E\,\KSD_\pi(\overline\mu_T^N,\overline\mu_T)
\le
C(\log\log N)^{-1/2}.
\end{equation}

\end{theorem}

\begin{proof}
Fix $\delta>0$, and set
\[
T_N:=\delta\log\log N.
\]
We split the proof into the cases $T\le T_N$ and $T\ge T_N$.

\medskip
\noindent\textbf{Case 1: short averaging horizons, $0<T\le T_N$.}
By convexity of $\KSD_\pi$ in each argument, or equivalently by the norm representation in terms of the Stein witness map,
\begin{align*}
\KSD_\pi(\overline\mu_T^N,\overline\mu_T)
&=\left\|\frac1T\int_0^T\{S_\pi(\muNt)-S_\pi(\mut)\}\,dt\right\|_{\cH^d} \\
&\le \frac1T\int_0^T \KSD_\pi(\muNt,\mut)\,dt.
\end{align*}
Taking expectations and using Lemma~\ref{lem:finite-time-poc},
\begin{align}\label{eq:short-time-bound}
\E\left[\sup_{0\le T\le T_N}\KSD_\pi(\overline\mu_T^N,\overline\mu_T)\, \mathbf{1}_{\{\int \|x\|^2 \mu^N_0(dx) \le R\}}\right]
\le C_0 N^{-\beta/2}\exp\{c_0(1+R)\exp(c_0 T_N)\}.
\end{align}
Since $T_N=\delta\log\log N$,
\(
\exp(c_0T_N)=\exp(c_0\delta\log\log N)=(\log N)^{c_0\delta}.
\)

\medskip
\noindent\textbf{Case 2: long averaging horizons, $T\ge T_N$.}
Using the triangle inequality \eqref{eq:triangle-pi} for the Langevin KSD with reference $\pi$, followed by $\KSD_\pi(\rho,\pi)=\KSD(\rho\|\pi)$, we obtain
\begin{align*}
\KSD_\pi(\overline\mu_T^N,\overline\mu_T)
&\le \KSD_\pi(\overline\mu_T^N,\pi)+\KSD_\pi(\pi,\overline\mu_T) \\
&= \KSD(\overline\mu_T^N\|\pi)+\KSD(\overline\mu_T\|\pi).
\end{align*}
Taking expectations and applying Lemmas \ref{ass:particle-to-pi} and \ref{ass:mean-field-to-pi},
\[
\E\,\KSD_\pi(\overline\mu_T^N,\overline\mu_T)
\le C_1\left(\frac1{\sqrt T}+\frac1{\sqrt N}\right)+\frac{C_2}{\sqrt T}.
\]
Since $T\ge T_N$,
\begin{equation}\label{eq:long-time-bound}
\E\,\KSD_\pi(\overline\mu_T^N,\overline\mu_T)
\le C\left(\frac1{\sqrt{T_N}}+\frac1{\sqrt N}\right)
\le C(\delta\log\log N)^{-1/2}.
\end{equation}

Combining \eqref{eq:short-time-bound} and \eqref{eq:long-time-bound} gives \eqref{eq:main-bound-b}.

To prove \eqref{KSDprob}, observe that by the initialization assumptions,
$$
\sup_{N \ge 1} \E\left[\left(\int_{\mathbb{R}^d} \|x\|^2 \,\mu^N_0(dx)\right)^{1/2}\right] \le \sqrt{2}\sup_{N \ge 1}\,\E\,\left[W_2(\mu^N_0,\mu_0)\right] + \sqrt{2}\left( \int_{\mathbb R^d}\|x\|^2\,\mu_0(dx)\right)^{1/2} < \infty.
$$
Using this observation, along with the bound \eqref{eq:main-bound-b}, we obtain for any $R,\delta>0$,
\begin{align*}
&\sup_{T>0}\mathbb{P}\left[\KSD_\pi\bigl(\overline\mu_T^N,\overline\mu_T\bigr) \ge M (\log \log N)^{-1/2}\right]\\
&\le \sup_{T>0}\mathbb{P}\left[\KSD_\pi\bigl(\overline\mu_T^N,\overline\mu_T\bigr) \ge M (\log \log N)^{-1/2}, \, \int_{\mathbb{R}^d} \|x\|^2\, \mu^N_0(dx) \le R\right] + \mathbb{P}\left[\int_{\mathbb{R}^d} \|x\|^2 \,\mu^N_0(dx) > R\right]\\
&\le C M^{-1}(\delta \log \log N)^{1/2}\left[N^{-\beta/2}\exp\{c_0(1+R)(\log N)^{c_0\delta}\} + (\delta\log\log N)^{-1/2}\right] + CR^{-1/2}.
\end{align*}
Choosing and fixing $\delta>0$ such that $c_0\delta<1$, we thus obtain
\begin{equation*}
    \limsup_{N \to \infty}\sup_{T>0}\mathbb{P}\left[\KSD_\pi\bigl(\overline\mu_T^N,\overline\mu_T\bigr) \ge M (\log \log N)^{-1/2}\right] \le CM^{-1} + CR^{-1/2}.
\end{equation*}
Taking $M \to \infty$ above,
$$
\lim_{M \to \infty}\,\limsup_{N \to \infty}\sup_{T>0}\mathbb{P}\left[\KSD_\pi\bigl(\overline\mu_T^N,\overline\mu_T\bigr) \ge M (\log \log N)^{-1/2}\right] \le CR^{-1/2}.
$$
Since $R>0$ is arbitrary, \eqref{KSDprob} follows.

We now show \eqref{KSDebd} under the additional uniform exponential-square moment bound on \(R_N\). By \eqref{Kpath},
\[
\sup_{0\le s\le t}\KSD_\pi(\mu_s^N,\mu_s)
\le
C\exp\{c(1 + R_N)\exp(ct)\}
\left(D_N+D_N^{1/2}\right),
\qquad
D_N:=W_2(\mu_0^N,\mu_0).
\]
Write $a_N:=c e^{cT_N}=c(\log N)^{c\delta}$. Note that, by Young's inequality, $2a_N R_N \le \lambda R_N^2 + \lambda^{-1} a_N^2$, which gives $C \in (0,\infty)$ such that for any $N \ge 1$,
$$
\E\left[e^{2a_N R_N}\right] \le e^{a_N^2/\lambda}\E\left[e^{\lambda R_N^2}\right].
$$
Using this observation, along with \(D_N^2\le 2R_N+2\int |x|^2\,d\mu_0\), the exponential-square bound on \(R_N\), and \(\mathbb E D_N\le C_{\rm init}N^{-\beta}\), we obtain by H\"older's inequality,
\[
\mathbb E\left[e^{a_NR_N}\left(D_N+D_N^{1/2}\right)\right]
\le
C\exp\{a_N^2/\lambda\}N^{-\beta/3}.
\]
Taking \(T_N=\delta\log\log N\) with \(2c\delta <1\), the short-time contribution is
\[
\mathbb{E}\left[\sup_{0\le s\le T_N}\KSD_\pi(\mu_s^N,\mu_s)\right] \le CN^{-\frac{\beta}{3} + o(1)}.
\]
The long-time estimate is unchanged:
\[
\sup_{T\ge T_N}
\mathbb E\,\KSD_\pi(\overline\mu_T^N,\overline\mu_T)
\le
C\left(T_N^{-1/2}+N^{-1/2}\right)
\le
C(\log\log N)^{-1/2}.
\]
Combining the two regimes proves \eqref{KSDebd}.
\end{proof}

\section{Results in Wasserstein-1 distance}\label{sec:w1}

We record a Wasserstein--1 analogue of the uniform-in-averaging-horizon KSD estimate.  The propagation-of-chaos in a stronger metric requires a detailed comparison between rates of KSD and Wasserstein convergence. This has been recently achieved by Kanagawa--Barp--Gretton--Mackey~\cite{KanagawaBarpGrettonMackey2025} when the kernel has a special normalized-bilinear $+$ Mat\'ern structure.  This keeps the SVGD kernel bounded, while still giving enough linear-growth test functions in the Stein RKHS to control first moments.

\begin{assumption}
\label{ass:w1-normalized-matern}
The following conditions hold.
\begin{enumerate}[label=(\roman*),leftmargin=*]
\item  \textbf{Kernel.}  Fix $\tau>0$, a strictly positive definite matrix $\Sigma\in\R^{d\times d}$, and a Mat\'ern smoothness parameter $\nu>2$.  Define
\[
        w_1(x):=(\tau^2+\|x\|^2)^{1/2},
        \qquad
        \overline k_{\rm lin}(x,y)
        :=\frac{\tau^2+\langle x,y\rangle}{w_1(x)w_1(y)}.
\]
Let
\[
        \Phi_{\nu,\Sigma}(z)
        :=\frac{2^{1-(d/2+\nu)}}{\Gamma(d/2+\nu)}
          \|\Sigma z\|^{\nu}_2 K_{-\nu}(\|\Sigma z\|_2),
        \qquad z\in\R^d,
\]
where $K_{-\nu}$ is the modified Bessel function of the second kind.  The SVGD kernel is
\[
        k_+(x,y):=\Phi_{\nu,\Sigma}(x-y)+\overline k_{\rm lin}(x,y).
 \]

\item \textbf{Potential regularity and growth.}  The potential $V$ is thrice continuously differentiable and, for some $L < \infty$,
\begin{equation}\label{eq:V-growth-hessian}
        \norm{\nabla V(x)}\le L(1+\norm{x}),
        \qquad
        \norm{\nabla^2V(x)}_{\mathrm{op}}\le L,
        \qquad x\in\R^d.
    \end{equation}
In particular, $\sup_x \Delta V(x)<\infty$.

\item \textbf{Langevin stability assumptions.} Let
\[
        \, d Z_t=-\frac12\nabla V(Z_t)\,\, d t+\, d B_t,
        \qquad Z_0=x,
\]
be the ordinary Langevin diffusion with invariant law $\pi$.  There are constants $a>0$ and $\beta_0,\beta_1\ge0$ such that the following dissipativity assumption holds (see \cite[Th1]{barp2024targeted})
\begin{equation}\label{eq:dissipativity-assumption}
        -\langle x,\nabla V(x)\rangle+d
        \le -a\|x\|^2+\beta_1\|x\|+\beta_0,
        \qquad x\in\R^d .
\end{equation}
The diffusion has $W_s$-rate functions $\rho_s:[0,\infty)\to[0,\infty)$, $s=1,2$, satisfying
\[
        W_s\bigl(\Law(Z_t^x),\Law(Z_t^y)\bigr)
        \le \rho_s(t)\|x-y\|,
        \qquad x,y\in\R^d,
        \qquad t\ge0,
\]
and
\[
        \int_0^\infty \rho_1(t)\,\, d t<\infty .
\]
See \cite[Assumption 4]{KanagawaBarpGrettonMackey2025}.

\item \textbf{Initial data.}  We make the same initialization assumption as in Assumption~\ref{ass:poc}.
\end{enumerate}

Well-posedness of the particle system and mean-field equations again follow from Assumption~\ref{ass:w1-normalized-matern} and \cite{LuLuNolen2019}.

\end{assumption}

\subsection{Key estimates}
The following lemma furnishes the required finite-time stability by adapting the arguments of Lu--Lu--Nolen~\cite{LuLuNolen2019}.
\begin{lemma}[Finite-time propagation-of-chaos in $W_1$]
\label{lem:finite-time-w1}
Let Assumption~\ref{ass:w1-normalized-matern} hold.  Then there exist constants $C_0,c_0<\infty$ such that, for any $R>0$, $N\ge3$ and all $t\ge0$,
\begin{equation}\label{W1stab}
        \mathbb E\Big[\sup_{0\le s\le t}
        W_1(\muNs,\mus)\, \mathbf{1}_{\{\int \|x\|^2 \mu^N_0(dx) \le R\}}\Big]
        \le C_0N^{-\beta}\exp\{c_0(1+R)\exp(c_0t)\}.
\end{equation}
\end{lemma}

\begin{proof}
We first check that the finite-time stability argument of
Lu--Lu--Nolen~\cite[Theorem~2.4 and 2.7]{LuLuNolen2019} applies to the kernel \(k_+\).  The Matérn part
\(\Phi_{\nu,\Sigma}(x-y)\) is positive definite and belongs to
\(C_b^4(\mathbb R^d\times\mathbb R^d)\) when \(\nu>2\).  The normalized
bilinear part can be written as
\[
        \bar k_{\rm lin}(x,y)
        =
        \Big\langle
        \frac{(\tau,x)}{(\tau^2+\|x\|^2)^{1/2}},
        \frac{(\tau,y)}{(\tau^2+\|y\|^2)^{1/2}}
        \Big\rangle_{\mathbb R^{d+1}},
\]
and is therefore positive definite.  Since \(\tau>0\), the map
\(x\mapsto (\tau,x)(\tau^2+\|x\|^2)^{-1/2}\) is smooth with bounded
derivatives of all orders.  Hence \(k_+\) is a symmetric positive definite
kernel in \(C_b^4(\mathbb R^d\times\mathbb R^d)\).

It remains to check the growth assumptions on \(V\).  By Assumption~\ref{ass:w1-normalized-matern},
\[
        \|\nabla V(x)\|\le L(1+\|x\|),
        \qquad
        \|\nabla^2 V(x)\|_{\mathrm{op}}\le L .
\]
The dissipativity condition
\[
        -\langle x,\nabla V(x)\rangle+d
        \le -a\|x\|^2+\beta_1\|x\|+\beta_0
\]
implies, after adding a harmless constant to \(V\), that
\[
        1+\|x\|^2 \le C(1+V(x)).
\]
Indeed, integrating the resulting radial lower bound for
\(\theta\cdot\nabla V(r\theta)\) gives quadratic growth from below.
Consequently,
\[
        \|\nabla V(x)\|^2 \le C(1+V(x)),
\]
and the boundedness of \(\nabla^2V\) gives
\[
        \|\nabla^2V(\lambda x+(1-\lambda)y)\|^2
        \le C \le C(1+V(x)+V(y)).
\]
Thus Lu--Lu--Nolen~\cite[Assumption~2.2]{LuLuNolen2019}, with $q=2$, is satisfied.

Strictly speaking, the normalized bilinear term is not translation
invariant, so~\cite[Theorem~2.4 and 2.7]{LuLuNolen2019} do not apply verbatim.  However, their
Dobrushin stability proof only uses the boundedness of sufficiently many
kernel derivatives and the preceding \(q=2\) growth bounds.  For completeness
we recall the resulting estimate.  Put
\[
        B(x,y):=-k_+(x,y)\nabla V(y)+\nabla_2 k_+(x,y),
        \qquad
        v_\rho(x):=\int B(x,y)\,\rho(dy).
\]
The bounds just proved imply
\[
        |B(x,y)-B(x',y)|
        \le C(1+\|y\|)|x-x'|,
\]
and
\[
        |B(x,y)-B(x,z)|
        \le C(1+\|y\|+\|z\|)\|y-z\|.
\]
Therefore, if \(\rho_t,\nu_t\) are two mean-field solutions and
\(M_2(t):=\int\|x\|^2\rho_t(dx)+\int\|x\|^2\nu_t(dx)\), then the
coupling argument of~\cite[Theorem~2.7]{LuLuNolen2019} gives
\[
        \frac{d}{dt}W_2(\rho_t,\nu_t)
        \le C(1+M_2(t)^{1/2})\,W_2(\rho_t,\nu_t).
\]
The moment estimate~\cite[Theorem~2.4]{LuLuNolen2019} gives, on
each finite interval,
\[
        \sup_{0\le s\le t}M_2(s)\le C\exp(ct).
\]
Thus, if $\int\|x\|^2 \,\rho_0(dx) \le R_1$ and $\int\|x\|^2 \,\nu_0(dx) \le R_2$, Gronwall's inequality then yields
\[
        \sup_{0\le s\le t}W_2(\rho_s,\nu_s)
        \le A(t,R_1,R_2)W_2(\rho_0,\nu_0),
        \qquad
        A(t)\le C\exp\{c(R_1 + R_2)\exp(ct)\}.
\]
Since \(W_1\le W_2\), applying this with
\(\rho_0=\mu_0^N\) and \(\nu_0=\mu_0\), and using the identification of the
particle empirical law with the mean-field equation started from
\(\mu_0^N\), gives on the event $\{\int \|x\|^2 \mu^N_0(dx) \le R\}$,
\begin{equation}\label{w1path}
        \sup_{0\le s\le t}W_1(\muNs,\mus)
        \le C\exp\{c(1 + R)\exp(ct)\}\,W_2(\mu_0^N,\mu_0).
\end{equation}
Taking expectations and using the initialization assumption
\(
        \mathbb E\left[ W_2(\mu_0^N,\mu_0)\right] \le C_{\rm init}N^{-\beta}
\)
proves \eqref{W1stab}.
\end{proof}

\begin{lemma}[Time-averaged KSD convergence]
\label{lem:time-averaged-ksd}
Let Assumption~\ref{ass:w1-normalized-matern} hold.  There is a constant $C_1<\infty$ such that, for all $N\ge1$ and all $T>0$,
\[
        \E \KSD(\overline \mu_T^N \| \pi)
        \le C_1\left(T^{-1/2}+N^{-1/2}\right),
        \qquad
        \KSD(\overline\mu_T \| \pi )\le C_1T^{-1/2}.
\]
\end{lemma}

\begin{proof}
This follows as in Lemmas~\ref{ass:particle-to-pi} and \ref{ass:mean-field-to-pi} on observing that the bounded $C^*$ condition in  Assumption~\ref{ass:poc} holds with kernel $k_+$ as $\nabla_2 k_+(z,z)=0$ for all $z \in \mathbb{R}^d$ and $\sup_x \Delta V(x)<\infty$.
\end{proof}

\begin{lemma}[KSD-to-$W_1$ comparison for the normalized-bilinear--Mat\'ern kernel]
\label{lem:ksd-to-w1}
Let Assumption~\ref{ass:w1-normalized-matern} hold and let $D(\rho)=\KSD(\rho\Vert\pi)$ with kernel $k_+$.  Set
\[
        t_0:=\frac{7d}{6}+2\nu,
        \qquad
        \gamma:=\frac{1}{3(1+t_0)}.
\]
Then there is a constant $C_2<\infty$ such that every probability measure $\rho\in\cP_1(\R^d)$ with finite $D(\rho)$ satisfies
\[
        W_1(\rho,\pi)
        \le C_2\{D(\rho)^\gamma+D(\rho)\}.
\]
\end{lemma}

\begin{proof}
We apply Kanagawa--Barp--Gretton--Mackey \cite[Theorem~3.3]{KanagawaBarpGrettonMackey2025} with $q=1$, $q_m=0$, and $m\equiv\Id_d$.  For the pseudo-Lipschitz integrated probability metric (IPM) $d_{\cF_1}$ used there (see its definition in \cite[Equation (4)]{KanagawaBarpGrettonMackey2025}), it gives a constant $A<\infty$ such that, for every admissible $\rho$,
\[
        d_{\cF_1}(\pi,\rho)
        \le A\left(1\vee D(\rho)^{t_0/(1+t_0)}
        \vee D(\rho)^{(4/3)/(1+t_0)}\right)D(\rho)^{1/(1+t_0)},
\]
where
\[
        t_0=\left(1+\frac16\right)d+\frac{1+5}{3}\nu
        =\frac{7d}{6}+2\nu.
\]
Put $a:=(1+t_0)^{-1}$.  Since $\max\{u,v,w\}\le u+v+w$ for nonnegative $u,v,w$,
\[
        d_{\cF_1}(\pi,\rho)
        \le A\{D(\rho)^a+D(\rho)+D(\rho)^{7a/3}\}.
\]
The comparison between $d_{\cF_1}$ and $W_1$ in \cite[Equation (16)]{KanagawaBarpGrettonMackey2025} gives constants $c_1,c_2>0$, depending only on $d$, such that
\[
        \min\{c_1W_1(\rho,\pi),c_2W_1(\rho,\pi)^3\}
        \le d_{\cF_1}(\pi,\rho).
\]
Consequently,
\[
        W_1(\rho,\pi)
        \le C\{d_{\cF_1}(\pi,\rho)^{1/3}+d_{\cF_1}(\pi,\rho)\}.
\]
Combining the two displays and using subadditivity of $x\mapsto x^{1/3}$ gives
\[
        W_1(\rho,\pi)
        \le C\{D^{a/3}+D^{1/3}+D^{7a/9}+D^a+D+D^{7a/3}\},
\]
where $D=D(\rho)$.  Since $\nu>2$ and $d\ge1$, we have $a<6/37$.  Hence each exponent above lies in the interval $[a/3,1]$.  With $\gamma=a/3$, the elementary bound $D^p\le D^\gamma+D$ for $p\in[\gamma,1]$ and $D\ge0$ proves
\[
        W_1(\rho,\pi)\le C_2\{D(\rho)^\gamma+D(\rho)\}.
\]
\end{proof}

Lemmas~\ref{lem:time-averaged-ksd} and \ref{lem:ksd-to-w1} give the following estimate.

\begin{lemma}[Time-averaged convergence to $\pi$ in $W_1$]
\label{lem:time-averaged-w1-to-pi}
Let Assumption~\ref{ass:w1-normalized-matern} hold and let $\gamma$ be as in Lemma~\ref{lem:ksd-to-w1}.  There is a constant $C_3<\infty$ such that, for all $N\ge1$ and all $T>0$,
\[
        \E W_1(\overline\mu_T^N,\pi)
        \le C_3\left(T^{-\gamma/2}+N^{-\gamma/2}\right),
        \qquad
        W_1(\overline\mu_T,\pi)\le C_3T^{-\gamma/2}.
\]
\end{lemma}

\subsection{Uniform-in-averaging-horizon propagation-of-chaos in $W_1$}

\begin{theorem}[Uniform-in-$T$ POC for the time-averaged flow in $W_1$]
\label{prop:uniform-time-averaged-w1}
Let Assumption~\ref{ass:w1-normalized-matern} hold.  Recall the constant $c_0$ from Lemma~\ref{lem:finite-time-w1}. Let $$\gamma := (3(1+7d/6+2\nu))^{-1}.$$ Then, there exists a finite constant $C>0$ such that, for any $\delta>0$ and all sufficiently large $N$,
\begin{equation}\label{eq:main-bound-W1}
\sup_{T>0}
\E\,\left[W_1\bigl(\overline\mu_T^N,\overline\mu_T\bigr)\, \mathbf{1}_{\{\int \|x\|^2 \mu^N_0(dx) \le R\}}\right]
\le C N^{-\beta}\exp\{c_0(1+R)(\log N)^{c_0\delta}\} + C(\delta\log\log N)^{-\gamma/2}.
\end{equation}
Consequently,
\begin{equation}\label{W1prob}
    \lim_{M \to \infty} \, \limsup_{N \to \infty}\,\sup_{T>0}\mathbb{P}\left[W_1\bigl(\overline\mu_T^N,\overline\mu_T\bigr) \ge M (\log \log N)^{-\gamma/2}\right] = 0.
\end{equation}
Moreover, suppose that the empirical initial second moments satisfy the uniform sub-Gaussian bound
\begin{equation}\label{Gaussexp}
R_N:=\int_{\mathbb R^d}|x|^2\,\mu_0^N(dx),
\qquad
\sup_{N\ge1}\mathbb E\exp\{\lambda R_N^2\}<\infty
\end{equation}
for some \(\lambda>0\). Then there exists \(C<\infty\) such that, for all sufficiently large \(N\),
\begin{equation}\label{W1ebd}
\sup_{T>0}
\mathbb E\,W_1(\overline\mu_T^N,\overline\mu_T)
\le
C(\log\log N)^{-\gamma/2}.
\end{equation}
\end{theorem}

\begin{proof}
Choose $\delta>0$ so small that $c_0\delta<1$, where $c_0$ is the constant from Lemma~\ref{lem:finite-time-w1}.  Set
\[
        T_N:=\delta\log\log N.
\]
We split the proof into short and long averaging horizons.

\medskip
\noindent\textbf{Short averaging horizons: $0<T\le T_N$.}
By Kantorovich--Rubinstein duality,
\[
\begin{aligned}
        W_1(\overline\mu_T^N,\overline\mu_T)
        &=\sup_{\Lip(f)\le1}\int f\,\, d(\overline\mu^N_T-\overline\mu_T) \\
        &=\sup_{\Lip(f)\le1}\frac1T\int_0^T
          \int f\,\, d(\muNt-\mut)\,\, d t
        \le \frac1T\int_0^T W_1(\muNt,\mut)\,\, d t.
\end{aligned}
\]
Taking expectations and using Lemma~\ref{lem:finite-time-w1},
\[
        \mathbb E\Big[\sup_{0\le T\le T_N}
        W_1(\overline\mu^N_T,\overline\mu_T)\, \mathbf{1}_{\{\int \|x\|^2 \mu^N_0(dx) \le R\}}\Big]
        \le C_0N^{-\beta}\exp\{c_0(1+R)\exp(c_0T_N)\}.
\]
\medskip
\noindent\textbf{Long averaging horizons: $T\ge T_N$.}
By the triangle inequality,
\[
        W_1(\overline\mu_T^N,\overline\mu_T)
        \le W_1(\overline\mu_T^N,\pi)+W_1(\overline\mu_T,\pi).
\]
Taking expectations and applying Lemma~\ref{lem:time-averaged-w1-to-pi}, with $\gamma$ as in Lemma~\ref{lem:ksd-to-w1},
\[
        \E W_1(\overline\mu_T^N,\overline\mu_T)
        \le C\left(T^{-\gamma/2}+N^{-\gamma/2}\right),
\]
which gives
\[
        \sup_{T\ge T_N}\E W_1(\overline\mu_T^N,\overline\mu_T)
        \le C(\delta\log\log N)^{-\gamma/2}
\]
for all sufficiently large $N$.  Combining the two regimes completes the proof of \eqref{eq:main-bound-W1}. 

The proofs of \eqref{W1prob} and \eqref{W1ebd} follow exactly as those of \eqref{KSDprob} and \eqref{KSDebd}.
\end{proof}

The following theorem gives the associated \(k\)-particle propagation-of-chaos estimate for any fixed \(k \ge 1\). It quantifies, uniformly over the averaging horizon, the discrepancy between the time-averaged joint law of \(k\) particles and the corresponding time-averaged product law, under exchangeable initialization.

\begin{theorem}[Time-averaged \(k\)-particle chaos]\label{kpartW1}
Let Assumption~\ref{ass:w1-normalized-matern} hold and suppose the joint initial law of the particles is exchangeable. For fixed \(k\ge1\), define
the random \(k\)-fold empirical occupation measure
\[
        \widehat\mu_T^{N,k}
        :=
        \frac1T\int_0^T (\mu_t^N)^{\odot k}\,dt,
        \qquad
        \overline{\mu^{\otimes k}}_T
        :=
        \frac1T\int_0^T \mu_t^{\otimes k}\,dt,
\]
where
\[
(\mu_t^N)^{\odot k}
:=
\frac{1}{(N)_k}
\sum_{\substack{1\le i_1,\ldots,i_k\le N\\
i_1,\ldots,i_k\ {\rm distinct}}}
\delta_{\bigl(x_{i_1}^N(t),\ldots,x_{i_k}^N(t)\bigr)},
\qquad
(N)_k:=N(N-1)\cdots(N-k+1).
\]
Let \(W_1^{(k)}\) be the Wasserstein distance on \((\mathbb R^d)^k\)
associated with the product cost
\(
        d_k(x,y):=\sum_{\ell=1}^k \|x_\ell-y_\ell\| .
\)
Then, with $\gamma$ as in Theorem~\ref{prop:uniform-time-averaged-w1},
\begin{equation}\label{kpartprob}
\lim_{M \to \infty} \, \limsup_{N \to \infty}\,\sup_{T>0}
\mathbb P\left(
W_1^{(k)}
\left(
\widehat\mu_T^{N,k},
\overline{\mu^{\otimes k}}_T
\right)
>
M(\log \log N)^{-\gamma/2}
\right)=0.
\end{equation}
Moreover, suppose the initial distributions satisfy \eqref{Gaussexp}. 
Define, for $k \ge 1$,
\[
        \bar\mu_T^{N,k}
        :=
        \frac1T\int_0^T
        \mathcal L(x_1^N(t),\ldots,x_k^N(t))\,dt,
        \qquad
        \overline{\mu^{\otimes k}}_T
        :=
        \frac1T\int_0^T \mu_t^{\otimes k}\,dt ,
\]
where $L(x_1^N(t),\ldots,x_k^N(t))$ denotes the joint law of $(x_1^N(t),\ldots,x_k^N(t))$.
Then there is \(C_k<\infty\) such that, for all sufficiently large \(N\),
\begin{equation}\label{kpartW1exp}
        \sup_{T>0}
        W_1^{(k)}
        \left(
        \bar\mu_T^{N,k},
        \overline{\mu^{\otimes k}}_T
        \right)
        \le
        C_k(\log\log N)^{-\gamma/2}.
\end{equation}
\end{theorem}

\begin{proof}
Throughout the proof \(C_k\) denotes a finite constant depending on \(k\)
and on the constants in Assumption~\ref{ass:w1-normalized-matern}, whose
value may change from line to line. 
We use the following standard consequence of exchangeability. Since the
initial law is exchangeable and the particle dynamics are permutation
equivariant, the law of the particle system is exchangeable for every
\(t\ge0\). Hence, for every deterministic \(\nu\in\mathcal P_1(\mathbb R^d)\),
\begin{equation}\label{3.7}
W_1^{(k)}
\Big(
(\mu_t^N)^{\odot k},\nu^{\otimes k}
\Big)
\le
k\,W_1(\mu_t^N,\nu)
+
\frac{k^2(k-1)}{N}\,
\int_{\mathbb{R}^d} \|x\|\,\mu_t^N(dx).
\end{equation}
Indeed, the $k$-fold product of the optimal coupling between $\mu^N_t$ and $\nu$ gives
\[
W_1^{(k)}\big((\mu_t^N)^{\otimes k},\nu^{\otimes k}\big)
\le k W_1(\mu_t^N,\nu),
\]
while comparing sampling with and without replacement gives
\[
W_1^{(k)}\big((\mu_t^N)^{\odot k},(\mu_t^N)^{\otimes k}\big)
\le
\frac{k^2(k-1)}{N}\int\|x\|\,\mu_t^N(dx).
\]
These observations give~\eqref{3.7}. Then, by convexity of \(W_1^{(k)}\),
\begin{equation}
W_1^{(k)}
\left(
\widehat\mu_T^{N,k},\overline{\mu^{\otimes k}}_T
\right)
\le
\frac1T\int_0^T
W_1^{(k)}
\Big(
(\mu_t^N)^{\odot k},\mu_t^{\otimes k}
\Big)\,dt .
\label{3.8}
\end{equation}
We now argue exactly as in the proof of Theorem~\ref{prop:uniform-time-averaged-w1}.
Let \(T_N=\delta\log\log N\), with \(\delta>0\) small.

If \(0<T\le T_N\), then \eqref{3.7} (with $\nu = \mu_t$), \eqref{3.8}, Lemma~\ref{lem:finite-time-w1},
and the finite-time moment bound~\cite[Theorem~2.4]{LuLuNolen2019} give constants $C_k,c$ such that for any $R>0$,
\begin{multline*}
\mathbb{E}\left[\sup_{0 \le T \le T_N}\,W_1^{(k)}
\left(
\widehat\mu_T^{N,k},\overline{\mu^{\otimes k}}_T
\right)\, \mathbf{1}_{\{\int \|x\|^2 \mu^N_0(dx) \le R\}}\right]\\
 \le
C_k N^{-\beta}\exp\{c(1+R)\exp(cT_N)\}
+
C_kN^{-1}\exp(cT_N).
\end{multline*}

If \(T\ge T_N\), use the triangle inequality through \(\pi^{\otimes k}\).
By~\eqref{3.7}, now with \(\nu=\pi\), \eqref{3.8}, and by the $k$-fold product of the optimal coupling for
\(\mu_t\) and \(\pi\),
\begin{align}\label{3.9}
\mathbb{E}\left[W_1^{(k)}
\left(
\widehat\mu_T^{N,k},\overline{\mu^{\otimes k}}_T
\right)\right] & \le \frac1T\int_0^T
\mathbb{E} \,W_1^{(k)}
\Big(
(\mu_t^N)^{\odot k},\mu_t^{\otimes k}
\Big)\,dt\notag\\
&\le \frac1T\int_0^T
\mathbb{E}\,W_1^{(k)}
\Big(
(\mu_t^N)^{\odot k},\pi^{\otimes k}
\Big)\,dt + \frac1T\int_0^T
\mathbb{E}\,W_1^{(k)}
\Big(\mu_t^{\otimes k},\pi^{\otimes k}
\Big)\,dt\notag\\
&\le
\frac{k}{T}\int_0^T
\mathbb E W_1(\mu_t^N,\pi)\,dt
+
\frac{k}{T}\int_0^T
W_1(\mu_t,\pi)\,dt  \notag\\
&\qquad
+
\frac{k^2(k-1)}{NT}\int_0^T
\mathbb E\int \|x\|\,\mu_t^N(dx)\,dt .
\end{align}
The same entropy-dissipation/KSD-to-\(W_1\) argument used in
Lemmas~\ref{lem:ksd-to-w1}--\ref{lem:time-averaged-w1-to-pi}
gives the time-integrated estimates
\[
\frac1T\int_0^T
\mathbb E W_1(\mu_t^N,\pi)\,dt
\le
C\bigl(T^{-\gamma/2}+N^{-\gamma/2}\bigr),
\qquad
\frac1T\int_0^T
W_1(\mu_t,\pi)\,dt
\le
CT^{-\gamma/2}.
\]
The remaining first-moment term in~\eqref{3.9} is uniformly bounded by the observation \(\int\|x\|\,\mu_t^N(dx) \le \int \|x\|\,\pi(dx) + W_1(\mu^N_t,\pi)\) in conjunction with the above averaged \(W_1\)
bound. Therefore, for \(T\ge T_N\),
\begin{align}\label{intk}
\mathbb{E}\left[W_1^{(k)}
\left(
\widehat\mu_T^{N,k},\overline{\mu^{\otimes k}}_T
\right)\right]
&\le \frac1T\int_0^T
\mathbb{E} \,W_1^{(k)}
\Big(
(\mu_t^N)^{\odot k},\mu_t^{\otimes k}
\Big)\,dt\notag\\
&\le C_k\bigl(T^{-\gamma/2}+N^{-\gamma/2}+N^{-1}\bigr)
\le
C_k(\log\log N)^{-\gamma/2}.
\end{align}
Combining the two regimes proves
\begin{multline*}
        \sup_{T>0}
        \mathbb{E}\left[W_1^{(k)}
\left(
\widehat\mu_T^{N,k},\overline{\mu^{\otimes k}}_T
\right)\, \mathbf{1}_{\{\int \|x\|^2 \mu^N_0(dx) \le R\}}\right]\\
        \le C_k N^{-\beta}\exp\{c(1+R)\exp(cT_N)\}
+
C_kN^{-1}\exp(cT_N) +
        C_k(\log\log N)^{-\gamma/2}.
\end{multline*}
\eqref{kpartprob} now follows from the above estimate exactly as \eqref{KSDprob} followed from \eqref{eq:main-bound-b}.

To obtain \eqref{kpartW1exp}, observe that, by exchangeability, \(\mathcal L(x_1^N(t),\ldots,x_k^N(t))
=\E[(\mu_t^N)^{\odot k}]\), and hence \(\bar\mu_T^{N,k} = \E[\widehat\mu_T^{N,k}]\) for $t,T \ge 0$.
Hence, by convexity of \(W_1^{(k)}\),
\[
W_1^{(k)}
        \left(
        \bar\mu_T^{N,k},
        \overline{\mu^{\otimes k}}_T
        \right) \le \mathbb E\, W_1^{(k)}
\left(
\widehat\mu_T^{N,k},\overline{\mu^{\otimes k}}_T
\right).
\]
By \eqref{3.8}, \eqref{3.7}, \eqref{w1path} and~\cite[Theorem~2.4]{LuLuNolen2019},
\begin{align*}
\sup_{0 \le T \le T_N}\,W_1^{(k)}
\left(
\widehat\mu_T^{N,k},\overline{\mu^{\otimes k}}_T
\right) &\le \sup_{0 \le T \le T_N}\,W_1^{(k)}
\left(
(\mu_T^N)^{\odot k},\mu^{\otimes k}_T
\right)\\
&\le C\exp\{c(1 + R_N)\exp(ct)\}\,W_2(\mu_0^N,\mu_0) + C_k N^{-1} \exp(cT_N)\,R_N^{1/2},
\end{align*}
where recall $R_N:=\int_{\mathbb R^d}|x|^2\,\mu_0^N(dx)$. Now, proceeding exactly as in the proof of \eqref{KSDebd}, using the sub-Gaussian assumption on $R_N$ and choosing $\delta$ sufficiently small, we obtain
\begin{equation}\label{ann1}
\sup_{0 \le T \le T_N}\,W_1^{(k)}
\left(
\bar\mu_T^{N,k},\overline{\mu^{\otimes k}}_T
\right) \le \mathbb E\left[\sup_{0 \le T \le T_N}\,W_1^{(k)}
\left(
\widehat\mu_T^{N,k},\overline{\mu^{\otimes k}}_T
\right)\right] = o\left((\log \log N)^{-\gamma/2}\right).
\end{equation}
For the long-time estimate, using convexity of \(W_1^{(k)}\) again,
\begin{align*}
W_1^{(k)}
        \left(
        \bar\mu_T^{N,k},
        \overline{\mu^{\otimes k}}_T
        \right) &\le \frac{1}{T}\int_0^T W_1^{(k)}
\Big(\mathcal L(x_1^N(t),\ldots,x_k^N(t)),\mu_t^{\otimes k}\Big) \, dt\\
&\le
\frac{1}{T}\int_0^T\E W_1^{(k)}\big((\mu_t^N)^{\odot k},\mu_t^{\otimes k}\big)\, dt.
\end{align*}
Hence, by \eqref{intk}, for $T \ge T_N$,
\begin{equation}\label{baro}
\mathbb E\left[
W_1^{(k)}\left(
\bar\mu_T^{N,k},\overline{\mu^{\otimes k}}_T
\right)\right] \le C_k(\log\log N)^{-\gamma/2}.
\end{equation}
\eqref{kpartW1exp} follows from \eqref{ann1} and \eqref{baro}.
\end{proof}


\section{Results in Wasserstein-2 distance}
\label{sec:wasserstein-poc}

We now obtain a Wasserstein--2 version of Theorem~\ref{prop:uniform-time-avg-poc}. Although the basic idea still relies on a cutoff argument separately treating moderate and large times, the unbounded kernels considered here, which are crucial for KSD-to-\(W_2\) control \cite{KanagawaBarpGrettonMackey2025}, present several challenges. These kernels have a \emph{non-normalized bilinear component} in addition to a Mat\'ern part (see \eqref{eq:bilinear-matern-kernel}). However, somewhat surprisingly, this unbounded bilinear part produces additional stability which results in linear growth bounds of the associated second moments along the SVGD flow (see Lemma~\ref{lem:BBG-lem5}). This, in conjunction with a novel Dobrushin type stability estimate for unbounded kernels obtained in Lemma~\ref{lem:finite-time-w2-stability}, and Wasserstein--2 control obtained in \cite{BBG2025} for the large-time regime (recorded in Lemma~\ref{lem:large-time-w2-to-pi}), leads to a uniform-in-time propagation-of-chaos estimate with a \emph{better rate}; see Theorem~\ref{prop:uniform-T-W2}.

We also remark here that the additional bilinear term in the kernel leads to exponential dependence of the stability estimate in Lemma~\ref{lem:finite-time-w2-stability} on the second moment process. Consequently, Theorem~\ref{prop:uniform-T-W2} is a uniform-in-time propagation-of-chaos result in probability compared to the statements for KSD and $W_1$ in expectation given in Theorems~\ref{prop:uniform-time-avg-poc} and \ref{prop:uniform-time-averaged-w1}.


\begin{assumption}\label{ass:Wass}
\begin{enumerate}[label=(\roman*),leftmargin=*]

\item \textbf{Kernel.} The SVGD kernel is of the bilinear-plus-Mat\'ern
form
\begin{equation}\label{eq:bilinear-matern-kernel}
        \widetilde k(u,v)
        = 1+\langle u,v\rangle + \Psi(u-v),
\end{equation}
where $\Psi = \Psi_{\rm mk}$ is the Mat\'ern component appearing in
\cite[Equation~(12)]{BBG2025}. More precisely, for a strictly positive definite matrix $\Sigma$ and a
smoothness parameter $\nu>0$, set
\[
    \Psi_{\rm mk}(z)
    :=
    \frac{2^{1-(d/2+\nu)}}{\Gamma(d/2+\nu)}
    \|\Sigma z\|^{\nu}_{2}
    K_{-\nu}\bigl(\|\Sigma z\|_{2}\bigr),
    \qquad z\in\R^d,
\]
where $K_{-\nu}$ denotes the modified Bessel function of the second kind.
We also assume
$
    \nu>3/2.
$
This, in particular, implies that $\Psi$ is thrice continuously differentiable and $\Psi$ along with all its derivatives up to order $3$ are uniformly bounded on $\mathbb{R}^d$.

\item \textbf{Potential. }We assume $V$ is positive, thrice continuously differentiable and satisfies~\eqref{eq:V-growth-hessian}  for some $L < \infty$.
In particular, $\sup_x \Delta V(x)<\infty$.    
Moreover, for some constants $a>0$ and
$\beta_0,\beta_1 \ge 0$, the dissipativity assumption~\eqref{eq:dissipativity-assumption} holds.
Finally, assume the Langevin
contractivity condition used in \cite[Assumption~4]{BBG2025}\footnote{Note that the latter holds for instance if the potential is strongly convex outside of a ball, see \cite[Remark 5]{BBG2025}.}, so that
Kanagawa--Barp--Gretton--Mackey \cite[Theorem~3.5]{KanagawaBarpGrettonMackey2025} applies to
the kernel \eqref{eq:bilinear-matern-kernel}.  

\item \textbf{Initialization. }
Let \(P_0^N\) be the joint law of the initial configuration and
\(\mu_0^N=N^{-1}\sum_i\delta_{X_i^N(0)}\). Assume
\[
        \sup_N \frac1N\KL(P_0^N \| \pi^{\otimes N})<\infty,
\]
and there are positive constants $\alpha_{\rm init}, \beta_{\rm init}, c_{\rm init},C_{\rm init}<\infty$ such that
\begin{equation}\label{eq:init-w2-prob}
    \mathbb P\left(W_2(\mu_0^N,\mu_0)>c_{\rm init}N^{-\alpha_{\rm init}}\right)
    \le C_{\rm init}\,(\log N)^{-\beta_{\rm init}},
    \qquad N\ge1,
\end{equation}
where \(\mu_0\in\mathcal P_2(\mathbb R^d)\) is a deterministic
probability measure with density \(p_0\) satisfying
\[
        \int_{\mathbb R^d}\|x\|^2\,\mu_0(dx)<\infty,
        \qquad
        \KL(\mu_0 \| \pi)<\infty .
\]
\end{enumerate}
\end{assumption}

\subsection{Key estimates}
 Consider the continuous-time SVGD mean-field equation~\eqref{eq:mf-svgd}
where the velocity field is
\begin{equation}\label{eq:velocity}
    v_\mu(x)
    :=\int_{\R^d}\{-\widetilde k(x,y)\nabla V(y)+\nabla_2\widetilde k(x,y)\}\,\mu(dy).
\end{equation}

A crucial technical input in this section is a strengthening of the uniform time-averaged second-moment estimate of
\cite[Lemma~5]{BBG2025}. This estimate serves as a Lyapunov function and, in particular, also furnishes global well-posedness for the particle and mean-field SVGD equations. Note that such well-posedness results do not follow from existing works as the associated kernel is unbounded. 

\begin{lemma}[Well-posedness and uniform time-averaged second-moment bound]
\label{lem:BBG-lem5}
Let Assumption~\ref{ass:Wass} hold. Then, for every $N\ge2$ and every initial
configuration $(x^N_1(0),\ldots,x^N_N(0))\in(\mathbb R^d)^N$, the
$N$-particle SVGD system~\eqref{eq:svgd-continuous} associated with the kernel \eqref{eq:bilinear-matern-kernel}
has a unique global classical solution. Moreover, there exists
$C\in(0,\infty)$ such that, for all $T>0$,
\[
        \int_0^T \frac1N\sum_{i=1}^N \|x^N_i(t)\|^2\,dt
        \le
        CT+\frac{C}{N}\sum_{i=1}^N V(x^N_i(0)).
\]
Further, under the initialization in Assumption~\ref{ass:Wass},
\begin{equation}\label{eq:entvf}
        \sup_{N\ge2}
        \mathbb E\left[\frac1N\sum_{i=1}^N V(x^N_i(0))\right]
        <\infty .
\end{equation}
The corresponding mean-field equation \eqref{eq:mf-svgd}
is globally defined in the weak continuity-equation sense and satisfies, for all
$T>0$,
\[
        \int_0^T \int_{\mathbb R^d}\|x\|^2\,\mu_t(dx)\,dt
        \le
        CT+ C\int_{\mathbb R^d}V(x)\,\mu_0(dx) < \infty.
\]
\end{lemma}

\begin{proof}
We first consider the particle system. Local existence and uniqueness
follow from local Lipschitz continuity of the driving vector field. Indeed,
$V\in C^3$ with bounded Hessian, while $\Psi$ and its derivatives up to
the required order are bounded.

It remains to rule out explosion. Let
\[
        M_N(t):=\frac1N\sum_{i=1}^N \|x^N_i(t)\|^2 .
\]
For the bilinear--Matérn kernel,
\[
        \nabla_2\widetilde k(x,y)=x-\nabla\Psi(x-y).
\]
Using $\|\nabla V(y)\|\le L(1+\|y\|)$ and the boundedness of
$\Psi$ and $\nabla\Psi$, we obtain, for every probability measure
$\rho$ with finite second moment,
\[
        \|v_\rho(x)\|
        \le C_{V,\Psi}\int (1+\|x\|\|y\|)(1+\|y\|)\,\rho(dy)+C(1+\|x\|)
        \le C_{V,\Psi}(1+m_2(\rho))(1+\|x\|).
\]
where
\(
        m_2(\rho):=\int \|y\|^2\,\rho(dy),
\)
and $C_{V,\Psi}$ is a finite constant depending only on $V,\Psi$.
Hence, along the particle system,
\[
        \|\dot x^N_i(t)\|
        \le
        C_{V,\Psi}(1+M_N(t))(1+\|x^N_i(t)\|).
\]

Let
\[
        \tau^N_R:=\inf\Big\{t\ge0:\max_{1\le i\le N}\|x^N_i(t)\|\ge R\Big\}.
\]
On $t\le T\wedge\tau^N_R$, using $\dot{x}_i(t) = v_{\muNt}(x_i(t))$ and the above velocity bound,
\[
        \|x^N_i(t)\|
        \le
        \|x^N_i(0)\|
        +
        C_{V,\Psi}\int_0^t (1+M_N(s))(1+\|x^N_i(s)\|)\,ds .
\]
Thus, by Gronwall's inequality,
\begin{equation}\label{eq:GE}
        \sup_{0\le s\le T\wedge\tau^N_R}\|x^N_i(s)\|
        \le
        \Big(\|x^N_i(0)\|
        +
        1\Big)
        \exp\Big\{
        C_{V,\Psi}\int_0^{T\wedge\tau^N_R}(1+M_N(s))\,ds
        \Big\}.
\end{equation}
The Lyapunov calculation of \cite[Lemma~5]{BBG2025}, applied to the
localized solution on $[0,T\wedge\tau^N_R]$, gives
\begin{equation}\label{eq:ne}
        \int_0^{T\wedge\tau^N_R} M_N(s)\,ds
        \le
        CT+\frac{C}{N}\sum_{i=1}^N V(x^N_i(0)),
\end{equation}
with the constant $C$ above independent of $R$. Using this bound in \eqref{eq:GE}, we obtain
\[
\sup_{0\le s\le T\wedge\tau^N_R}\|x^N_i(s)\|
        \le
\Big(\|x^N_i(0)\|
        +
        1\Big)
        \exp\Big\{C_{V,\Psi}T + 
        C_{V,\Psi}\left( CT+\frac{C}{N}\sum_{i=1}^N V(x^N_i(0))\right)
        \Big\}.
\]
Choosing $R$
larger than this bound shows that $\tau^N_R>T$. Since $T<\infty$ was
arbitrary, the local solution cannot explode and is therefore global.
The displayed second-moment estimate follows by letting $R\to\infty$ in \eqref{eq:ne}.

Taking expectations and using the entropy variational formula, for
$\delta\in(0,1)$ sufficiently small,
\begin{equation*}
        \mathbb E\left[\frac1N\sum_{i=1}^N V(x^N_i(0))\right]
        \le
        \frac1{\delta N}
        \KL(P_0^N\|\pi^{\otimes N})
        +
        \frac1\delta
        \log\int_{\mathbb R^d} e^{\delta V(x)}\,\pi(dx),
\end{equation*}
which is uniformly bounded under Assumption~\ref{ass:Wass}.

The mean-field argument is very similar, but now carried out along characteristics. Standard local well-posedness for
the characteristic equation gives a local weak solution. Let
$X_t(z)$ denote the local characteristic starting from $z$, so that
$\mu_t=(X_t)_\#\mu_0$ as long as the solution is defined. Set
\[
        M(t):=\int_{\mathbb R^d}\|x\|^2\,\mu_t(dx).
\]
The same velocity bound gives
\[
        \|v_{\mu_t}(x)\|
        \le
        C(1+M(t))(1+\|x\|).
\]
Set
\[
        \tau^{(z)}_R:=\inf\{t\ge0:\|X_t(z)\|\ge R\}.
\]
Gronwall's inequality combined with the above velocity bound similarly yields
\begin{equation}\label{eq:charmf}
\sup_{0\le s\le T\wedge\tau^{(z)}_R}\|X_s(z)\|
        \le
\Big(\|z\|
        +
        1\Big)
        \exp\Big\{C_{V,\Psi}\int_{0}^{T \wedge \tau^{(z)}_R}\left( 1 +M(s)\right) \,ds
        \Big\}.
\end{equation}
The same Lyapunov calculation may now be
performed at the level of the mean-field equation. Let
\[
    E(t):=\int_{\R^d} V(x)\,\mu_t(dx).
\]
We claim that
\begin{equation}\label{eq:diffVjust}
E(t) - E(0)= \int_0^t\int_{\mathbb{R}^d} \nabla V(x)\cdot v_{\mu_s}(x) \, \mu_s(dx)\,ds, \quad t \in [0,T \wedge \sigma_K].
\end{equation}
To show this, we use the weak-formulation of the continuity equation, and provide a justification sketch of the use of the unbounded test function \(V\) by localization. Fix
\(T>0\) and \(K<\infty\), and set
\[
\sigma_K:=\inf\{t\ge0: M(t)\ge K\}.
\]
Let \(\chi_L\in C_c^\infty(\mathbb R^d)\) be a standard cutoff with
\(\chi_L=1\) on the Euclidean ball \(B_L\) of radius $L$ in $\mathbb{R}^d$, \(\chi_L=0\) outside \(B_{2L}\), and
\(\|\nabla\chi_L\|_{\infty}\le C/L\), and put \(V_L=\chi_L V\). Since \(V_L\in C_c^1(\mathbb R^d)\), it is admissible in the weak-formulation. On \([0,T\wedge\sigma_K]\), \(M(t)\le K\), and the growth bounds on \((V,\nabla V)\) and \(v_{\mu_t}\) imply that all cutoff-error terms are bounded by
\[
\int_0^{T\wedge\sigma_K} \int_{\|x\| \ge L}C_K(1+|x|^2)\,\mu_s(dx)\,ds,
\]
for some constant $C_K \in (0,\infty)$. Letting \(L\to\infty\) therefore
justifies the following computation with \(V\) in place of \(V_L\), for all
\(t \in [0,T \wedge \sigma_K]\).

Using the form of $v_{\mu_t}$ in~\eqref{eq:mf-svgd} and differentiating
\(E(t)\) using \eqref{eq:diffVjust}, one obtains for a.e. $t \in [0,T \wedge \sigma_K)$,
\begin{align}\label{eq:lb}
E'(t)
&=
-\left\|\int_{\R^d}\nabla V(x)\,\mu_t(dx)\right\|^2 
-\sum_{\ell,\ell'=1}^d
\left(
    \int_{\R^d} x_\ell\,\partial_{\ell'}V(x)\,\mu_t(dx)
\right)^2 \nonumber\\
&\quad +\int_{\R^d}\langle x,\nabla V(x)\rangle\,\mu_t(dx)
-\iint_{\R^d\times\R^d}
    \langle \nabla V(x),\nabla\Psi(x-y)\rangle
    \,\mu_t(dx)\mu_t(dy) \nonumber\\
&\quad
-\iint_{\R^d\times\R^d}
    \Psi(x-y)\langle \nabla V(x),\nabla V(y)\rangle
    \,\mu_t(dx)\mu_t(dy).
\end{align}

The first term is nonpositive, and the final term is nonpositive by the positive definiteness of \(\Psi\). Also, note that
\[
\sum_{\ell,\ell'=1}^d
\left(
    \int x_\ell\,\partial_{\ell'}V(x)\,\mu_t(dx)
\right)^2
\ge
\frac1d
\left(
    \int \langle x,\nabla V(x)\rangle\,\mu_t(dx)
\right)^2 .
\]
Moreover, since \(\nabla\Psi\) is bounded and \(\|\nabla V(x)\|\le L(1+\|x\|)\),
\[
        \left|\int_{\R^d}\langle x,\nabla V(x)\rangle\,\mu_t(dx) + 
        \iint \langle \nabla V(x),\nabla\Psi(x-y)\rangle\,
        \mu_t(dx)\mu_t(dy)
        \right|
        \le C(1+M(t)).
\]
Using these observations in \eqref{eq:lb}, we obtain finite positive constants $C_1,C_2$ such that for a.e. $t \in [0,T \wedge \sigma_K)$,
\[
        E'(t) \le C_1 + C_2M(t) - \frac{a^2}{4d}M(t)^2,
\]
where $a>0$ appears in the dissipativity assumption \eqref{eq:dissipativity-assumption}.
Integrating over \([0,T \wedge \sigma_K]\) and using \(V\ge0\), we obtain
\begin{equation}\label{secb}
        \int_0^{T \wedge \sigma_K} M(t)\,dt
        \le
        CT+C \int_{\mathbb{R}^d} V(x)\, \mu_0(dx).
\end{equation}
Note that, again using the variational formula for relative entropy: for $\delta \in (0,1)$ sufficiently small,
\begin{equation*}
        \int_{\mathbb{R}^d} V(x)\, \mu_0(dx)
        \le
        \frac1\delta\KL(\mu_0\|\pi)
        +
        \frac1\delta
        \log\int_{\mathbb R^d} e^{\delta V(x)}\,\pi(dx),
\end{equation*}
and recalling from Assumption~\ref{ass:Wass} that $\KL(\mu_0 \| \pi)<\infty$, we obtain \(\int_{\mathbb{R}^d} V(x)\, \mu_0(dx)<\infty\). As the bound in \eqref{secb} does not depend on $K$, this gives the stated uniform time-averaged second-moment estimate for $\mu_t$ on taking $K \to \infty$.
Using this bound in~\eqref{eq:charmf} gives
\[
\sup_{0\le s\le T\wedge\tau^{(z)}_R}\|X_s(z)\|
        \le
\Big(\|z\|
        +
        1\Big)
        \exp\Big\{C_{V,\Psi}(1+C)T +C_{V,\Psi} C\int_{\mathbb{R}^d} V(x)\, \mu_0(dx))
        \Big\}.
\]
As the bound above does not depend on $R$, global well-posedness follows exactly as in the particle ODE case.
\end{proof}

\begin{remark}\label{rem:geometry}
The estimate in Lemma~\ref{lem:BBG-lem5}
can be viewed as a manifestation of the geometry selected by the
bilinear kernel.  The non-normalized linear component enriches the Stein
RKHS by affine directions, and these directions interact directly with
the quadratic geometry underlying \(W_2\).  Analytically, this appears in
the Lyapunov bound above through the coercive quadratic term involving $\int \langle x,\nabla V(x)\rangle\,\mu_t(dx)$ and its empirical analogue.
Thus the improved second-moment control is not merely a technical
estimate; it reflects the compatibility between the bilinear Stein
geometry and Wasserstein--\(2\) moment structure.
\end{remark}

We now prove the finite-time stability estimate needed for furnishing finite-time propagation-of-chaos rates. The proof extends the Dobrushin estimate in Lu--Lu--Nolen~\cite[Theorem~2.7]{LuLuNolen2019} for solutions to the mean-field SVGD equation to kernels with an unbounded bilinear part. 

\begin{lemma}[Finite-time $W_2$ stability for the bilinear kernel]\label{lem:finite-time-w2-stability}
Under Assumption~\ref{ass:Wass}, let $(\mu_t)_{0\le t\le T}$ and $(\nu_t)_{0\le t\le T}$ be two (weak) solutions of~\eqref{eq:mf-svgd}. 
Also, suppose that for every $T \in (0,\infty)$, there is $M_T<\infty$ such that, writing $ m_2(\rho):=\int \norm{x}^2\,\rho(dx)$,
    \begin{equation}\label{eq:int-moment-bound}
        \int_0^T\{1+m_2(\mu_t)+m_2(\nu_t)\}\,dt\le M_T.
    \end{equation} Then there exists a finite constant $C>0$, depending only on the constants in the bounds of Assumption~\ref{ass:Wass}, such that for any $T>0$,
\begin{equation}\label{eq:w2-stability}
    \sup_{0\le t\le T}W_2(\mu_t,\nu_t)
    \le
    \exp\{C M_T\}\,W_2(\mu_0,\nu_0).
\end{equation}
\end{lemma}

\begin{proof}
In this proof, $\|\cdot\|$ will denote the Frobenius norm of the associated vector of matrix.
Let $(X_0,Y_0)$ be an optimal coupling of $(\mu_0,\nu_0)$, and let $(X_t,Y_t)$ be the characteristic coupling defined by
\[
    \dot X_t=v_{\mu_t}(X_t),
    \qquad
    \dot Y_t=v_{\nu_t}(Y_t).
\]
Set
\[
    D_t:=\E\norm{X_t-Y_t}^2.
\]
Then $W_2^2(\mu_t,\nu_t)\le D_t$, and
\begin{equation}\label{eq:Dprime}
    \frac{d}{dt}D_t
    =2\E\ip{X_t-Y_t}{v_{\mu_t}(X_t)-v_{\nu_t}(Y_t)}.
\end{equation}
We will show that
\begin{equation}\label{eq:main-lip-estimate}
    \frac{d}{dt}D_t
    \le C\{1+m_2(\mu_t)+m_2(\nu_t)\}D_t.
\end{equation}
The conclusion then follows from Gr\"onwall's inequality and \eqref{eq:int-moment-bound}.

Write the velocity as the sum of three terms:
\begin{align*}
    v_\mu(x)
    &=v_\mu^{\mathrm b}(x)+v_\mu^0+v_\mu^\Psi(x),
\end{align*}
where
\begin{align*}
    v_\mu^{\mathrm b}(x)
    &:=x-\int \ip{x}{z}\nabla V(z)\,\mu(dz),\\    v_\mu^0
    &:=-\int \nabla V(z)\,\mu(dz),\\    v_\mu^\Psi(x)
    &:=-\int \Psi(x-z)\nabla V(z)\,\mu(dz)
      -\int \nabla\Psi(x-z)\,\mu(dz).
\end{align*}
For the bilinear part, define
\begin{equation}\label{eq:Bmu}
    B_\mu:=\int \nabla V(z)z^\top\,\mu(dz).
\end{equation}
Then
\begin{equation*}\label{eq:bilinear-linear}
    v_\mu^{\mathrm b}(x)=(I-B_\mu)x.
\end{equation*}
By \eqref{eq:V-growth-hessian},
\begin{equation}\label{eq:Bmu-bound}
    \norm{B_\mu}_{\mathrm{op}}
    \le \int \norm{\nabla V(z)}\norm{z}\,\mu(dz)
    \le C\{1+m_2(\mu)\}.
\end{equation}
Consequently,
\begin{equation}\label{eq:bilinear-spatial-lip}
    \norm{v_\mu^{\mathrm b}(x)-v_\mu^{\mathrm b}(y)}
    \le C\{1+m_2(\mu)\}\norm{x-y}.
\end{equation}
The translation-invariant part satisfies, using boundedness of $\nabla\Psi$ and $\nabla^2\Psi$ together with \eqref{eq:V-growth-hessian},
\begin{align}\label{eq:Psi-spatial-lip}
    \norm{v_\mu^\Psi(x)-v_\mu^\Psi(y)}
    &\le C\left(1+\int \norm{\nabla V(z)}\,\mu(dz)\right)\norm{x-y} \notag\\
    &\le C\{1+m_2(\mu)^{1/2}\}\norm{x-y}.
\end{align}
Combining \eqref{eq:bilinear-spatial-lip} and \eqref{eq:Psi-spatial-lip},
\begin{equation}\label{eq:spatial-lip}
    \norm{v_\mu(x)-v_\mu(y)}
    \le C\{1+m_2(\mu)\}\norm{x-y}.
\end{equation}
This controls the spatial part of \eqref{eq:Dprime}:
\begin{equation}\label{eq:spatial-part}
    2\E\ip{X_t-Y_t}{v_{\mu_t}(X_t)-v_{\mu_t}(Y_t)}
    \le C\{1+m_2(\mu_t)\}D_t.
\end{equation}

It remains to control the measure-dependence term
\[
    v_{\mu_t}(Y_t)-v_{\nu_t}(Y_t).
\]
Let $(Z,W)$ be any coupling of $(\mu,\nu)$.  First, since $\nabla^2V$ is bounded in operator norm,
\begin{equation*}
    \norm{\int\nabla V\,d\mu-\int\nabla V\,d\nu}
    \le L\,\E\norm{Z-W}.
\end{equation*}
Taking the infimum over couplings yields
\begin{equation}\label{eq:gradV-measure}
    \norm{\int\nabla V\,d\mu-\int\nabla V\,d\nu}
    \le L\,W_2(\mu,\nu).
 \end{equation} 
For the bilinear coefficient, define $F(z):=\nabla V(z)z^\top$.  Then
\[
    B_\mu-B_\nu=\E[F(Z)-F(W)].
\]
Observe that \(F\) is matrix-valued and, for \(h\in\mathbb R^d\),
\[
    DF(z)[h]
    =
    \nabla^2 V(z)h\,z^\top+\nabla V(z)h^\top .
\]
Hence, by \eqref{eq:V-growth-hessian},
\[
    \|DF(z)[h]\|
    \le
    \|\nabla^2V(z)\|_{\mathrm{op}}\|h\|\|z\|
    +
    \|\nabla V(z)\|\|h\|
    \le
    C(1+\|z\|)\|h\|.
\]
Taking the supremum over \(\|h\|=1\), we get
\[
    \|DF(z)\|\le C(1+\|z\|).
\]
Consequently, by the mean-value theorem,
\[
    \|F(z)-F(w)\|
    \le
    C(1+\|z\|+\|w\|)\|z-w\|.
\]
In particular, using Cauchy--Schwarz inequality,
\[
\begin{aligned}
\|B_\mu-B_\nu\| 
&=
\|\mathbb E[F(Z)-F(W)]\|  \\
&\le
C\mathbb E\big[(1+\|Z\|+\|W\|)\|Z-W\|\big] \\
&\le
C\bigl(1+m_2(\mu)^{1/2}+m_2(\nu)^{1/2}\bigr)
        \bigl(\mathbb E\|Z-W\|^2\bigr)^{1/2}.
\end{aligned}
\]
Taking the infimum over couplings yields
\[
    \|B_\mu-B_\nu\|_{\mathrm{op}} \le \|B_\mu-B_\nu\|
    \le
    C\bigl(1+m_2(\mu)^{1/2}+m_2(\nu)^{1/2}\bigr)
    W_2(\mu,\nu).
\]
Thus
\begin{equation}\label{eq:measure-bilinear}
    \norm{v_\mu^{\mathrm b}(y)-v_\nu^{\mathrm b}(y)}
    \le C\{1+m_2(\mu)^{1/2}+m_2(\nu)^{1/2}\}\norm{y}\,W_2(\mu,\nu).
\end{equation}
For the $\Psi$-part, define
\[
    G_y(z):=\Psi(y-z)\nabla V(z) + \nabla\Psi(y-z).
\]
Then
\[
    DG_y(z)=-\nabla\Psi(y-z)\otimes\nabla V(z)+\Psi(y-z)\nabla^2V(z) + \nabla^2\Psi(y-z),
\]
and therefore, using boundedness of $\nabla\Psi$ and $\nabla^2\Psi$, and \eqref{eq:V-growth-hessian},
\[
    \norm{DG_y(z)}\le C(1+\norm{z}).
\]
Thus, arguing exactly as how we obtained \eqref{eq:measure-bilinear},
\begin{equation}\label{eq:Psi-measure}
    \norm{v_\mu^\Psi(y)-v_\nu^\Psi(y)}
    \le C\{1+m_2(\mu)^{1/2}+m_2(\nu)^{1/2}\}W_2(\mu,\nu),
\end{equation}
uniformly in $y$.

Apply \eqref{eq:gradV-measure}, \eqref{eq:measure-bilinear}, and \eqref{eq:Psi-measure} with $\mu=\mu_t$ and $\nu=\nu_t$.  Since $W_2(\mu_t,\nu_t)\le D_t^{1/2}$ and $\E\norm{Y_t}^2=m_2(\nu_t)$, Cauchy--Schwarz gives
\begin{align}\label{eq:measure-part}
    &2\E\ip{X_t-Y_t}{v_{\mu_t}(Y_t)-v_{\nu_t}(Y_t)}\le C\{1+m_2(\mu_t)+m_2(\nu_t)\}D_t.
\end{align}
Combining \eqref{eq:spatial-part} and \eqref{eq:measure-part} proves \eqref{eq:main-lip-estimate}.  Hence
\[
    D_t\le D_0\exp\left\{C\int_0^t\{1+m_2(\mu_s)+m_2(\nu_s)\}\,ds\right\}.
\]
Using \eqref{eq:int-moment-bound},
\[
    D_t^{1/2}\le D_0^{1/2}\exp\{CM_T/2\}.
\]
Taking the infimum over initial couplings gives \eqref{eq:w2-stability}.  
\end{proof}

By \cite[Theorem~3.5]{KanagawaBarpGrettonMackey2025}, there exists a constant $C_{\rm KGM}<\infty$ such that, with
$\mathfrak r=\mathfrak r(d,\nu)\in(0,1)$ given by
\begin{equation}\label{eq:rdef}
\mathfrak r(d,\nu) :=
\frac{d}{3(4d+1)}
\cdot
\frac{1}{
\frac{3d}{2}
+
\frac{11}{6}
+
\left(\frac83+\frac1d\right)\nu
},
\end{equation}
for every
probability measure $\rho$ with finite second moment,
\begin{equation}\label{eq:kanagawa-w2-from-ksd}
        W_2(\rho,\pi)
        \le
        C_{\rm KGM}
        \bigl(1\vee \KSD(\rho\|\pi)^{1-\mathfrak r}\bigr)
        \KSD(\rho\|\pi)^{\mathfrak r}.
\end{equation}
Observe that $\mathfrak r(d,\nu) \approx  \frac{1}{18 d}$ for large $d$. 

We will use the following two large-time estimates.  They are the Wasserstein analogues of the
large-time KSD estimates used in Theorem~\ref{prop:uniform-time-avg-poc}.

\begin{lemma}[Time-averaged convergence to $\pi$ in $W_2$]
\label{lem:large-time-w2-to-pi}
Under Assumption~\ref{ass:Wass}, there is a constant $C<\infty$ such that, for all sufficiently large
$N$ and all $T\ge 1$,
\begin{align}
        \E W_2(\overline\mu_T^N,\pi)
        &\le
        C\left(\frac1{\sqrt T}+\frac1{\sqrt N}\right)^{\mathfrak r},
        \label{eq:particle-avg-w2-to-pi} \\
        W_2(\overline\mu_T,\pi)
        &\le
        C T^{-\mathfrak r/2}.
        \label{eq:mean-field-avg-w2-to-pi}
\end{align}
\end{lemma}

\begin{proof}
For the particle system, the proof of \cite[Theorem~4]{BBG2025}
shows, under Assumption~3 and the kernel \eqref{eq:bilinear-matern-kernel}, that the same
entropy calculation as in \cite[Theorem~1]{BBG2025} gives
\begin{equation}\label{eq:KSD-large-time-W2-section}
        \E\KSD(\overline\mu_T^N\|\pi)
        \le
        C\left(\frac1{\sqrt T}+\frac1{\sqrt N}\right),
        \qquad T\ge 1.
\end{equation}
To see why the argument of \cite[Theorem 4]{BBG2025} applies in the present
unbounded-kernel setting, recall $C^*(\cdot)$ from \eqref{eq:Cstardef} with the kernel $\widetilde k =1+\langle u,v\rangle+\Psi(u-v)$.
The Mat\'ern
part contributes only bounded diagonal terms, while the bilinear part gives
\(\nabla_2\tilde k(z,z)=z+O(1)\) and
\(\tilde k(z,z)=O(1+\|z\|^2)\). Hence, using the linear growth of
\(\nabla V\) and the boundedness of \(\Delta V\),
\[
        |C^*(z)|\le C(1+\|z\|^2).
\]
Thus the entropy calculation of
\cite[Theorem 4]{BBG2025} gives
\[
 \frac1T\int_0^T
        \mathbb E\,\KSD^2(\mu_t^N\|\pi)\,dt
 \le
 \frac{1}{NT}\KL(P_0^N\|\pi^{\otimes N})
 +
 \frac{C}{NT}\int_0^T
        \mathbb E\left[
        1+\frac1N\sum_{i=1}^N \|x_i(t)\|^2
        \right]dt .
\]
Lemma~\ref{lem:BBG-lem5}, the $\KL$ bound in Assumption~\ref{ass:Wass} and Jensen's inequality now give~\eqref{eq:KSD-large-time-W2-section}.

Now, we apply \eqref{eq:kanagawa-w2-from-ksd} with $\rho=\overline\mu_T^N$.  Since the right side of
\eqref{eq:KSD-large-time-W2-section} is at most one for all sufficiently large $N$ and
$T\ge T_N$, where $T_N \to \infty$ in the cutoff argument below, and since $\mathfrak r\in(0,1)$, Jensen's inequality gives
\[
        \E W_2(\overline\mu_T^N,\pi)
        \le
        C\,\E\Big[\KSD(\overline\mu_T^N\|\pi)^{\mathfrak r}
              +\KSD(\overline\mu_T^N\|\pi)\Big]
        \le
        C\bigl(\E\KSD(\overline\mu_T^N\|\pi)\bigr)^{\mathfrak r}.
\]
Together with \eqref{eq:KSD-large-time-W2-section}, this proves
\eqref{eq:particle-avg-w2-to-pi}.

For the mean-field flow, Lemma~\ref{ass:mean-field-to-pi} gives
\[
        \KSD(\overline\mu_T\|\pi)
        \le C T^{-1/2}.
\]
Applying \eqref{eq:kanagawa-w2-from-ksd} to $\rho=\overline\mu_T$ gives
\eqref{eq:mean-field-avg-w2-to-pi}, after increasing $C$ if necessary.
\end{proof}

\subsection{Uniform-in-time propagation-of-chaos in Wasserstein-2 distance}

\begin{theorem}[Uniform-in-averaging-horizon POC in $W_2$ in probability]
\label{prop:uniform-T-W2}
Let Assumption~\ref{ass:Wass} hold. Recall the exponent $\mathfrak r=\mathfrak r(d,\nu)$ in \eqref{eq:rdef}, and let
$
    \gamma_*:=\frac{\mathfrak r}{4} \wedge \beta_{\rm init},
$
where $\beta_{\rm init}$ is the logarithmic decay rate of the probability associated with the initialization in Assumption~\ref{ass:Wass}. Then there is a $N_0 \in \mathbb{N}$ and $C \in (0,\infty)$ such that,
for all $N\ge N_0$,
\begin{equation}\label{eq:uniform-T-W2-poc-prob}
    \sup_{T>0}\,
    \mathbb P\left(
        W_2(\overline\mu_T^N,\overline\mu_T)
        > (\log N)^{-\mathfrak r / 4}
    \right)
    \le C\,(\log N)^{-\gamma_*}.
\end{equation}
Moreover,
\[
W_2(\overline\mu_T^N,\overline\mu_T)
=
O_{\mathbb P}\bigl((\log N)^{-\mathfrak r/2}\bigr)
\]
uniformly over deterministic averaging horizons \(T>0\), in the sense that
\begin{equation}\label{eq:conp}
\lim_{M\to\infty}\limsup_{N\to\infty}
\sup_{T>0}
\mathbb P\left(
W_2(\overline\mu_T^N,\overline\mu_T)>M(\log N)^{-\mathfrak r/2}
\right)
=0.
\end{equation}
\end{theorem}

\begin{proof}
Define
\[
    m_2^N(t):=\frac1N\sum_{i=1}^N \|x_i(t)\|^2,
    \qquad
    \mathcal H_N:= \frac1N\sum_{i=1}^N V(x_i(0)).
\]
By Lemma~\ref{lem:BBG-lem5}, there is a constant $C<\infty$ such that $\int_0^T m_2^N(t)\, dt \le CT + C\mathcal H_N$ for all $T > 0$, $N \ge 2$, and
\[
    \sup_N \E\mathcal H_N\le C .
\]
Let
\[
    R_N:=(\log N)^{\mathfrak r/4},
    \qquad
    \mathcal G_N:=\{\mathcal H_N\le R_N\}.
\]
Then Markov's inequality gives
\begin{equation}\label{eq:GN-complement}
    \mathbb P(\mathcal G_N^c)
    \le C R_N^{-1}
    = C(\log N)^{-\mathfrak r/4}.
\end{equation}
The mean-field part of Lemma~\ref{lem:BBG-lem5} gives a deterministic
constant $C_{\rm mf}<\infty$ such that for all $S>0$,
\begin{equation}\label{eq:mf-avg-moment-for-prop}
    \int_0^S m_2(\mut)\,dt
    \le C_{\rm mf}(1+ S),
    \qquad
    m_2(\rho):=\int \|x\|^2\rho(dx).
\end{equation}
Set
\[
    T_N:=\eta \log N,
\]
where $\eta>0$ is chosen below.

First note that, on $\mathcal G_N$, for any $T\in (0,T_N]$,
\begin{align*}
    \int_0^T \{1+m_2^N(t)+m_2(\mut)\}\,dt
    &\le T\{1+C +C_{\rm mf}\} + CR_N + C_{\rm mf}\\
    &\le T_N\{1+C +C_{\rm mf}\} + CR_N + C_{\rm mf}\\
    &\le C\eta\log N + C(\log N)^{\mathfrak r/4} + C_{\rm mf}.
\end{align*}
Applying Lemma~\ref{lem:finite-time-w2-stability} with one solution
initialized from the empirical initial law and the other from $\mu_0$, and
then using convexity of $W_2^2$ under time averaging, gives on
$\mathcal G_N$,
\begin{equation}\label{eq:short-W2-on-good-event}
    \sup_{0 < T\le T_N}\,W_2(\overline\mu_T^N,\overline\mu_T)
    \le
    N^{C\eta + o(1)} W_2(\mu_0^N,\mu_0).
\end{equation}
Choose $\eta>0$ so small that $C\eta<\alpha_{\rm init}/2$, where $\alpha_{\rm init}$ is the
initial empirical rate in \eqref{eq:init-w2-prob}. Then
\eqref{eq:init-w2-prob}, \eqref{eq:GN-complement}, and
\eqref{eq:short-W2-on-good-event} imply that, for all sufficiently large
$N$,
\begin{equation}\label{eq:short-prob}
    \mathbb P\left(\sup_{0 < T\le T_N}\,
        W_2(\overline\mu_T^N,\overline\mu_T)
        > N^{-\alpha_{\rm init}/2}
    \right)
    \le C(\log N)^{-\gamma_*}.
\end{equation}
It remains to fix $T\ge T_N$. By the triangle inequality,
\[
    W_2(\overline\mu_T^N,\overline\mu_T)
    \le W_2(\overline\mu_T^N,\pi)+W_2(\overline\mu_T,
    \pi).
\]
The deterministic mean-field bound in Lemma~\ref{lem:large-time-w2-to-pi}
gives
\[
    W_2(\overline\mu_T,\pi)
    \le C T^{-\mathfrak r/2}
    \le C T_N^{-\mathfrak r/2}.
\]
For the particle term, Lemma~\ref{lem:large-time-w2-to-pi} and Markov's
inequality yield, for sufficiently large $N$ and every $T \ge T_N$,
\begin{align*}
    \mathbb P\left(W_2(\overline\mu_T^N,\pi)>(\log N)^{-\mathfrak r/4}\right)
   &\le
    \frac{C}{(\log N)^{-\mathfrak r/4}}\left(\frac1{\sqrt T}+\frac1{\sqrt N}\right)^{\mathfrak r}\\
    &\le
   C\eta^{-\mathfrak r/2}(\log N)^{\mathfrak r/4 - \mathfrak r/2} =  C\eta^{-\mathfrak r/2}(\log N)^{-\mathfrak r/4}.
\end{align*}
Thus,
we obtain
\begin{equation}\label{eq:long-prob-fixed-T}
    \sup_{T\ge T_N}\mathbb P\left(
        W_2(\overline\mu_T^N,\overline\mu_T)
        > (\log N)^{-\mathfrak r/4}
    \right)
    \le C\eta^{-\mathfrak r/2}(\log N)^{-\mathfrak r/4}.
\end{equation}
Combining \eqref{eq:short-prob} and
\eqref{eq:long-prob-fixed-T} gives \eqref{eq:uniform-T-W2-poc-prob}.

\eqref{eq:conp} follows similarly upon replacing $R_N$ by $M$ in the definition of the event $\mathcal G_N$, estimating $\sup_{T\ge T_N}\mathbb P\left(
        W_2(\overline\mu_T^N,\overline\mu_T)
        > M(\log N)^{-\mathfrak r/2}
    \right)$ similarly as in \eqref{eq:long-prob-fixed-T}, and taking the limits in the order prescribed in \eqref{eq:conp}.
\end{proof}

\begin{remark}[On the stronger maximal-in-$T$ formulation]
The estimates used above justify the uniform-in-time
probability bound \eqref{eq:uniform-T-W2-poc-prob}. To replace its left-hand
side by
\[
    \mathbb P\left(
        \sup_{T>0}W_2(\overline\mu_T^N,\overline\mu_T)
        > (\log N)^{-\mathfrak r/4}
    \right),
\]
one would additionally need a maximal version of the large-time estimate in
Lemma~\ref{lem:large-time-w2-to-pi}, for example a bound on
\(\mathbb E\sup_{T\ge T_N}W_2(\overline\mu_T^N,\pi)\). The pointwise
expectation estimate in Lemma~\ref{lem:large-time-w2-to-pi} alone does not
imply such a maximal bound. Similar strengthening remarks hold for the uniform-in-time propagation-of-chaos results in KSD and Wasserstein--1 distances.
\end{remark}
The following theorem gives the $k$-particle propagation-of-chaos for fixed $k \ge 1$.
\begin{theorem}[Time-averaged \(k\)-particle chaos in \(W_2\)]\label{kpartW2}
Let Assumption~\ref{ass:Wass} hold and suppose the joint
initial law of the particles is exchangeable. For fixed \(k\ge1\), define
the random \(k\)-fold empirical occupation measure
\[
        \widehat\mu_T^{N,k}
        :=
        \frac1T\int_0^T (\mu_t^N)^{\odot k}\,dt,
        \qquad
        \overline{\mu^{\otimes k}}_T
        :=
        \frac1T\int_0^T \mu_t^{\otimes k}\,dt,
\]
where
\[
(\mu_t^N)^{\odot k}
:=
\frac{1}{(N)_k}
\sum_{\substack{1\le i_1,\ldots,i_k\le N\\
i_1,\ldots,i_k\ {\rm distinct}}}
\delta_{\bigl(x_{i_1}^N(t),\ldots,x_{i_k}^N(t)\bigr)},
\qquad
(N)_k:=N(N-1)\cdots(N-k+1).
\]
Let \(W_2^{(k)}\) be the Wasserstein distance on \((\mathbb R^d)^k\)
associated with the quadratic product cost
\(
        d_k(x,y)^2:=\sum_{\ell=1}^k \|x_\ell-y_\ell\|^2 .
\)
Then, with $\mathfrak r=\mathfrak r(d,\nu)$ as in~\eqref{eq:rdef}, there is \(C_k<\infty\)
such that, for all sufficiently large \(N\),
\[
\sup_{T>0}
\mathbb P\left(
W_2^{(k)}
\left(
\widehat\mu_T^{N,k},
\overline{\mu^{\otimes k}}_T
\right)
>
(\log N)^{-r/4}
\right)
\le
C_k(\log N)^{-\min\{r/4,\beta_{\rm init}\}} .
\]
\end{theorem}

\begin{proof}
The only additional ingredient beyond the proof of
Theorem~\ref{prop:uniform-T-W2} is the standard comparison between
sampling with and without replacement. For every deterministic
\(\nu\in\mathcal P_2(\mathbb R^d)\), similarly as in~\eqref{3.7},
\begin{equation}\label{4.28}
W_2^{(k)}
\bigl((\mu_t^N)^{\odot k},\nu^{\otimes k}\bigr)^2
\le
2k W_2(\mu_t^N,\nu)^2
+
\frac{2k^2(k-1)}{N}\int_{\mathbb R^d}\|x\|^2\,\mu_t^N(dx).
\end{equation}
By convexity of \(W_2^2\),
\begin{align}\label{4.29}
W_2^{(k)}
\left(
\widehat\mu_T^{N,k},
\overline{\mu^{\otimes k}}_T
\right)^2
&\le
\frac1T\int_0^T
W_2^{(k)}
\bigl(
(\mu_t^N)^{\odot k},\mu_t^{\otimes k}
\bigr)^2\,dt .
\end{align}
Combining \eqref{4.28} and \eqref{4.29}, with \(\nu=\mu_t\), gives
\begin{align}\label{4.30}
W_2^{(k)}
\left(
\widehat\mu_T^{N,k},
\overline{\mu^{\otimes k}}_T
\right)^2
\le
\frac{2k}{T}\int_0^T W_2(\mu_t^N,\mu_t)^2\,dt
+
\frac{2k^2(k-1)}{NT}\int_0^T
\int_{\mathbb R^d}\|x\|^2\,\mu_t^N(dx)\,dt .
\end{align}
We now repeat the cutoff argument from
Theorem~\ref{prop:uniform-T-W2}. Let
\(
        T_N:=\delta\log N,
\)
with \(\delta>0\) chosen sufficiently small. For \(0<T\le T_N\), the
finite-time \(W_2\) stability estimate used in the proof of
Theorem~\ref{prop:uniform-T-W2}, together with the initial
probability bound in Assumption~\ref{ass:Wass}, gives
\begin{equation}\label{pck1}
\mathbb P\left(
\frac1T\int_0^T W_2(\mu_t^N,\mu_t)^2\,dt
>
(\log N)^{-r/2}
\right)
\le
C(\log N)^{-\beta_{\rm init}}
\end{equation}
for all sufficiently large \(N\). 

It remains to check that the second term in~\eqref{4.30} is negligible
in the short-time regime.  Write
\[
        M_N(t):=\int \|x\|^2\,\mu_t^N(dx)
        =\frac1N\sum_{i=1}^N\|x_i^N(t)\|^2,
        \qquad
        A_N:=\frac1N\sum_{i=1}^N V(x_i^N(0)).
\]
Combining the Gronwall estimate~\eqref{eq:GE} with the localized
Lyapunov bound~\eqref{eq:ne}, and then letting \(R\to\infty\), gives,
for every fixed \(T<\infty\),
\[
        \sup_{0\le s\le T}M_N(s)
        \le
        C(1+A_N)\exp\{CT+CA_N\}.
\]
Here we also used the quadratic coercivity of \(V\), which follows from
the dissipativity assumption, to bound
\(N^{-1}\sum_i\|X_i^N(0)\|^2\) by \(C(1+A_N)\). Define the event $\mathcal{A}_N := \{A_N \le \delta\log N\}$. By~\eqref{eq:entvf}, 
\begin{equation}\label{miss1}
    \mathbb{P}\left(\mathcal{A}_N^c\right) \le C(\delta\log N)^{-1}.
\end{equation}
On $\mathcal{A}_N$, for
\(T\le T_N=\delta\log N\),
\begin{equation}\label{pck2}
\frac{1}{NT}\int_0^T M_N(s)\,ds
\le
\frac1N\sup_{0\le s\le T_N}M_N(s) = O(N^{-1 + 3C\delta}).
\end{equation}
Choosing \(\delta>0\) sufficiently small makes this term negligible
relative to the logarithmic threshold used in the short-time estimate.

For \(T\ge T_N\), use the triangle inequality through \(\pi^{\otimes k}\)
and \eqref{4.28}, now with \(\nu=\pi\), to obtain
\begin{align}\label{4.31}
W_2^{(k)}
\left(
\widehat\mu_T^{N,k},
\overline{\mu^{\otimes k}}_T
\right)^2
&\le
\frac{4k}{T}\int_0^T W_2(\mu_t^N,\pi)^2\,dt
+
\frac{4k}{T}\int_0^T W_2(\mu_t,\pi)^2\,dt \notag \\
&\quad+
\frac{4k^2(k-1)}{NT}\int_0^T
\int_{\mathbb R^d}\|x\|^2\,\mu_t^N(dx)\,dt .
\end{align}
The same entropy-dissipation/KSD-to-\(W_2\) argument used in
Lemma~\ref{lem:large-time-w2-to-pi}, using the proof of~\cite[Theorem~4]{BBG2025} and \eqref{eq:kanagawa-w2-from-ksd}, gives
\[
\mathbb E\left[
\frac1T\int_0^T W_2(\mu_t^N,\pi)^2\,dt
\right]
\le
C\left(T^{-\mathfrak r}+N^{-\mathfrak r}\right),
\qquad
\frac1T\int_0^T W_2(\mu_t,\pi)^2\,dt
\le
CT^{-\mathfrak r}.
\]
The remaining second-moment term is controlled by
Lemma~\ref{lem:BBG-lem5}. 

Therefore, there exists a constant $C_k \in (0,\infty)$ such that for
\(T\ge T_N\),
\begin{equation}\label{pck3}
\mathbb P\left(W_2^{(k)}
\left(
\widehat\mu_T^{N,k},
\overline{\mu^{\otimes k}}_T
\right) >
(\log N)^{-r/4}
\right)
\le C_k(\log N)^{-r/4}.
\end{equation}
In particular, by Markov's inequality, \eqref{pck1}, \eqref{miss1}, \eqref{pck2} and \eqref{pck3}, and the same
probability bookkeeping as in Theorem~\ref{prop:uniform-T-W2},
\[
\sup_{T>0}
\mathbb P\left(
W_2^{(k)}
\left(
\widehat\mu_T^{N,k},
\overline{\mu^{\otimes k}}_T
\right)
>
(\log N)^{-r/4}
\right)
\le
C_k(\log N)^{-\min\{r/4,\beta_{\rm init}\}} .
\]
This proves the theorem.
\end{proof}

\section{Matrix-kernel SVGD, Conjugacy and Polynomial rates}\label{sec:conjgauss}

The preceding sections show that, for broad distributional metrics, uniform comparisons between the finite-particle and mean-field SVGD dynamics are possible after combining short-time propagation of chaos with long-time convergence-to-equilibrium estimates.  This cutoff mechanism yields uniform-in-averaging-horizon bounds in KSD, \(W_1\), and \(W_2\), but with logarithmic or iterated-logarithmic rates in \(N\).  By contrast, in the Gaussian case with bilinear kernels, the SVGD dynamics close on finitely many moments and this closed system can be controlled uniformly in physical time with parametric, polynomially decaying rates~\cite{liu2023towards}.  It is therefore natural to ask whether these polynomial, uniform-in-physical-time estimates are special to the Gaussian setting, or whether they can be established for a broader class of non-Gaussian targets by modifying the target--kernel pair in a compatible way. Towards answering this, in Section~\ref{sec:mvSVGD}, we first consider a natural generalization of the SVGD flow to matrix-valued kernels~\cite{wang2019stein}. This turns out to be useful in establishing propagation-of-chaos for more general targets.

In Section~\ref{sec:conj}, we illustrate a \emph{conjugacy phenomenon}, which gives a systematic way of constructing a family of target-kernel pairs from a reference (base) pair via deterministic orientation-preserving diffeomorphisms, such that each constructed pair also satisfies an SVGD flow with an explicitly transformed velocity and conjugate Stein features. Consequently, establishing polynomial rate uniform-in-time propagation-of-chaos estimates in the forthcoming feature metric \eqref{eq:feature-metric} for any pair in this family boils down to analyzing just the base pair in its intrinsic feature metric. See Corollary~\ref{cor:conjugacy-poc-transfer} for the precise statement.

In Section~\ref{sec:Gaus}, we analyze Gaussian targets with constant and bilinear kernels driving the SVGD flow.  The constant
kernel gives a first-moment feature and a simple warm-up uniform POC estimate.
The Gaussian-bilinear pair is the main example: its Stein features are exactly
the first and second moments, so the SVGD dynamics reduces to a closed
finite-dimensional moment flow.  We prove a direct uniform stability estimate
for this flow and obtain an $\mathbb E[\sup_{t\ge0}]$ propagation-of-chaos
bound with rate $N^{-1}$ in squared feature distance, allowing non-Gaussian initial laws.  In the Gaussian-initialization case, the same argument,
combined with an affine particle representation of the SVGD flow and standard empirical
Gaussian moment-matching estimates, gives a stronger version of the uniform-in-time
Wasserstein POC theorem of Liu--Ghosal--Balasubramanian--Pillai~\cite{liu2023towards}.

Finally, in Section~\ref{sec:pullback}, we return to the conjugacy framework and analyze diffeomorphism-based pullbacks of a base pair comprising a Gaussian target and bilinear (or constant) kernel.  This produces target-adapted
matrix-valued kernels for Gaussian-conjugate targets, which includes a large family of non-Gaussian and multi-modal densities. The resulting Stein
features are simply the first moment feature $A(x)$, for the pulled-back
constant kernel, and the first and second moment features of $A(x)$, for the
pulled-back bilinear kernel.  The uniform POC estimates from Section~\ref{sec:Gaus} therefore
transfer directly to these non-Gaussian targets in the corresponding intrinsic
feature metrics.  The section also clarifies why matrix-valued kernels are
natural here: in dimension $d\ge2$, nonlinear changes of variables do not
preserve the scalar-kernel form.

\subsection{Matrix-kernel SVGD and Stein features}\label{sec:mvSVGD}
Let
\[
        \pi(dx)=Z^{-1}e^{-V(x)}\,dx
\]
be a probability measure on $\R^d$.  A matrix-valued kernel is a map
\[
        K:\R^d\times\R^d\to\R^{d\times d}
\]
such that $K(x,y)=K(y,x)^\top$ and
\[
        \sum_{i,j=1}^n a_i^\top K(x_i,x_j)a_j\ge0,
        \qquad a_i\in\R^d.
\]
For a probability measure $\rho$, the matrix-SVGD velocity is
\begin{equation}\label{eq:matrix-velocity}
        v_\rho(x)
        :=
        \int_{\R^d}
        \left\{
        -K(x,y)\nabla V(y)+\diver_y K(x,y)
        \right\}\rho(dy),
\end{equation}
where $\diver_yK(x,y)\in\R^d$ has $a$th coordinate
\[
        [\diver_yK(x,y)]_a
        :=
        \sum_{b=1}^d\partial_{y_b}K_{ab}(x,y).
\]
The mean-field SVGD equation is given by~\eqref{eq:mf-svgd} with the above velocity and $\mu_{t=0}=\mu_0$.

The $N$-particle system is
\begin{equation}\label{eq:particle-svgd}
        \dot X_i^N(t)=v_{\mu_t^N}(X_i^N(t)),
        \qquad
        \mu_t^N:=\frac1N\sum_{i=1}^N\delta_{X_i^N(t)} .
\end{equation}


We will use finite-rank matrix kernels of the form
\begin{equation}\label{eq:finite-rank-K}
        K(x,y)=B(x)B(y)^\top,
        \qquad B(x)\in\R^{d\times m}.
\end{equation}

Writing $b_j(x)$ for the $j$th column of $B(x)$, define the Stein features
\begin{equation}\label{eq:matrix-features}
        f_j(x):=\diver b_j(x)-b_j(x)\cdot\nabla V(x),
        \qquad j=1,\ldots,m,
\end{equation}
and set $F=(f_1,\ldots,f_m)$.  Then
\begin{equation}\label{eq:velocity-finite-rank}
        v_\rho(x)=B(x)\rho F,
        \qquad
        \rho F:=\int F(x)\rho(dx)\in\R^m.
\end{equation}
Thus, the measure-dependence of both the particle system and the mean-field equation arises through the finite feature vector $\rho F$.

For such a feature vector we use the restricted metric
\begin{equation}\label{eq:feature-metric}
        d_F(\rho,\nu):=\norm{\rho F-\nu F}_2 .
\end{equation}
Equivalently, \(d_F\) is the integral probability pseudometric generated by the finite-dimensional class
\[
    \mathcal G_F
    :=
    \bigl\{x\mapsto a^\top F(x):a\in\mathbb R^m,\ \|a\|_2\le1\bigr\},
    \qquad
    d_F(\rho,\nu)
    =
    \sup_{g\in\mathcal G_F}\bigl|\rho g-\nu g\bigr|.
\]

For a canonical finite-rank factorization with the standard feature-space inner product, this is the two-sample KSD associated with the chosen finite-rank kernel. In a nonminimal factorization, it should be viewed as the chosen coefficient-space Stein-feature integral probability (semi)-metric. It is nevertheless substantially weaker than \(W_1\), which tests all \(1\)-Lipschitz functions, or a KSD generated by a rich infinite-dimensional kernel.  In particular, \(F\) need not be measure determining, so distinct measures may satisfy \(d_F(\rho,\nu)=0\).  The advantage is that \(d_F\) gives an interpretable and dynamically tractable comparison of selected statistics.  For example, \(F(x)=x\) compares means, while \(F(x)=(-x,\operatorname{vec}(I_d-xx^\top))\) compares first and second moments.  For the Gaussian-conjugate constructions considered below, the corresponding features compare the latent moments \(\rho A\) and \(\rho(AA^\top)\), for some statistic $A$.  Thus, \(d_F\) sacrifices full distributional discrimination in exchange for sharp control of a prescribed finite collection of informative observables.

\textbf{Constructing features from statistics.}
For practical use, the construction above can be reversed by prescribing the statistics to be controlled and then solving a Stein equation for the kernel features.  Specifically, let \(h_1,\ldots,h_m\) be observables of interest and set \(f_j:=h_j-\pi h_j\).  If one can find sufficiently regular vector fields \(b_j:\mathbb R^d\to\mathbb R^d\) satisfying
\[
    \diver b_j-b_j\cdot\nabla V=f_j,
    \qquad\text{equivalently}\qquad
    \diver\!\left(e^{-V}b_j\right)=e^{-V}f_j,
\]
then, with \(B=(b_1,\ldots,b_m)\) and \(K(x,y)=B(x)B(y)^\top\), the resulting Stein features are exactly the prescribed centered statistics \(f_j\), and hence
\[
    d_F(\rho,\nu)
    =
    \left(\sum_{j=1}^m\{\rho h_j-\nu h_j\}^2\right)^{1/2}.
\]
A convenient choice is \(b_j=\nabla u_j\), where \(u_j\) solves the weighted Poisson equation
\[
    \Delta u_j-\nabla V\cdot\nabla u_j=h_j-\pi h_j.
\]
In one dimension the construction is explicit: writing \(q(x)=Z^{-1}e^{-V(x)}\),
\[
    b_j(x)
    =
    \frac{1}{q(x)}
    \int_{-\infty}^x
    \bigl(h_j(y)-\pi h_j\bigr)q(y)\,dy,
\]
provided the above integral is well-defined, and the possibly unknown normalizing constant cancels from this expression.  For example, consider the smooth bimodal target
\[
    q_a(x)=\frac12\phi(x-a)+\frac12\phi(x+a),
    \qquad a>1,
\]
where \(\phi\) is the standard Gaussian density.  Choosing
\(h_1(x)=x\) and \(h_2(x)=x^2\), for which
\(\pi_a h_1=0\) and \(\pi_a h_2=1+a^2\), define
\[
    b_{1,a}(x)
    :=
    \frac{
        a\{\Phi(x-a)-\Phi(x+a)\}
        -\phi(x-a)-\phi(x+a)
    }{2q_a(x)},
\]
and
\[
    b_{2,a}(x)
    :=
    \frac{1}{q_a(x)}
    \int_{-\infty}^x
    \bigl(y^2-(1+a^2)\bigr)q_a(y)\,dy .
\]
Then
\[
    k_a(x,y)
    =
    b_{1,a}(x)b_{1,a}(y)
    +
    b_{2,a}(x)b_{2,a}(y)
\]
is a positive-definite, target-adapted rank-two kernel whose Stein features are
\(x\) and \(x^2-(1+a^2)\).  Consequently, its restricted metric compares precisely the first two moments:
\[
    d_F(\rho,\nu)
    =
    \left[
        \{\rho x-\nu x\}^2
        +
        \{\rho x^2-\nu x^2\}^2
    \right]^{1/2}.
\]
Thus, for a fixed target, the kernel can be designed to emphasize scientifically relevant observables; the remaining analytical task is to verify the regularity, growth, closure, and stability properties of the resulting feature dynamics.

\begin{remark}[Relation to mean-field limits for matrix-kernel SVGD]
Matrix-valued kernels for SVGD were introduced in \cite{wang2019stein}
as a way to incorporate geometric preconditioning, such as Hessian or Fisher
information matrices, into the particle update. Rigorous mean-field limits
and well-posedness theory for SVGD are by now available in considerable
generality for scalar kernels; see, for example, \cite{LuLuNolen2019}.
For sufficiently regular bounded matrix-valued kernels, or for unbounded kernels where the target-kernel pair has certain stability properties (like in Section~\ref{sec:wasserstein-poc}), one expects the same
McKean--Vlasov transport arguments to yield Dobrushin-type stability results and, consequently, finite-time mean-field limits.
The present section uses a more rigid structure. Namely, for finite-rank
matrix kernels the Stein velocity closes on finitely many Stein features,
so both the empirical and mean-field dynamics reduce to a finite-dimensional
feature flow. This closure, rather than a general matrix-kernel mean-field
theory, is the mechanism behind the uniform-in-time estimates proved below.
\end{remark}

\begin{remark}[Global well-posedness]
We will only use matrix kernels for which the particle system and mean-field equations are globally well-posed. For the finite-rank kernels used below, global well-posedness follows from non-explosion of the closed feature vector together with the explicit growth structure of $B$. In the concrete Gaussian base cases treated in Section~\ref{sec:Gaus}, this moment vector satisfies a closed globally well-posed ODE. For the conjugate kernels in Section~\ref{sec:pullback}, well-posedness
is transferred from the Gaussian base system through the diffeomorphism.
\end{remark}

\subsection{Conjugacy under a diffeomorphism}\label{sec:conj}

Let
\[
        \widetilde\pi(dy)=\widetilde Z^{-1}e^{-\widetilde V(y)}\,dy
\]
be a base target on $\R^d$.  Let $H:\R^d\to\R^d$ be a $C^2$ orientation-preserving diffeomorphism, set
\[
        J(y):=DH(y),
\]
and define the transformed potential by
\begin{equation}\label{eq:potential-transform}
        V_H(H(y)):=\widetilde V(y)+\log\det J(y).
\end{equation}
Then $\pi_H:=H_{\#}\widetilde\pi$ has density proportional to $e^{-V_H}$.
Indeed,
\[
        e^{-V_H(H(y))}\det J(y)\,dy=e^{-\widetilde V(y)}\,dy.
\]

Let the base kernel be finite-rank,
\[
        \widetilde K(y,y')=\widetilde B(y)\widetilde B(y')^\top, \qquad \widetilde B(y)\in\R^{d\times m}.
\]
Define the transformed factor and kernel by
\begin{equation}\label{eq:factor-transform}
        B_H(H(y)):=J(y)\widetilde B(y),
\end{equation}
and
\begin{equation}\label{eq:kernel-transform}
        K_H(H(y),H(y'))
        :=
        J(y)\widetilde K(y,y')J(y')^\top .
\end{equation}
Then $K_H(x,x')=B_H(x)B_H(x')^\top$ is again positive definite.

\begin{lemma}[Conjugacy of Stein features]\label{lem:feature-conjugacy}
Let $\widetilde b_j$ be the columns of $\widetilde B$, and let $b_j^H$ be the columns of $B_H$.  Define
\[
        \widetilde f_j(y)
        :=\diver_y\widetilde b_j(y)-\widetilde b_j(y)\cdot\nabla\widetilde V(y),
\]
and
\[
        f_j^H(x)
        :=\diver_x b_j^H(x)-b_j^H(x)\cdot\nabla V_H(x).
\]
Then
\begin{equation}\label{eq:features-conjugate}
        f_j^H(H(y))=\widetilde f_j(y),
        \qquad j=1,\ldots,m.
\end{equation}
Consequently, if $\rho_H=H_{\#}\widetilde\rho$, then
\begin{equation}\label{eq:moment-conjugacy}
        \rho_H F_H=\widetilde\rho\,\widetilde F.
\end{equation}
\end{lemma}

\begin{proof}
Use the weighted-divergence identity
\[
        f_j^H=e^{V_H}\diver_x(e^{-V_H}b_j^H).
\]
By change of variables along with the divergence theorem, for any smooth vector field $W$,
\[
        \diver_x W(H(y))
        =
        \frac{1}{\det J(y)}
        \diver_y\left(\det J(y)J(y)^{-1}W(H(y))\right).
\]
Applying this with $W=e^{-V_H}b_j^H$ gives
\begin{align*}
\diver_x(e^{-V_H}b_j^H)(H(y))
&=
\frac{1}{\det J(y)}
\diver_y\left(
\det J(y)J(y)^{-1}
\frac{e^{-\widetilde V(y)}}{\det J(y)}
J(y)\widetilde b_j(y)
\right)\\
&=
\frac{1}{\det J(y)}
\diver_y\left(e^{-\widetilde V(y)}\widetilde b_j(y)\right).
\end{align*}
Multiplying by $e^{V_H(H(y))}=\det J(y)e^{\widetilde V(y)}$ gives
\[
        f_j^H(H(y))
        =e^{\widetilde V(y)}\diver_y(e^{-\widetilde V}\widetilde b_j)(y)
        =\widetilde f_j(y).
\]
The identity for expectations follows by the change of variables $x=H(y)$.
\end{proof}

\begin{lemma}[Velocity conjugacy]\label{lem:velocity-conjugacy}
Let $\rho_H=H_\#\widetilde\rho$.  Then, writing $\widetilde v_{\widetilde \rho}$ for the SVGD velocity at $\widetilde \rho$ for potential $\widetilde V$ and kernel $\widetilde K$, and $ v_{\rho_H}^H$ for the velocity at $\rho_H$ for potential $V_H$ and kernel $K_H$, we have
\begin{equation}\label{eq:velocity-conjugacy}
        v_{\rho_H}^H(H(y))=J(y)\widetilde v_{\widetilde\rho}(y).
\end{equation}
Consequently, if $Y_t$ solves the base characteristic equation
\[
        \dot Y_t=\widetilde v_{\widetilde\rho_t}(Y_t),
\]
then $X_t:=H(Y_t)$ solves the transformed characteristic equation
\[
        \dot X_t=v_{\rho_t^H}^H(X_t),
        \qquad \rho_t^H=H_\#\widetilde\rho_t.
\]
\end{lemma}

\begin{proof}
By the finite-rank representation and Lemma~\ref{lem:feature-conjugacy},
\[
        v_{\rho_H}^H(H(y))
        =B_H(H(y))\rho_HF_H
        =J(y)\widetilde B(y)\widetilde\rho\widetilde F
        =J(y)\widetilde v_{\widetilde\rho}(y).
\]
The characteristic statement follows from $\dot X_t=DH(Y_t)\dot Y_t$.
\end{proof}

The above observations lead to the following useful corollary.

\begin{corollary}[Transfer of feature-level POC rates within a conjugacy class]
\label{cor:conjugacy-poc-transfer}
Let $(\widetilde\pi,\widetilde K)$ be a finite-rank base pair with Stein
feature vector $\widetilde F$, and let $(\pi_H,K_H)$ be any conjugate pair
constructed from a diffeomorphism $H$ as in \eqref{eq:potential-transform}
and \eqref{eq:kernel-transform}.  Denote its Stein feature vector by $F_H$.

Fix an initial law $\widetilde\mu_0$ for the base system, and define the conjugate initial law \( \mu_0^H := H_{\#}\widetilde\mu_0 . \) Let $(\widetilde\mu_t)_{t\ge0}$ be the base mean-field solution started from $\widetilde\mu_0$, and let $(\mu_t^H)_{t\ge0}$ be the transformed mean-field solution started from $\mu_0^H$. Let \( \widetilde X_1(0),\ldots,\widetilde X_N(0) \quad\text{be i.i.d. with law }\widetilde\mu_0, \) and define the transformed initial particles by \( X_i^H(0):=H(\widetilde X_i(0)), \, i=1,\ldots,N. \) Let $\widetilde\mu_t^N$ and $\mu_t^{H,N}$ denote the corresponding empirical particle solutions.

Then, for every $t\ge0$,
\begin{equation}\label{eq:conjugacy-feature-isometry}
        d_{F_H}(\mu_t^{H,N},\mu_t^H)
        =
        d_{\widetilde F}(\widetilde\mu_t^N,\widetilde\mu_t).
\end{equation}
Consequently, any propagation-of-chaos estimate for the base pair in the
restricted feature metric transfers to every conjugate pair with the same rate
and the same constant.  In particular, if for some rate $\varepsilon_N\to0$,
\[   
        \E\left[\sup_{t\ge0}\, d_{\widetilde F}^2(\widetilde\mu_t^N,\widetilde\mu_t)\right]
        \le C\varepsilon_N,
\]
then
\begin{equation}\label{eq:conjugacy-transferred-uniform-poc}
        \E\left[\sup_{t\ge0}\, d_{F_H}^2(\mu_t^{H,N},\mu_t^H)\right]
        \le C\varepsilon_N .
\end{equation}
\end{corollary}

\subsection{Uniform-in-time POC for Gaussian targets under constant and bilinear kernels}\label{sec:Gaus}

In this section the base target is the standard Gaussian
\[
        \gamma_d(dy)=(2\pi)^{-d/2}e^{-\norm y^2/2}\,dy.
\]

\subsubsection{The constant kernel}

As a warm-up, we start with the very simple kernel
\begin{equation}\label{eq:constant-kernel}
        \widetilde K_1(y,y')=\Id_d.
\end{equation}
This has factor $\widetilde B_1(y)\equiv\Id_d$.  The columns are the constant vector fields $e_1,\ldots,e_d$, and hence the Stein feature vector is
\begin{equation}\label{eq:constant-features}
        \widetilde F_1(y)=-y.
\end{equation}
Therefore
\begin{equation}\label{eq:constant-velocity}
        \widetilde v_{\widetilde\rho}(y)=-\widetilde m_\rho,
        \qquad
        \widetilde m_\rho:=\int y\,\widetilde\rho(dy).
\end{equation}
If $\widetilde m(t):=\widetilde\mu_t y$, then
\begin{equation}\label{eq:constant-moment-ode}
        \dot{\widetilde m}(t)=-\widetilde m(t),
        \qquad
        \widetilde m(t)=e^{-t}\widetilde m(0).
\end{equation}
[The derivative above can be justified whenever the initial distribution $\widetilde \mu_0$ satisfies $\int\|y\| \, \widetilde \mu_0(dy)<\infty$.]

For the particle system, $\widetilde m^N(t):=\widetilde\mu_t^N y$ satisfies the same scalar ODE with initial condition $\widetilde m^N(0)$.

\begin{proposition}[Direct uniform POC for the constant Gaussian kernel]\label{prop:constant-gaussian-poc}
Let $Y_1(0),\ldots,Y_N(0)$ be i.i.d. with law $\widetilde\mu_0$, and assume
\[
        \E\norm{Y(0)}^2<\infty.
\]
Let $\widetilde\mu_t$ and $\widetilde\mu_t^N$ be the mean-field and particle solutions for the kernel \eqref{eq:constant-kernel}.  Then
\begin{equation}\label{eq:constant-uniform-poc}
        \E\left[\sup_{t\ge0}\,\norm{\widetilde\mu_t^N y-\widetilde\mu_t y}^2\right]
        \le
        \frac{\Tr\Cov(Y(0))}{N}.
\end{equation}
\end{proposition}

\begin{proof}
By \eqref{eq:constant-moment-ode},
\[
        \widetilde\mu_t^Ny-\widetilde\mu_ty
        =e^{-t}\left(\widetilde\mu_0^Ny-\widetilde\mu_0y\right).
\]
Taking expectations and using i.i.d. empirical averaging gives
\[
       \E\norm{\widetilde\mu_0^Ny-\widetilde\mu_0y}^2
        =
        \frac{\Tr\Cov(Y(0))}{N},
\]
which proves the lemma.
\end{proof}

\subsubsection{The bilinear kernel}

Now, we consider the bilinear kernel given by
\begin{equation}\label{eq:affine-kernel}
        \widetilde K_2(y,y')=(1+y^\top y')\Id_d.
\end{equation}
A finite-rank factor is
\begin{equation}\label{eq:affine-factor}
        \widetilde B_2(y)=\bigl[\,\Id_d,\ y_1\Id_d,\ldots,y_d\Id_d\,\bigr]
        \in\R^{d\times(d+d^2)}.
\end{equation}
The Stein features are
\begin{equation}\label{eq:affine-features}
        -y_i,
        \qquad
        \delta_{ij}-y_iy_j,
        \qquad 1\le i,j\le d.
\end{equation}
Writing
\[
        \widetilde m_\rho:=\int y\,\widetilde\rho(dy),
        \qquad
        \widetilde S_\rho:=\int yy^\top\,\widetilde\rho(dy),
\]
the velocity is
\begin{equation}\label{eq:affine-velocity}
        \widetilde v_{\widetilde\rho}(y)
        =-\widetilde m_\rho+(\Id_d-\widetilde S_\rho)y.
\end{equation}
Thus, for
\[
        \widetilde m(t):=\widetilde\mu_t y,
        \qquad
        \widetilde S(t):=\widetilde\mu_t(yy^\top),
\]
the first two moments solve
\begin{equation}\label{eq:affine-moment-ode}
        \dot{\widetilde m}=-\widetilde S\widetilde m,
        \qquad
        \dot{\widetilde S}=2\widetilde S-2\widetilde S^2-2\widetilde m\widetilde m^\top .
\end{equation}
[The derivatives above can be justified whenever the initial distribution $\widetilde \mu_0$ satisfies $\int\|y\|^2 \, \widetilde \mu_0(dy)<\infty$.]

The same ODE is satisfied by the empirical pair
\[
        \widetilde m^N(t):=\widetilde\mu_t^Ny,
        \qquad
        \widetilde S^N(t):=\widetilde\mu_t^N(yy^\top),
\]
with the corresponding empirical initial condition.

For later use define the moment distance
\begin{equation}\label{eq:latent-moment-distance}
        D_2\bigl((m,S),(\bar m,\bar S)\bigr)
        :=
        \left(\norm{m-\bar m}^2+\norm{S-\bar S}_F^2\right)^{1/2},
\end{equation}
where $\|\cdot\|_F$ denotes the Frobenius norm.

\begin{lemma}[Uniform stability on bounded nondegenerate moment sets]
\label{lem:affine-stability}
Fix $R<\infty$ and $c_0>0$, and define
\[
        \mathcal{D}_{R,c_0}
        :=
        \Bigl\{
        (m,S):\ \|m\|+\|S\|_F\le R,\quad
        S-mm^\top\ge c_0 I_d
        \Bigr\}.
\]
Then there exist constants
\[
        \eta=\eta(R,c_0,d)>0,
        \qquad
        L=L(R,c_0,d)<\infty,
\]
such that, for every $(m_0,S_0)\in\mathcal D_{R,c_0}$ and every
$(m,S)$ in
$
       U\coloneqq \{ (m,S): D_2\bigl((m,S),(m_0,S_0)\bigr)\le \eta\},
$
the flow $\Phi_t$ of \eqref{eq:affine-moment-ode} satisfies
\[
        \sup_{t\ge0}
        D_2\bigl(\Phi_t(m,S),\Phi_t(m_0,S_0)\bigr)
        \le
        L\,D_2\bigl((m,S),(m_0,S_0)\bigr).
\]
\end{lemma}

\begin{proof}
Write
\[
        z=(m,S),\qquad \Sigma=S-mm^\top,
\]
and let
\[
        \mathcal D:=\{(m,S): S=S^\top,\ \Sigma=S-mm^\top>0\}.
\]
The vector field in \eqref{eq:affine-moment-ode} is polynomial, hence smooth on
$\R^d\times\R^{d\times d}_{\rm sym}$.  We first record two elementary
a priori estimates on the domain $\mathcal D$.

Let $(m(t),S(t))$ be a solution of \eqref{eq:affine-moment-ode} and set
$\Sigma(t)=S(t)-m(t)m(t)^\top$.  We obtain
\begin{align*}
        \dot \Sigma
        &=\dot S-\dot m\,m^\top-m\,\dot m^\top  \\
        &=2S-2S^2-2mm^\top+Smm^\top+mm^\top S  \\
        &=(\Id_d-S)\Sigma+\Sigma(\Id_d-S).
\end{align*}
Consequently, as long as the solution exists,
\[
        \Sigma(t)=R(t)\Sigma(0)R(t)^\top,
        \qquad
        \dot R(t)=(\Id_d-S(t))R(t),\quad R(0)=\Id_d.
\]
Thus positive definiteness of $\Sigma$ is preserved.

Next set
\[
        q(t):=\Tr S(t).
\]
Since $S(t)=\Sigma(t)+m(t)m(t)^\top\ge0$, we have using the Cauchy-Schwarz inequality,
\begin{equation}\label{cs}
        \Tr(S(t)^2)\ge \frac{1}{d}(\Tr S(t))^2.
\end{equation}
Taking traces in the $S$ equation gives
\[
        \dot q(t)
        =
        2q(t)-2\Tr(S(t)^2)-2\norm{m(t)}^2
        \le
        2q(t)-\frac{2}{d}q(t)^2.
\]
Hence, by comparison with the scalar logistic equation,
\begin{equation}\label{eq:trace-S-uniform-bound}
        \Tr S(t)\le \max\{\Tr S(0),d\},
        \qquad t\ge0.
\end{equation}

We also use the Lyapunov function, given by the `Gaussian entropy functional' 
\[
        \mathcal H(m,S)
        :=
        \frac12\Bigl(\Tr S-\log\det \Sigma-d\Bigr),
        \qquad \Sigma=S-mm^\top .
\]
Along any solution in $\mathcal D$, invoking Jacobi's formula,
\[
        \frac{d}{dt}\mathcal H(m(t),S(t))
        =
        \frac12\Tr \dot S(t)
        -
        \frac12\Tr\bigl(\Sigma(t)^{-1}\dot \Sigma(t)\bigr).
\]
Using
\[
        \dot \Sigma=(\Id_d-S)\Sigma + \Sigma(\Id_d-S),
\]
we get
\[
        \Tr(\Sigma^{-1}\dot \Sigma)
        =
        2\Tr(\Id_d-S).
\]
Therefore
\begin{align*}
        \frac{d}{dt}\mathcal H(m(t),S(t))
        &=
        \frac12\Bigl(
        2\Tr S-2\Tr(S^2)-2\norm m^2
        -
        2\Tr(\Id_d-S)
        \Bigr) \\
        &=
        2\Tr S-\Tr(S^2)-d-\norm m^2 \\
        &=
        -\norm{S-\Id_d}_F^2-\norm m^2.
\end{align*}
Thus
\begin{equation}\label{eq:entropy-dissipation-affine}
        \frac{d}{dt}\mathcal H(m(t),S(t))
        =
        -\Bigl(\norm{S(t)-\Id_d}_F^2+\norm{m(t)}^2\Bigr).
\end{equation}

We now make the constants uniform in the starting point chosen in the appropriate `good' set.  
Using continuity, choose a neighborhood of $\mathcal{D}_{R,c_0}$ as
$$
\mathcal{D}^\eta_{R,c_0} := \{(m,S) \in  \mathcal{D} : D_2\bigl((m,S),(m_0,S_0)\bigr)\le \eta \text{ for some } (m_0,S_0) \in \mathcal{D}_{R,c_0}\},$$ 
by choosing $\eta$ as in the lemma small enough that, for all
$(m,S)\in \mathcal{D}^\eta_{R,c_0}$,
\[
        \norm m+\norm S_F\le R+1,
        \qquad
        S-mm^\top\ge \frac{c_0}{2}\Id_d .
\]

For initial data in $\mathcal{D}^\eta_{R,c_0}$, \eqref{eq:trace-S-uniform-bound} and \eqref{cs} give
\[
        \Tr S(t)\le Q
        :=
        \max\{\sqrt d\,(R+1),d\},
        \qquad t\ge0.
\]
Moreover, since initially
\[
        \Sigma(0)\ge \frac{c_0}{2}\Id_d,
\]
$\mathcal H(m(0),S(0))$ is bounded above by a constant $H_\ast<\infty$ depending
only on $R,c_0,d$.  By \eqref{eq:entropy-dissipation-affine},
\[
        \mathcal H(m(t),S(t))\le H_\ast,
        \qquad t\ge0.
\]
Hence
\[
        -\log\det \Sigma(t)\le 2H_\ast+d,
\]
and therefore
\[
        \det \Sigma(t)\ge e^{-2H_\ast-d}.
\]
Since $0<\Sigma(t)\le S(t)$ and $\Tr S(t)\le Q$, all eigenvalues of $\Sigma(t)$ are at
most $Q$.  It follows that
\[
        \lambda_{\min}(\Sigma(t))
        \ge
        \kappa
        :=
        e^{-2H_\ast-d}Q^{-(d-1)}.
\]
Thus every trajectory starting from $\mathcal{D}^\eta_{R,c_0}$ remains for all $t\ge0$ in
the compact set
\[
        \mathcal K
        :=
        \Bigl\{
        (m,S)\in\mathcal D:
        \Tr S\le Q,\quad S-mm^\top\ge \kappa\Id_d
        \Bigr\}.
\]
In particular the solutions are global.

On $\mathcal K$, the functions
\[
        (m,S)\mapsto \mathcal H(m,S),
        \qquad
        (m,S)\mapsto
        E(m,S):=\norm m^2+\norm{S-\Id_d}_F^2
\]
vanish only at the equilibrium $(0,\Id_d)$.  Since $\mathcal K$ is compact and
the Hessian of $\mathcal H$ at $(0,\Id_d)$ is positive definite in the variables
$(m,S)$, there are constants $0<a\le A<\infty$, depending only on
$R,c_0,d$, such that
\[
        a\,E(m,S)
        \le
        \mathcal H(m,S)
        \le
        A\,E(m,S),
        \qquad (m,S)\in\mathcal K.
\]
Combining this with \eqref{eq:entropy-dissipation-affine} gives
\[
        \frac{d}{dt}\mathcal H(m(t),S(t))
        =
        -E(m(t),S(t))
        \le
        -A^{-1}\mathcal H(m(t),S(t)).
\]
Consequently,
\[
        \mathcal H(m(t),S(t))
        \le
        e^{-t/A}\mathcal H(m(0),S(0)),
\]
and hence
\begin{equation}\label{eq:uniform-exp-convergence-affine}
        \norm{m(t)}^2+\norm{S(t)-\Id_d}_F^2
        \le
        C e^{-ct},
        \qquad t\ge0,
\end{equation}
with constants $C,c>0$ depending only on $R,c_0,d$, uniformly for all initial
data in $\mathcal{D}^\eta_{R,c_0}$.

It remains to turn this uniform convergence into a uniform Lipschitz estimate
for the flow.  Let
\[
        z(t)=(m(t),S(t)),
        \qquad
        \bar z(t)=(\bar m(t),\bar S(t))
\]
be two solutions starting from $\mathcal{D}^\eta_{R,c_0}$.  Since the vector field is $C^1$
on the compact set $\mathcal K$, there is a constant $M<\infty$, depending
only on $\mathcal K$, such that
\[
        \norm{D\mathcal F(z)}\le M,
        \qquad z\in\mathcal K,
\]
where $\mathcal F$ denotes the driving vector field in \eqref{eq:affine-moment-ode}.  Therefore
Gronwall's inequality gives
\begin{equation}\label{eq:finite-time-lipschitz-affine}
        D_2(z(t),\bar z(t))
        \le
        e^{Mt}D_2(z(0),\bar z(0)),
        \qquad t\ge0.
\end{equation}

We now show that, once both solutions are sufficiently close to the equilibrium,
the distance cannot grow.  Write
\[
        A(t):=S(t)-\Id_d,
        \qquad
        \bar A(t):=\bar S(t)-\Id_d,
\]
and set
\[
        u(t):=m(t)-\bar m(t),
        \qquad
        B(t):=A(t)-\bar A(t)=S(t)-\bar S(t).
\]
In these variables the moment equations become
\[
        \dot m=-m-Am,
        \qquad
        \dot A=-2A-2A^2-2mm^\top.
\]
Subtracting the two equations gives
\[
        \dot u=-u-Au-B\bar m
\]
and
\[
        \dot B
        =
        -2B-2(AB+B\bar A)
        -2\bigl(u m^\top+\bar m u^\top\bigr).
\]
Assume that
\[
        \norm m+\norm{\bar m}+\norm A_F+\norm{\bar A}_F\le r.
\]
Then, using $\norm A_{\rm op}\le \norm A_F$, we have
\begin{align*}
        \frac{d}{dt}\norm u^2
        &=
        2u^\top\dot u  \\
        &\le
        -2\norm u^2
        +2\norm A_F\norm u^2
        +2\norm B_F\norm{\bar m}\norm u  \\
        &\le
        (-2+3r)\norm u^2+r\norm B_F^2.
\end{align*}
Similarly,
\begin{align*}
        \frac{d}{dt}\norm B_F^2
        &=
        2\ip{B}{\dot B}_F  \\
        &\le
        -4\norm B_F^2
        +4(\norm A_F+\norm{\bar A}_F)\norm B_F^2
        +4\norm B_F\norm u(\norm m+\norm{\bar m}) \\
        &\le
        2r\norm u^2+(-4+6r)\norm B_F^2.
\end{align*}
Hence
\[
        \frac{d}{dt}
        \left(\norm u^2+\norm B_F^2\right)
        \le
        (-2+5r)\norm u^2+(-4+7r)\norm B_F^2.
\]
Choosing $r>0$ sufficiently small, we get
\begin{equation}\label{eq:local-contraction-affine}
        \frac{d}{dt}
        \left(\norm u^2+\norm B_F^2\right)
        \le
        -\left(\norm u^2+\norm B_F^2\right)
\end{equation}
whenever both trajectories lie in the $r$-neighborhood of $(0,\Id_d)$.

By the uniform exponential convergence estimate
\eqref{eq:uniform-exp-convergence-affine}, there exists a time
$T<\infty$, depending only on $R,c_0,d$, such that every solution starting from
$\mathcal{D}^\eta_{R,c_0}$ belongs to this $r$-neighborhood of $(0,\Id_d)$ for all $t\ge T$.
Therefore \eqref{eq:local-contraction-affine} implies that, for $t\ge T$,
\[
        D_2(z(t),\bar z(t))
        \le
        D_2(z(T),\bar z(T)).
\]
Using \eqref{eq:finite-time-lipschitz-affine} up to time $T$, we obtain
\[
        \sup_{t\ge0}D_2(z(t),\bar z(t))
        \le
        e^{MT}D_2(z(0),\bar z(0)).
\]
Finally take
\[
        \bar z(0)=(m_0,S_0),
        \qquad
        z(0)=(m,S),
\]
and set
\(
        L:=e^{MT}.
\)
This gives
\[
        \sup_{t\ge0}
        D_2\bigl(\Phi_t(m,S),\Phi_t(m_0,S_0)\bigr)
        \le
        L D_2\bigl((m,S),(m_0,S_0)\bigr),
        \qquad (m,S)\in\mathcal U.
\]
All constants depend only on the bound $R$, on $c_0$, and on $d$, as claimed.
\end{proof}

\begin{proposition}[Direct uniform POC for the affine Gaussian kernel]
\label{prop:affine-gaussian-poc}
Let $Y_1(0),\ldots,Y_N(0)$ be i.i.d. with law $\widetilde\mu_0$.  Assume that
\[
        \E \exp\{\alpha \|Y(0)\|^2\}<\infty,
\quad
        \Cov(Y(0))\ge c_0 I_d
\]
for some $\alpha>0, c_0>0$.  Let $\widetilde\mu_t$ and $\widetilde\mu_t^N$ be the
mean-field and particle solutions for the kernel \eqref{eq:affine-kernel}.
Then there is a constant
$
        C=C\bigl(d,c_0,\alpha,\E e^{\alpha\|Y(0)\|^2}\bigr)<\infty
$
such that
\[
        \E\left[\sup_{t\ge0}\, D_2^2\bigl(
        (\widetilde m^N(t),\widetilde S^N(t)),
        (\widetilde m(t),\widetilde S(t))
        \bigr)\right]
        \le
        C N^{-1}.
\]
\end{proposition}

\begin{proof}
Write
\[
        z(t):=(\widetilde m(t),\widetilde S(t)),
        \qquad
        z^N(t):=(\widetilde m^N(t),\widetilde S^N(t)),
\]
and denote the initial moment pairs by
\[
        z_0:=(\widetilde m(0),\widetilde S(0)),
        \qquad
        z_0^N:=(\widetilde m^N(0),\widetilde S^N(0)).
\]
Both $z(t)$ and $z^N(t)$ solve the same
finite-dimensional ODE \eqref{eq:affine-moment-ode}, with initial data
$z_0$ and $z_0^N$, respectively.

Let
$
        K_\alpha:=\E e^{\alpha\norm{Y(0)}^2}<\infty.
$
This implies that all polynomial moments of
$Y(0)$ are finite, with bounds depending only on $\alpha$ and $K_\alpha$.
In particular,
\[
        \E\norm{Y(0)}^2+\E\norm{Y(0)}^4 + \E\norm{Y(0)}^8 \le C(d,\alpha,K_\alpha) <\infty.
\]
Choose
\[
        R\ge 1+\norm{\widetilde m(0)}+\norm{\widetilde S(0)}_F .
\]
Using the preceding moment bounds, $R$ may be chosen depending only on
$d,\alpha$ and $K_\alpha$.  Since
\[
        \widetilde \Sigma(0)
        :=
        \widetilde S(0)-\widetilde m(0)\widetilde m(0)^\top
        =
        \Cov(Y(0))
        \ge c_0\Id_d,
\]
we have
\[
        z_0\in \mathcal D_{R,c_0},
        \qquad
        \mathcal D_{R,c_0}
        :=
        \Bigl\{
        (m,S):\norm m+\norm S_F\le R,\ 
        S-mm^\top\ge c_0\Id_d
        \Bigr\}.
\]
Let $\eta=\eta(R,c_0,d)>0$ and $L=L(R,c_0,d)<\infty$ be the constants from
Lemma~\ref{lem:affine-stability}.  Define the good event
\[
        \Omega_N
        :=
        \Bigl\{
        D_2(z_0^N,z_0)\le \eta
        \Bigr\}.
\]

We first control the initial empirical moment error.  Since the particles are
i.i.d.,
\[
        \E\norm{\widetilde m^N(0)-\widetilde m(0)}^2
        =
        \frac{1}{N}\Tr\Cov(Y(0))
        \le
        \frac{1}{N}\E\norm{Y(0)}^2.
\]
Similarly,
\[
\begin{aligned}
        \E\norm{\widetilde S^N(0)-\widetilde S(0)}_F^2
        &=
        \frac{1}{N}
        \E\norm{Y(0)Y(0)^\top-\E[Y(0)Y(0)^\top]}_F^2  \\
        &\le
        \frac{1}{N}
        \E\norm{Y(0)Y(0)^\top}_F^2
        =
        \frac{1}{N}\E\norm{Y(0)}^4 .
\end{aligned}
\]
Therefore, 
\begin{equation}\label{eq:initial-moment-fluctuation-exp}
        \E D_2^2(z_0^N,z_0)
        \le
        C(d,\alpha,K_\alpha)N^{-1}.
\end{equation}

We next record a concentration estimate for $\Omega_N$.  The coordinates of
$Y(0)$ are sub-Gaussian, and the coordinates of
$Y(0)Y(0)^\top$ are sub-exponential.  Thus Bernstein's inequality for sub-exponential random variables, applied coordinatewise
and combined with a union bound over the finitely many first- and second-moment
coordinates, gives constants $C_1,c_1>0$, depending only on
$d,\alpha,K_\alpha$ and $\eta$, such that
\begin{equation}\label{eq:good-event-exp-prob}
        \mathbb{P}(\Omega_N^c)
        =
        \mathbb{P}\bigl(D_2(z_0^N,z_0)>\eta\bigr)
        \le
        C_1 e^{-c_1N}.
\end{equation}

On $\Omega_N$, Lemma~\ref{lem:affine-stability} yields
\[
        \sup_{t\ge0}D_2(z^N(t),z(t))
        \le
        L D_2(z_0^N,z_0).
\]
Consequently,
\begin{equation}\label{eq:good-event-poc}
\begin{aligned}
        \E\left[
        \mathbf 1_{\Omega_N}
        \sup_{t\ge0}D_2^2(z^N(t),z(t))
        \right]
        &\le
        L^2\,\E D_2^2(z_0^N,z_0)  \le
        C N^{-1},
\end{aligned}
\end{equation}
where $C$ depends only on $d,c_0,\alpha,K_\alpha$.

It remains to control the complementary event.  For any solution $(m(t), S(t))$ of the moment
ODE, by \eqref{eq:trace-S-uniform-bound},
\[
        \Tr S(t)\le \max\{\Tr S(0),d\},
        \qquad t\ge0.
\]
This applies both to the mean-field moment pair and to the empirical moment
pair.  Since $S(t)\ge0$ and $S^N(t)\ge0$, we have
\[
        \norm{m(t)}^2\le \Tr S(t),
        \qquad
        \norm{S(t)}_F\le \Tr S(t),
\]
and the same bounds hold for $m^N(t),S^N(t)$.  Hence there is a constant
$C_d<\infty$ such that
\[
        \sup_{t\ge0}D_2(z^N(t),z(t))
        \le
        C_d\Bigl(
        1+\Tr \widetilde S^N(0)+\Tr \widetilde S(0)
        \Bigr).
\]
Therefore, by the Cauchy--Schwarz inequality,
\begin{align}
        \E\left[
        \mathbf 1_{\Omega_N^c}
        \sup_{t\ge0}D_2^2(z^N(t),z(t))
        \right] \notag \le
        C_d^2
        \left(
        \E\bigl(1+\Tr \widetilde S^N(0)+\Tr \widetilde S(0)\bigr)^4
        \right)^{1/2}
        \mathbb{P}(\Omega_N^c)^{1/2}. \label{eq:bad-event-cs}
\end{align}
Now
\[
        \Tr \widetilde S^N(0)
        =
        \frac1N\sum_{i=1}^N \norm{Y_i(0)}^2,
\]
and so the fourth moment of $\Tr \widetilde S^N(0)+\Tr \widetilde S(0)$ is bounded uniformly in
$N$ by a constant depending only on $\E\norm{Y(0)}^8$, hence only on
$\alpha$ and $K_\alpha$.  Combining this with
\eqref{eq:good-event-exp-prob}, we obtain
\begin{equation}\label{eq:bad-event-poc}
        \E\left[
        \mathbf 1_{\Omega_N^c}
        \sup_{t\ge0}D^2_2(z^N(t),z(t))
        \right]
        \le
        C e^{-c_1N/2}
        \le
        C N^{-1},
\end{equation}
where $C$ depends only on $c_0,\alpha,d,K_\alpha$.

Finally, combining \eqref{eq:good-event-poc} and \eqref{eq:bad-event-poc}
gives
\[
        \E\left[
        \sup_{t\ge0}D_2^2(z^N(t),z(t))
        \right]
        \le
        C N^{-1}.
\]
This proves the proposition.
\end{proof}

\begin{corollary}[Gaussian initialization recovers the Wasserstein POC theorem of \cite{liu2023towards}]
\label{cor:recover-liu-wasserstein-poc}
Consider the affine Gaussian kernel \eqref{eq:affine-kernel} and the
standard Gaussian target on $\mathbb R^d$.  Assume that
\[
        \widetilde\mu_0=N(m_0,\Sigma_0),
        \qquad \Sigma_0>0,
\]
and let $Y_1(0),\ldots,Y_N(0)$ be i.i.d. with law $\widetilde\mu_0$.  Let
$\widetilde\mu_t$ denote the mean-field solution and let
\[
        \widetilde\mu_t^N:=\frac1N\sum_{i=1}^N\delta_{Y_i(t)}
\]
be the empirical measure of the particle system.  Then there is a constant
$C=C(d,m_0,\Sigma_0)<\infty$ such that, for all $N$ sufficiently large,
\begin{equation}\label{eq:liu-w2-recovered}
        \mathbb E \left[\sup_{t\ge0}\, W_2^2(\widetilde\mu_t^N,\widetilde\mu_t)\right]
        \le
        C r_d(N),
\end{equation}
where
\[
        r_d(N):=
        \begin{cases}
        N^{-1}\log\log N, & d=1,\\
        N^{-1}(\log N)^2, & d=2,\\
        N^{-2/d}, & d\ge3.
        \end{cases}
\]
Moreover, for fixed \(k\ge1\), writing $\mathcal L(Y_1^N(t),\ldots,Y_k^N(t))$ for the joint law of $Y_1^N(t),\ldots,Y_k^N(t)$,
\begin{equation}\label{eq:Gausspoc}
        \sup_{t\ge0}
        W_2^{(k)}\left(\mathcal L(Y_1^N(t),\ldots,Y_k^N(t)),\widetilde\mu_t^{\otimes k}\right)^2
        \le
        C'\left(k r_d(N)+\frac{k^2(k-1)}{N}\right),
\end{equation}
for some constant
$C'=C'(d,m_0,\Sigma_0)<\infty$.
\end{corollary}

\begin{proof}
Write
\[
        \widetilde\mu_t=N(m(t),\Sigma(t)),
        \qquad
        S(t)=\Sigma(t)+m(t)m(t)^\top .
\]
Let $m^N(t)$, $\Sigma^N(t)$, and $S^N(t)$ be the empirical mean, covariance,
and second moment of the particle system.  Define the Gaussian projection
of the empirical measure by
\[
        \Gamma_t^N:=N(m^N(t),\Sigma^N(t)).
\]
By the triangle inequality,
\begin{equation}\label{eq:w2-triangle-gaussian-projection}
        W_2^2(\widetilde\mu_t^N,\widetilde\mu_t)
        \le
        2W_2^2(\widetilde\mu_t^N,\Gamma_t^N)
        +
        2W_2^2(\Gamma_t^N,\widetilde\mu_t).
\end{equation}

We first control the second term involving the distance between the mean-field flow and the Gaussian-projection.  We use the following
elementary consequence of the Gaussian Wasserstein formula.  Let
\(
        \nu=N(m,\Sigma),
      \,
        \bar\nu=N(\bar m,\bar\Sigma),
\)
and suppose that
\(
        \Sigma\ge \lambda \Id_d,
        \,
        \bar\Sigma\ge \lambda \Id_d
\)
for some $\lambda>0$.  Then
\begin{equation}\label{eq:gaussian-w2-covariance-comparison}
        W_2^2(\nu,\bar\nu)
        \le
        \|m-\bar m\|^2
        +
        \frac{1}{4\lambda}\|\Sigma-\bar\Sigma\|_F^2 .
\end{equation}
Indeed, the Gaussian Wasserstein formula gives
\[
        W_2^2(\nu,\bar\nu)
        =
        \|m-\bar m\|^2+d_{\rm B}^2(\Sigma,\bar\Sigma),
\]
where
\[
        d_{\rm B}^2(\Sigma,\bar\Sigma)
        =
        \operatorname{Tr}\Sigma+\operatorname{Tr}\bar\Sigma
        -2\operatorname{Tr}
        \bigl(\bar\Sigma^{1/2}\Sigma\bar\Sigma^{1/2}\bigr)^{1/2}.
\]
Using the variational representation of the Bures distance,
\(
        d_{\rm B}(\Sigma,\bar\Sigma)
        =
        \min_{U\in O(d)}
        \|\Sigma^{1/2}-\bar\Sigma^{1/2}U\|_F,
\)
we have
\(
        d_{\rm B}(\Sigma,\bar\Sigma)
        \le
        \|\Sigma^{1/2}-\bar\Sigma^{1/2}\|_F .
\)
Since both covariance matrices are bounded below by $\lambda\Id_d$, the
matrix square-root map is Lipschitz in Frobenius norm on this set:
\(
        \|\Sigma^{1/2}-\bar\Sigma^{1/2}\|_F
        \le
        \frac{1}{2\sqrt{\lambda}}\|\Sigma-\bar\Sigma\|_F .
\)
This proves \eqref{eq:gaussian-w2-covariance-comparison}.

Now write
\[
        S=\Sigma+mm^\top,
        \qquad
        \bar S=\bar\Sigma+\bar m\bar m^\top .
\]
If additionally
\(
        \|m\|\vee\|\bar m\|\le M,
\)
then
\[
\begin{aligned}
        \|\Sigma-\bar\Sigma\|_F
        \le
        \|S-\bar S\|_F
        +
        \|mm^\top-\bar m\bar m^\top\|_F 
        \le
        \|S-\bar S\|_F
        +
        2M\|m-\bar m\|.
\end{aligned}
\]
Combining this with \eqref{eq:gaussian-w2-covariance-comparison} gives
\begin{equation}\label{eq:gaussian-w2-second-moment-comparison}
        W_2^2(\nu,\bar\nu)
        \le
        K_{\lambda,M}
        \left(
        \|m-\bar m\|^2+\|S-\bar S\|_F^2
        \right),
\end{equation}
where one may take
\(
        K_{\lambda,M}
        :=
        \max\left\{
        1+\frac{2M^2}{\lambda},
        \frac{1}{2\lambda}
        \right\}.
\)
Consequently,
\[
        W_2^2(\Gamma_t^N,\widetilde\mu_t)
        \le
        K_{\lambda,M} \,
        D_2\bigl((m^N(t),S^N(t)),(m(t),S(t))\bigr)^2 .
\]
Using the $L^2$ form of Proposition~\ref{prop:affine-gaussian-poc}, obtained
from the same proof by applying the $L^2$ concentration bound for the initial
empirical mean and second moment, gives
\begin{equation}\label{eq:projection-term-bound}
        \mathbb E\left[\sup_{t\ge0}\, W_2^2(\Gamma_t^N,\widetilde\mu_t)\right]
        \le
        K_{\lambda,M} \, N^{-1}.
\end{equation}

It remains to control the first term $W_2^2(\widetilde\mu_t^N,\Gamma_t^N)$ involving the distance between the empirical measure process and the Gaussian projection.  This is the only place where
Gaussian initialization is used.  It can be checked from the
bilinear-kernel driven SVGD that the particle system has the following affine representation:
\[
        Y_i(t)=m^N(t)+A_t^N\bigl(Y_i(0)-m^N(0)\bigr),
        \qquad i=1,\ldots,N,
\]
where $A_t^N$ solves the fundamental matrix differential equation
\[
\dot A_t^N= (I_d - S^N_t)A_t^N, \quad A_0^N = I_d.
\]
Consequently,
\begin{equation}\label{asig}
        \Sigma^N(t)=A_t^N \Sigma^N(0)(A_t^N)^\top .
\end{equation}
Thus $\widetilde\mu_t^N$ is the image of $\widetilde\mu_0^N$ under this
affine map governing particle trajectories, while $\Gamma_t^N$ is the image of
$N(m^N(0),\Sigma^N(0))$ under the same affine map.

Let
\[
        \mathcal G_N
        :=
        \left\{
        \Sigma^N(0)\ge \frac12 \Sigma_0,\quad
        \operatorname{Tr} S^N(0)\le R
        \right\},
\]
where $\operatorname{Tr} (\Sigma_0 + m_0m_0^T) <R<\infty$ is fixed sufficiently large depending only on $m_0$, $d_0$ and
$\Sigma_0$.  On $\mathcal G_N$, by \eqref{eq:trace-S-uniform-bound},
\[
        \sup_{t\ge0}\lVert \Sigma^N(t)\rVert_{\mathrm{op}}\le \max\{R,d\} .
\]
Since
\(
        A_t^N
        =
        A_t^N (\Sigma^N(0))^{1/2}(\Sigma^N(0))^{-1/2},
\)
we get using~\eqref{asig},
\[
        \lVert A_t^N\rVert_{\mathrm{op}}^2
        \le
        \lVert \Sigma^N(t)\rVert_{\mathrm{op}}
        \lVert (\Sigma^N(0))^{-1}\rVert_{\mathrm{op}}
        \le 2\max\{R,d\}  \lVert \Sigma_0^{-1}\rVert_{\mathrm{op}},
        \qquad t\ge0,
\]
on $\mathcal G_N$.  Hence, using the Lipschitz property of the affine particle map, we obtain
\begin{equation}\label{eq:shape-affine-bound}
        \sup_{t\ge0}
        W_2^2(\widetilde\mu_t^N,\Gamma_t^N)
        \le
        2\max\{R,d\} \lVert \Sigma_0^{-1}\rVert_{\mathrm{op}}\,
        W_2^2\bigl(\widetilde\mu_0^N,
                   N(m^N(0),\Sigma^N(0))\bigr)
        \qquad\text{on }\mathcal G_N .
\end{equation}

Next,
\[
\begin{aligned}
        W_2^2\bigl(\widetilde\mu_0^N,N(m^N(0),\Sigma^N(0))\bigr)
        \le
        2W_2^2\bigl(\widetilde\mu_0^N,\widetilde\mu_0\bigr) 
      +
        2W_2^2\bigl(N(m^N(0),\Sigma^N(0)),N(m_0,\Sigma_0)\bigr).
\end{aligned}
\]
By the standard empirical Wasserstein estimates for Gaussian samples
\cite{bobkov-ledoux-2019,ledoux-2017-gaussian-matching,ledoux-zhu-2021}, for some dimension-dependent constant $C_d$,
\[
        \mathbb E W_2^2\bigl(\widetilde\mu_0^N,\widetilde\mu_0\bigr)
        \le C_d r_d(N).
\]
The second term is controlled on $\mathcal G_N$ by the Gaussian
Wasserstein comparison, as in \eqref{eq:gaussian-w2-second-moment-comparison}, and standard Gaussian moment concentration:
\[
\begin{aligned}
&\mathbb E\left[
{\bf 1}_{\mathcal G_N}
W_2^2\bigl(N(m^N(0),\Sigma^N(0)),N(m_0,\Sigma_0)\bigr)
\right]  \\
&\hspace{2cm}
\le
K\,
\mathbb E\left[
\|m^N(0)-m_0\|^2+\|\Sigma^N(0)-\Sigma_0\|_F^2
\right]
\le
K N^{-1},
\end{aligned}
\]
for some constant $K = K(d,\Sigma_0)>0$.
Therefore
\begin{equation}\label{eq:initial-shape-bound}
       \mathbb E\left[
{\bf 1}_{\mathcal G_N} W_2^2\bigl(\widetilde\mu_0^N,N(m^N(0),\Sigma^N(0))\bigr)\right]
        \le C r_d(N).
\end{equation}
Combining \eqref{eq:shape-affine-bound} and
\eqref{eq:initial-shape-bound} gives
\[
        \mathbb E\left[
        {\bf 1}_{\mathcal G_N}
        \sup_{t\ge0}W_2^2(\widetilde\mu_t^N,\Gamma_t^N)
        \right]
        \le C r_d(N).
\]
Finally, by the same argument used in obtaining \eqref{eq:bad-event-poc} along with the observation \(W_2^2(\widetilde\mu_t^N,\Gamma_t^N) \le 4\operatorname{Tr}S^N(t),\) we obtain $ \mathbb E\left[
        {\bf 1}_{\mathcal G_N^c}
        \sup_{t\ge0} W_2^2(\widetilde\mu_t^N,\Gamma_t^N)\right]
        \le C r_d(N).$ Hence,
\begin{equation}\label{eq:shape-term-bound}
        \mathbb E\left[
        \sup_{t\ge0} W_2^2(\widetilde\mu_t^N,\Gamma_t^N)\right]
        \le C r_d(N).
\end{equation}
All the constants $C$ above depend only on $d,m_0,\Sigma_0$.

Putting \eqref{eq:projection-term-bound} and \eqref{eq:shape-term-bound} into
\eqref{eq:w2-triangle-gaussian-projection}, and using $N^{-1}\le C r_d(N)$,
yields \eqref{eq:liu-w2-recovered}. \eqref{eq:Gausspoc} follows from \eqref{eq:liu-w2-recovered} upon using \eqref{4.28}.
\end{proof}

\begin{remark}[Relation to the Gaussian--SVGD analysis of Liu et al. \cite{liu2023towards}]
The bilinear Gaussian base case considered in this section overlaps with the
Gaussian-target analysis of Liu--Ghosal--Balasubramanian--Pillai
\cite{liu2023towards}.  Their work studies SVGD with bilinear kernels through
the geometry of Gaussian variational inference: when the target is Gaussian and
the mean-field initializer is Gaussian, the mean-field SVGD solution remains on
the Gaussian manifold, and the dynamics reduces to a closed system for the mean
and covariance.  They also analyze the corresponding finite-particle system and
prove a uniform-in-time Wasserstein propagation-of-chaos estimate under
Gaussian initialization.

The purpose of the present section is slightly different.  We isolate the
finite-dimensional mechanism behind this phenomenon at the level of Stein
features.  For the bilinear kernel
\(
        \widetilde K_2(y,y')=(1+y^\top y')I_d,
\)
the Stein features are precisely the first and second Gaussian moment features,
and both the mean-field and empirical feature vectors solve the same closed
moment ODE.  Lemma~\ref{lem:affine-stability} proves a direct uniform stability
estimate for this ODE on bounded nondegenerate moment sets.  Consequently,
Proposition~\ref{prop:affine-gaussian-poc} gives a uniform-in-time
propagation-of-chaos estimate in the feature metric under checkable assumptions
on the latent initial law; in particular, the initial law need not be Gaussian. 

Corollary~\ref{cor:recover-liu-wasserstein-poc} shows how, in the special case of Gaussian initialization, this feature-level estimate upgrades to the full empirical Wasserstein estimate. The proof is more direct than the general Gaussian--SVGD argument in \cite{liu2023towards}: it decomposes \[ \widetilde\mu_t^N \longrightarrow \Gamma_t^N:=N(\widetilde m^N(t),\widetilde\Sigma^N(t)) \longrightarrow \widetilde\mu_t \] into two independent pieces. The second piece, \(W_2(\Gamma_t^N,\widetilde\mu_t)\), is controlled directly by Proposition~\ref{prop:affine-gaussian-poc} and the Gaussian Wasserstein formula, since it only depends on the first two moments. The first piece, \(W_2(\widetilde\mu_t^N,\Gamma_t^N)\), is the empirical ``shape'' fluctuation not seen by the Stein features; under Gaussian initialization, the affine representation of the particle system reduces this term to the standard empirical Wasserstein error for Gaussian samples. This yields the stronger estimate \( \mathbb E\left[ \sup_{t\ge0} W_2^2(\widetilde\mu_t^N,\widetilde\mu_t) \right] \le C r_d(N), \) with the same dimension-dependent rate \(r_d(N)\) as in \cite{liu2023towards}. We also note that both Proposition~\ref{prop:affine-gaussian-poc} and Corollary~\ref{cor:recover-liu-wasserstein-poc} obtain uniform-in-time
propagation-of-chaos in a strong sense where the $\sup_{t \ge 0}$ appears within the expectation, in contrast with \cite{liu2023towards}, where this appears outside the expectation. 

The advantage of the present formulation is that the first part of the argument does not rely on Gaussian initialization: Proposition~\ref{prop:affine-gaussian-poc} is a feature-level, non-Gaussian initialization statement. Moreover, by the conjugacy result of Section~\ref{sec:conj}, the same finite-dimensional stability mechanism transfers without loss to every Gaussian-conjugate target--kernel pair in its intrinsic Stein-feature metric.

We also note here that the restriction to the standard Gaussian target is only for normalization. The same Proposition~\ref{prop:affine-gaussian-poc} and
Corollary~\ref{cor:recover-liu-wasserstein-poc} hold for any nondegenerate
Gaussian target, with the constant and suitably modified bilinear
matrix kernels.
\end{remark}

\subsection{Gaussian-conjugates and the two pulled-back kernels}\label{sec:pullback}
Let $A:\R^d\to\R^d$ be a $C^3$ orientation-preserving diffeomorphism.  Set
\[
        H:=A^{-1},
        \qquad
        M(x):=DA(x)^{-1}=DH(A(x)).
\]
The Gaussian-conjugate target is
\begin{equation}\label{eq:gaussian-conjugate-target}
        \pi_A(dx)=(2\pi)^{-d/2} e^{-V_A(x)}\,dx,
        \qquad
        V_A(x)=-\log\det DA(x)+\frac12\norm{A(x)}^2.
\end{equation}
Equivalently, $A_{\#}\pi_A=\gamma_d$.

\textbf{Flexibility of the Gaussian-conjugate class.}
The terminology \emph{Gaussian-conjugate} refers only to the existence of an invertible latent Gaussian representation: if \(Z\sim\gamma_d\) and \(X=H(Z)=A^{-1}(Z)\), then \(X\sim\pi_A\).
It does not imply that \(\pi_A\) is approximately Gaussian in the original \(x\)-coordinates.  Indeed, the Jacobian factor \(\det DA(x)\) can produce strong skewness, nonlinear dependence, heavy or light tails, and multiple modes.  For example, in one dimension, writing $\Phi$ for the normal cdf, any strictly positive smooth density \(q\) with cdf $F_q$ for which \(A=\Phi^{-1}\circ F_q\) is a \(C^3\) diffeomorphism is Gaussian-conjugate; this includes Student \(t\) laws and many smooth multimodal densities.  For $d \ge 1$, writing \(a_\alpha(u)=u-\alpha\sin u\), \(0<\alpha<1\), the coordinatewise map \(A_\alpha(x)=(a_\alpha(x_1),\ldots,a_\alpha(x_d))\) gives
\[
    \pi_{A_\alpha}(x)\propto
    \prod_{j=1}^d(1-\alpha\cos x_j)
    \exp\!\left[-\frac12\sum_{j=1}^d
    (x_j-\alpha\sin x_j)^2\right],
\]
which is genuinely multimodal for, say, \(\alpha=0.9\), with at least \(2^d\) local modes.  Likewise, the choice \(A(x)=(\sinh x_1,\ldots,\sinh x_d)\) produces a strongly non-Gaussian, light-tailed target.  Thus the class is comparable to a smooth invertible normalizing-flow family and is highly flexible, although a global diffeomorphism necessarily restricts it to smooth, everywhere-positive densities on \(\mathbb R^d\), excluding singular or compactly supported targets.

Furthermore, the Gaussian-conjugate construction contains the Kim--Milman transport~\cite{kim2012generalization} as a distinguished contractive subclass.  Indeed, for a target of the form
\[
    \pi(dx)\propto
    \exp\!\left\{-\frac12\|x\|^2-W(x)\right\}dx,
    \qquad W\ \text{convex},
\]
Kim and Milman construct, via reverse heat flow, a contraction
\(H_{\mathrm{KM}}\) satisfying
\[
    (H_{\mathrm{KM}})_{\#}\gamma_d=\pi,
    \qquad
    \|H_{\mathrm{KM}}(z)-H_{\mathrm{KM}}(z')\|
    \le \|z-z'\|.
\]
Whenever \(H_{\mathrm{KM}}\) is sufficiently regular and invertible, our notation corresponds to
\[
    H=H_{\mathrm{KM}},
    \qquad
    A=H_{\mathrm{KM}}^{-1},
\]
and hence \(\pi=\pi_A\).  Moreover, the contraction property yields
\(
    \|M(x)\|_{\mathrm{op}}
    =
    \|DH_{\mathrm{KM}}(A(x))\|_{\mathrm{op}}
    \le 1.
\)
Thus, the Kim--Milman map provides a canonical Gaussian conjugacy with useful Jacobian control for targets that are more log-concave than the Gaussian.  The general Gaussian-conjugate class considered here is substantially larger, since \(A\) may be any sufficiently regular orientation-preserving diffeomorphism; in particular, Gaussian-conjugate targets need not be log-concave and may be skewed or multimodal.

\subsubsection{The pulled-back constant kernel}

Pulling back \eqref{eq:constant-kernel} gives
\begin{equation}\label{eq:conjugate-constant-kernel}
        K_A^{(1)}(x,x')=M(x)M(x')^\top.
\end{equation}
By Lemma~\ref{lem:feature-conjugacy}, the Stein features are
\begin{equation}\label{eq:conjugate-constant-features}
        F_A^{(1)}(x)=-A(x).
\end{equation}
Thus the natural feature metric is
\begin{equation}\label{eq:conjugate-d1}
        d_{A,1}(\rho,\nu):=\norm{\rho A-\nu A}.
\end{equation}

\begin{corollary}[Uniform POC for the first Gaussian-conjugate kernel]\label{cor:conjugate-constant-poc}
Let $X_i^N(0)$ be i.i.d. with law $\mu_0$, and set $Y_i^N(0):=A(X_i^N(0))$.  If
\[
        \E\norm{A(X(0))}^2<\infty,
\]
then the matrix-SVGD flow for $(\pi_A,K_A^{(1)})$ satisfies
\begin{equation}\label{eq:conjugate-constant-poc}
        \E\left[\sup_{t\ge0}\, d_{A,1}^2(\mu_t^N,\mu_t)\right]
        \le
        \frac{\Tr\Cov(A(X(0)))}{N}.
\end{equation}
The same estimate holds with $\mu_t^N,\mu_t$ replaced by their time averages.
\end{corollary}

\begin{proof}
Apply Proposition~\ref{prop:constant-gaussian-poc} in the latent coordinate $Y=A(X)$ and use Lemmas~\ref{lem:feature-conjugacy}--\ref{lem:velocity-conjugacy}.
\end{proof}

\subsubsection{The pulled-back bilinear kernel}

Pulling back \eqref{eq:affine-kernel} gives
\begin{equation}\label{eq:conjugate-affine-kernel}
        K_A^{(2)}(x,x')=\bigl(1+A(x)^\top A(x')\bigr)M(x)M(x')^\top .
\end{equation}
The Stein features are
\begin{equation}\label{eq:conjugate-affine-features}
        -A_i(x),
        \qquad
        \delta_{ij}-A_i(x)A_j(x),
        \qquad 1\le i,j\le d.
\end{equation}
Thus this kernel controls the first and second moments of the latent coordinate $A(X)$.  Define
\begin{equation}\label{eq:conjugate-d2}
        d_{A,2}(\rho,\nu)
        :=
        \left(
        \norm{\rho A-\nu A}^2
        +
        \norm{\rho(AA^\top)-\nu(AA^\top)}_F^2
        \right)^{1/2}.
\end{equation}

\begin{corollary}[Uniform POC for the second Gaussian-conjugate kernel]
\label{cor:conjugate-affine-poc}
Let $X_i^N(0)$ be i.i.d. with law $\mu_0$, and set
$
        Y_i^N(0):=A(X_i^N(0)).
$
Assume that
\begin{equation}\label{eq:conjugate-affine-initial-conditions}
        \E \exp\{\alpha\norm{A(X(0))}^2\}<\infty,
        \qquad
        \Cov(A(X(0)))\ge c_0\Id_d
\end{equation}
for some $\alpha>0,c_0>0$.  Then the matrix-SVGD flow for $(\pi_A,K_A^{(2)})$
satisfies
\begin{equation}\label{eq:conjugate-affine-poc}
        \E\left[\sup_{t\ge0}\, d_{A,2}^2(\mu_t^N,\mu_t)\right]
        \le
        C N^{-1},
\end{equation}
where
$
        C
        =
        C\Bigl(d,c_0,\alpha,
        \E e^{\alpha\norm{A(X(0))}^2}\Bigr)<\infty .
$
\end{corollary}

\begin{proof}
Apply Proposition~\ref{prop:affine-gaussian-poc} in the latent coordinate $Y=A(X)$ and use Lemmas~\ref{lem:feature-conjugacy}--\ref{lem:velocity-conjugacy}.
\end{proof}

\begin{remark}[A broader finite-dimensional closure program]

The Gaussian-conjugate examples above should be viewed as a proof of principle for a more general mechanism, rather than as the only possible setting in which polynomial uniform-in-time propagation-of-chaos rates can be obtained. Suppose that a target--kernel pair admits a finite collection of Stein features
\[
F=(f_1,\ldots,f_m)
\]
such that the SVGD dynamics closes, or approximately closes, at the level of the feature vector \(\rho F\). In the ideal case, this means that
\[
\frac{d}{dt}\rho_t F = G(\rho_tF)
\]
for a finite-dimensional vector field \(G\), and that the flow $(\Phi_t)_{t \ge 0}$ of the ODE $\dot z = G(z)$ satisfies a polynomial stability estimate
\[
|\Phi_t(z)-\Phi_t(z')|
\le C(1+t)^p|z-z'| ,
\]
for $z,z'$ in an `admissible' region where the trajectories live.
More generally, the same strategy should apply when the exact feature dynamics are not closed, but the discrepancy between the relevant driving vector fields is controlled by a polynomial function of time.

Under such a closure/stability mechanism, the empirical \(N^{-1/2}\) fluctuation of the initial feature vector propagates only polynomially on finite time intervals, giving estimates of the form
\begin{equation}\label{fintime}
\sup_{0\le t\le R}
\mathbb{E} \,d_F(\mu_t^N,\mu_t)
\le C N^{-1/2}(1+R)^p
\end{equation}
up to the corresponding vector-field approximation error. Combining this finite-time estimate with the long-time target estimates for the time-averaged measures used in the preceding sections of the form
\[
\mathbb{E} \,d_F(\bar\mu_T^N,\pi)
\le C(T^{-b}+N^{-q}),
\qquad
d_F(\bar\mu_T,\pi)\le C T^{-b},
\]
for suitable $b,q>0$, and optimizing the cutoff \(R=R_N\), one obtains a polynomial uniform-in-averaging-horizon bound of the form
\[
\sup_{T>0}
\mathbb{E} \,d_F(\bar\mu_T^N,\bar\mu_T)
\le
C N^{-\min\{q, b/(2p+2b)\}} .
\]
In the uniformly stable case \(p=0\), the long-time analysis is not needed and one simply applies \eqref{fintime} to obtain the sharp feature-level \(N^{-1/2}\) rate.

Thus the essential structural inputs for polynomial uniform-in-time propagation-of-chaos are finite-dimensional closure, polynomial stability or polynomially controlled vector-field discrepancy, and long-time convergence of both the finite-particle and mean-field time averages to the common target. In the present work we carry this program out explicitly for the Gaussian-conjugate class, where the Stein features reduce to latent first and second moments and the closed moment dynamics can be analyzed directly. Finding genuinely new non-Gaussian target--kernel pairs satisfying these closure and stability conditions remains an interesting open direction.
\end{remark}

\begin{remark}[Relation to weak propagation-of-chaos methods in shallow neural networks]\label{GBrem}

In Section~\ref{sec:KSD} and Section~\ref{sec:conjgauss}, we obtain uniform-in-time propagation-of-chaos estimates in a `weak sense' by controlling the discrepancy between the empirical and mean-field flows for a class of observables without controlling the full empirical measure. This is similar in spirit to the recent work of Glasgow and Bruna \cite{glasgow2026uniform} in a shallow neural-network setting, where the authors establish a uniform-in-time weak propagation-of-chaos result by focusing on a specific observable of the particle system, namely the network output. Their approach exploits the decay of the mean-field loss to prevent the fluctuation error from accumulating over long time horizons. While our work also seeks long-time control of particle approximations, the underlying mechanism is quite different.

For SVGD, the analogous weak observable is the Stein witness
\begin{equation*}
S_\pi(\rho)
=
\int \tau(x)\rho(dx),
\qquad
\tau(x)
=
-k(\cdot,x)\nabla V(x)
+
\nabla_2 k(\cdot,x).
\end{equation*}
Indeed,
$
\KSD_{\pi}(\rho,\eta)
=
\|S_\pi(\rho)-S_\pi(\eta)\|_{\mathcal H^d}.
$
A formal analogue of the Glasgow--Bruna argument would differentiate
\begin{equation*}
E_t
=
\|S_\pi(\mu_t^N)-S_\pi(\mu_t)\|_{\mathcal H^d}
\end{equation*}
along the finite-particle and mean-field SVGD flows. If the resulting differentiated Stein observables can again be controlled by the same KSD distance, one expects an estimate of the form
\begin{equation*}
\frac{d}{dt}E_t
\le
C\,\KSD(\mu_t \|\pi)E_t
+
C\, E_t^2
\end{equation*}
for some finite constant $C$.
Thus the decay of the mean-field KSD would play the same role as the decay of the mean-field loss in \cite{glasgow2026uniform}: it acts as a stabilizing coefficient in the fluctuation estimate.

Carrying out this Glasgow--Bruna-type program for SVGD seems to require two additional ingredients that are not presently available in general. First, polynomial fluctuation estimates require the integrability condition
\begin{equation*}
\int_0^\infty \KSD(\mu_t \|\pi)\,dt < \infty ,
\end{equation*}
whereas the standard SVGD entropy dissipation only gives integrability of \(\KSD(\mu_t \|\pi)^2\). Second, differentiating the Stein witness produces observables such as
$
\nabla(T_\pi a)(x)\cdot b(x),
$
which would need to remain controlled by the same KSD dual class. This is a strong Stein-closure property and is not automatic for general kernels and targets. Thus the Glasgow--Bruna philosophy suggests a possible weak-observable route for SVGD, but our concrete polynomial-rate results rely on the explicitly checkable Gaussian-conjugacy closure.

\end{remark}

\begin{remark}[On possible infinite-rank extensions]\label{rem:infrank}
The finite-rank assumption is mainly used in this section to obtain a finite list of Stein features and, in the Gaussian bilinear case, a closed stable moment system. By contrast, the change-of-variables identities \eqref{eq:potential-transform} and \eqref{eq:kernel-transform} behind the conjugacy principle are not specific to finite rank.
Thus an infinite-rank extension of our results should be natural at the level of such covariance identities, for instance through a Hilbert-space factorization of the matrix-valued kernel. However, to obtain polynomial rates for uniform-in-time propagation-of-chaos at this level of generality, one would need a stable closed evolution for the resulting infinite-dimensional feature map, or a finite-rank approximation scheme with truncation errors controlled uniformly in time. We do not pursue this direction here.
\end{remark}

\subsection*{Acknowledgements}

 SB is partially supported by the NSF-CAREER award DMS-2141621 and the NSF RTG grant DMS-2134107.

KB is supported in part by National Science Foundation (NSF) grant DMS-2413426.

AK was funded by the European
Union (ERC, Optinfinite, 101201229). The views and opin-
ions expressed are, however, of the author(s) only and do
not necessarily reflect those of the European Union or the
European Research Council Executive Agency. Neither
the European Union nor the granting authority can be held
responsible for them.

The authors gratefully acknowledge the hospitality and support of the Banff International Research Station (BIRS) during the 2026 workshop on \emph{Stein's Method Meets Statistical Learning}, where this work was initiated.

\bibliographystyle{plain}
\bibliography{ref}

\end{document}